\documentclass[10pt]{article}

\usepackage{amsmath}
\usepackage{amssymb}
\usepackage{amsthm}
\usepackage{bm}
\usepackage{multirow}

\providecommand{\Div}{\operatorname{div}}          

\providecommand{\Det}{\operatorname{det}}                    

\newcommand{\VC}{{\mathbf{C}}}

\newcommand{\VE}{{\mathbf{E}}}
\newcommand{\VF}{{\mathbf{F}}}

\newcommand{\VN}{{\mathbf{N}}}

\newcommand{\VR}{{\mathbf{R}}}

\newcommand{\Dsf}{\mathsf{D}}

\providecommand{\Cl}{{\cal L}}

\providecommand{\Cv}{{\cal V}}
\providecommand{\Cw}{{\cal W}}

\usepackage[a4paper, total={6in, 8in}]{geometry}

\newcommand{\eps}{\varepsilon}

\newcommand{\Ui}{U^{(1)}}
\newcommand{\Ri}{R^{(1)}}
\newcommand{\Uii}{U^{(2)}}
\newcommand{\Uiii}{U^{(3)}}
\newcommand{\Rii}{R^{(2)}}
\newcommand{\vi}{u^{(1)}}
\newcommand{\vii}{u^{(2)}}

\newcommand{\PPi}{P^{(1)}}
\newcommand{\Si}{S^{(1)}}
\newcommand{\Pii}{P^{(2)}}
\newcommand{\Piii}{P^{(3)}}
\newcommand{\Sii}{S^{(2)}}
\newcommand{\wi}{p^{(1)}}
\newcommand{\wii}{p^{(2)}}

\newcommand{\Qi}{Q^{(1)}}
\newcommand{\Ti}{T^{(1)}}
\newcommand{\Qii}{Q^{(2)}}

\newcommand{\Tii}{T^{(2)}}
\newcommand{\mi}{q^{(1)}}
\newcommand{\mii}{q^{(2)}}

\newcommand{\epsb}{\bm{\epsilon}}

\newcommand{\mRi}{\mathcal{R}^{(1)}}
\newcommand{\pimL}{\partial_{\ell}^{(1)}\mathcal{L}}
\newcommand{\mRii}{\mathcal{R}^{(2)}}
\newcommand{\piimL}{\partial_\ell^{(2)}\mathcal{L}}
\newcommand{\li}{\ell_1}
\newcommand{\lii}{\ell_2}

\usepackage{hyperref}

\usepackage{color}

\numberwithin{equation}{section}

  \theoremstyle{definition}
  \newtheorem{theorem}{Theorem}[section]
  
  \newtheorem{proposition}[theorem]{Proposition}
  \newtheorem{lemma}[theorem]{Lemma}
  \newtheorem{definition}[theorem]{Definition}
  \newtheorem{remark}[theorem]{Remark}

  \newtheorem{assumption}{Assumption}

  \newtheorem*{assumption*}{Assumption}

\def\Xint#1{\mathchoice
{\XXint\displaystyle\textstyle{#1}}%
{\XXint\textstyle\scriptstyle{#1}}%
{\XXint\scriptstyle\scriptscriptstyle{#1}}%
{\XXint\scriptscriptstyle\scriptscriptstyle{#1}}%
\!\int}
\def\XXint#1#2#3{{\setbox0=\hbox{$#1{#2#3}{\int}$ }
\vcenter{\hbox{$#2#3$ }}\kern-.6\wd0}}

\def\dashint{\Xint-}

\begin{document}

\title{Adjoint based methods to compute higher order topological derivatives with an application to elasticity}
\author{Phillip Baumann, Kevin Sturm}

\maketitle

\tableofcontents
\newpage
\begin{abstract}
\noindent\textbf{Purpose - }The goal of this paper is to give a comprehensive and short review on how 
to compute the first and second order topological derivative and potentially higher order 
topological derivatives for PDE constrained 
shape functionals.\newline
\noindent\textbf{Design/methodology/approach - }We employ the adjoint and averaged adjoint variable within the Lagrangian framework and compare three different adjoint based methods to compute 
higher order topological derivatives. To illustrate the methodology proposed in this paper, we then apply the methods to a linear 
elasticity model.\newline
\noindent\textbf{Findings - }We compute the first and second order topological derivative of the linear elasticity model for various shape functionals in dimension two and three using Amstutz' method, the averaged adjoint method and Delfour's method.\newline
\noindent\textbf{Originality/value - }In contrast to other contributions regarding this subject, we not only compute the first and second order topological derivative, but additionally give some insight on various methods and compare their applicability and efficiency with respect to the underlying problem formulation.\newline
\noindent\textbf{Keywords - }Topological derivative, topology optimisation, elasticity.\newline
\noindent\textbf{Paper type - }Research paper.
\end{abstract}

\section{Introduction}

In this paper we provide a review of techniques for the computation of the first and second topological derivative. 
We compare and apply three techniques to the following model problem:
Let $\Dsf\subset \VR^d$, $d=2,3$, be a bounded and smooth domain. The goal is it to compute the topological derivative of the cost functional
\begin{equation}\label{eq:cost_func}
    \mathcal J(\Omega) = \gamma_f\int_\Dsf f_\Omega\cdot u_\Omega\;dx+\gamma_g\int_\Dsf|\nabla u_\Omega-\nabla u_d|^2\;dx +\gamma_m\int_{\Gamma^m}|u_\Omega-u_m|^2\;dS,    
\end{equation}
$\gamma_f, \gamma_g, \gamma_m\in \VR$, $\gamma_g=\gamma_m=0$ in $d=2$, $u_d\in H^1(\Dsf)$, $u_m\in L_2(\Gamma^m)$ subject to a design region $\Omega \subset \Dsf$ and the displacement field $u_\Omega\in H^1(\Dsf)^d$ solves the equation of linear elasticity
\begin{equation}\label{eq:elasticity_linear}
    \int_\Dsf \VC_\Omega \epsb(u_\Omega):\epsb(\varphi)\;dx = \int_\Dsf f_\Omega \cdot\varphi\;dx+\int_{\Gamma^N}u_N\cdot \varphi\;dS \quad \text{ for all } \varphi\in H^1_{\Gamma}(\Dsf)^d,
\end{equation}
$u_\Omega|_\Gamma=u_D$,
where $\Gamma\subset \partial \Dsf$, $\Gamma_N:=\partial \Dsf\setminus \Gamma$, $H^1_{\Gamma}(\Dsf)^d:= \{\varphi\in H^1(\Dsf):\; \varphi=0 \text{ on } \Gamma\}$, $u_D\in L_2(\Gamma)$, $u_N\in L_2(\Gamma^N)$,
\begin{equation}
        \VC_\Omega  = \VC_1 \chi_\Omega + \VC_2 \chi_{\Dsf\setminus\overline{\Omega}}, \quad f_\Omega  = f_1 \chi_\Omega + f_2 \chi_{\Dsf\setminus\overline{\Omega}}
\end{equation}
and $\VC_1,\VC_2:\VR^{d\times d} \to \VR^{d\times d}$ linear functions, $f_1,f_2\in H^1(\Dsf)^d \cap C^{2}(B_\delta(x_0))^d$, $\delta>0$ and $\epsb(u)$ denotes the symmetrised 
gradient of $u$, that is, $\epsb(u)= \frac12(\nabla u + \nabla u^\top)$.\newline
We are going to discuss the topological derivative of a singularly perturbed domain by adding a small inclusion $\omega_\eps\in \Dsf\setminus\overline{\Omega}$ to $\Omega$, i.e. $\Omega_\eps=\Omega\cup\omega_\eps$. In fact, we are going to consider the case $\Omega=\emptyset$, but the remaining cases $\Omega\neq\emptyset$ and $\Omega_\eps=\Omega\setminus\omega_\eps$ for $\omega_\eps\subset\Omega$ can be treated in a similar fashion.


The topological derivative was first introduced in \cite{a_ESKOSC_1994a} and later mathematically justified in \cite{a_SOZO_1999a,a_GAGUMA_2001a} with an application to linear elasticity. Follow up works of many authors studied the asymptotic behavior of shape functionals for various partial differential equations. For instance for Kirchhoff plates \cite{a_AMNO_2010a}, electrical impedance tomography \cite{a_HILA_2008a,a_HILANO_2011a}, Maxwell's equation \cite{a_MAPOSA_2005a}, Stokes' equation \cite{a_HAMA_2004a} and elliptic variational inequalities \cite{a_HILA_2011a}. We also refer to the monograph \cite{b_NOSO_2013a} for more applications and references therein.


The idea of the topological derivative is to perturb the design variable  with a singular perturbation and study the asymptotic behavior of the shape functional $\mathcal J$. The asymptotic expansion encodes information about the optimal topology of the design region and 
can be used numerically either in an iterative level-set method \cite{a_AMAN_2006a} or one-shot type methods \cite{a_HILA_2008a,a_SOZO_1999a} to obtain an optimal topology of the design region (in the sense of stationary points). 
Higher order topological derivatives are a viable means to improve the accuracy of one-shot type methods as done in \cite{a_HILA_2008a,a_BOCO_2017a}.  


Higher order topological derivatives are less studied, but have been computed for several problems. For instance 
in \cite{a_HILA_2008a} second order topological derivatives for an  electrical impedance tomography problem is studied. 
In \cite{a_BO_2018a} higher order topological derivatives in dimension two for linear elasticity using the method of \cite{a_NOFETAPA_2003aa}
is established. In \cite{a_BOCO_2017a} the expansion of higher order topological derivatives for a least square misfit function 
for linear elasticity in dimension three exploiting a Green's function is established. In \cite{a_BO_2018a} a similar misfit function subject to a scattering problem is expanded. 


The first ingredient to compute higher topological derivatives is the asymptotic behavior of the 
solution of the state equation, in our concrete example this is equation \eqref{eq:elasticity_linear}. The second ingredient is an expansion of the shape function and is mostly, although not necessary, done via the introduction of an adjoint variable. As is well-known from optimal control and shape optimisation theory (see, e.g., \cite{b_HIULUL_2009a,b_ITKU_2008a}) the advantage of using an adjoint variable is the numerically efficient computation of the topological derivative. First order topological 
derivatives for ball inclusions and linear problems can be computed solely from the knowledge of the state variable and the adjoint state variable; see, e.g., in \cite{a_SOZO_1999a}. For higher order topological derivatives in most cases additional exterior partial differential equations, so-called corrector equations, have to be solved, although in some cases these can also be solved explicitly; \cite{a_HILA_2008a}.


While most papers deal with linear partial differential equations also nonlinear partial differential equations have been studied. 
We refer to \cite{a_IGNAROSOSZ_2009a,a_BEMARA_2017a,a_ST_2020a,a_AM_2006b} for the study of first order topological derivatives for semilinear elliptic partial differential equations. To the authors knowledge there is no research for higher order topological derivatives for these equations and thus remains an open and challenging topic. Also quasi-linear problems have been studied first in \cite{a_AMBO_2017a} and more recently in \cite{a_GAST_2020a,a_AMGA_2019a,a_GAST_2021a}. In particular in \cite{a_GAST_2020a} a projection trick is used to avoid the use of a fundamental solution, which is in contrast to most works on semilinear partial differential equations.


An established method to compute the topological derivative and higher derivatives is the method of Amstutz \cite{phd_AM_2003a}. It amounts to study the asymptotic behavior of a perturbed adjoint equation, which depends on the unperturbed state equations. It has been used in some of the papers mentioned above such as \cite{a_MAPOSA_2005a,a_HAMA_2004a} and also \cite{a_AM_2006a,a_AM_2006b}, to only mention a few. The advantage of the method is that it  simplifies the computation of the topological derivative compared to a direct computation of the topological derivative by expanding the cost function with Taylor's expansion.


A second method which has been introduced in the context of shape optimisation and the computation of shape derivatives was used 
in \cite{a_ST_2020a} to compute topological derivatives for semilinear problems. It has been extended in \cite{a_GAST_2020a} 
to compute topological derivatives for quasilinear problems. In contrast to Amstutz' method the averaged adjoint variable 
also depends on the perturbed state equation, which makes the anaylsis of the asymptotic behavior of the adjoint variable 
more challenging. However, the advantage is that it seems to be readily applicable to a wide range of cost functions and also 
the computation of the final formula for higher order topological derivatives is straight forward once the asymptotics of the 
averaged adjoint variable is known. 


A third method was introduced in \cite{c_DE_2018b} and uses the usual unperturbed adjoint variable. The advantage is that no 
analysis of an perturbed adjoint variable is required, but as shown in \cite{a_GAST_2020a} it seems to be more difficult to apply 
this method to certain cost functions, such as the $L_2$-tracking type cost functions. 


Finally let us mention the method of \cite{a_NOFETAPA_2003aa} where a method to compute the topological derivatives is proposed as the limit of the shape derivative. This method is not always applicable, but it provides a fast method to compute also higher order topological derivatives; see \cite{a_SIMATO_2010a}.


In this paper we thoroughly study and review the first three mentioned methods and apply them to the model problem of linear elasticity introduced in \eqref{eq:elasticity_linear}. We first exam the asymptotic behavior of \eqref{eq:elasticity_linear} up to order two and then study the asymptotic behavior of Amstutz' perturbed adjoint variable and the averaged adjoint variable. We then apply the three methods to compute first and second order topological derivatives for three types of cost functions, the compliance, a boundary tracking-type cost function and a tracking-type cost function of the gradient.

\paragraph{Structure of the paper}
In Section \ref{sec:abstract_setting} we discuss three different teqhniques to compute the topological derivative. This is done by introducing the Lagrangian setting, which simplifies the notation.\newline
In Section~\ref{sec:sens_an_elasticity} we derive the complete asymptotic analysis for a linear elasticity model including 
remainder estimates. The section covers both the two dimensional and three dimensional case, whose analysis differs since 
the fundamental solution of the linear elasticity equation has a different asymptotic behavior.\newline
In Section \ref{sec:sens_adjoint} we derive the asymptotic analysis for the adjoint and averaged adjoint variable respectively. This is done in a similar fashion to Section ~\ref{sec:sens_an_elasticity}.\newline
In Section \ref{sec:topo_deriv} we employ the previously derived results to compute the topological derivative. That is, we apply the theoretical background derived in Section \ref{sec:abstract_setting} to our elasticity model and a versatile cost function.

\paragraph{Notation}
In the whole paper we denote by $W^1_p(\Dsf)$ (resp. their vector-valued counter parts by $W^1_p(\Dsf)^d$) for $1\le p \le \infty$ standard Sobolev spaces equipped with the usual norm. The gradient of a function $\varphi\in W^1_p(\Dsf)$ (resp. $\varphi\in W^1_p(\Dsf)^d$) will be denoted $\nabla \varphi$. Directional derivatives of functions $f:U\subset \VE\to \VR$ at $x\in \VE$ defined on an open subset $U\subset \VE$ of  Banach space $\VE$ will be denoted by $\partial f(x)(v)$, $x\in U$, $v\in \VE$ whenever it exists. Similarly for functions $(u,v)\mapsto f(u,v):\VE\times \VF\to \VR$ we denote their partial derivative with respect to the first (resp. second) argument by $\partial_u f(x_1,x_2)(v)$ (resp. $\partial_v f(x_1,x_2)(w)$).  We further define for $1 < p<\infty$
\[
    BL_p(\VR^d)^d := \{ \varphi \in W^1_{p,loc}(\Dsf)^d:\; \nabla \varphi \in L_p(\Dsf)^{d\times d}\}.
\]
Then we define the Beppo-Levi space $\dot{BL}_p(\VR^d)^d:= BL(\VR^d)^d/\VR$ equipped with $\|[\varphi] \|_{\dot{BL}_p(\VR^d)^d} := \|\nabla \varphi\|_{L_p(\VR^d)^{d\times d}}$, $\varphi \in [\varphi]$, $[\varphi]\in \dot{BL}_p(\VR^d)^d$. Here $/\VR$ means that we quotient out constants. 

The Euclidean norm on $\VR^d$ will be denoted $ |\cdot| $ and the corresponding operator norm $\VR^{d\times d}$ will be also denoted by $|\cdot|$. The Euclidean ball of radius $r>0$ located at $x_0\in \VR^d$ will be denoted by $B_r(x_0)$. Additionally, for a domain domain $\Omega$ with sufficiently smooth boundary $\partial \Omega$, we denote the outer normal vector by $n$.
The Slobodeckij seminorm $|\cdot|_{H^{\frac{1}{2}}(\Omega)}$ for $\Omega \subset \VR^d$ is defined by
\[|u|_{H^{\frac{1}{2}}(\Omega)}:=\left(\int_\Omega\int_\Omega \frac{|u(x)-u(y)|^2}{|x-y|^{d+1}}\; dxdy\right)^{\frac{1}{2}}.\]
For convenience we will later on use the abbreviated notation of the averaged integral defined as
\[\dashint_\Omega f\; dx:=\frac{1}{|\Omega|}\int_\Omega f\; dx,\]
for a bounded set $\Omega\subset \VR^d$.

\section{Lagrangian techniques to compute the topological derivative}\label{sec:abstract_setting}
In this section, we review Lagrangian techniques to compute topological derivatives. 
While it is well-established in optimisation algorithms to compute derivatives of 
PDE constrained problems with the help of Lagrangians it seems rather new to the 
topology optimisation community. However, we will show that actually Amustutz's method can be interpreted as a Lagrangian approach by introducing a suitable Lagrangian function and recasting his original result in terms of this Lagrangian. More recently another Lagrangian approach was proposed in \cite{c_DEST_2016a} where essentially an extra term appears when differentiating the Lagrangian function. Finally, we will review Delfour's approach of \cite[Thm.3.3]{c_DE_2018b} using only the unperturbed adjoint state variable.

\subsection{Abstract setting}
Let $\Cv,\Cw$ be real Hilbert spaces. For all parameter $\eps\ge 0$ small consider 
a function $u_\eps\in \Cv$ solving the variational problem of the form
\begin{equation}\label{eq:abstract_state}
    a_\eps(u_\eps,\varphi) = f_\eps(\varphi) \quad \text{ for all } \varphi\in \Cw,
\end{equation}
where $a_\eps$ is a bilinear form on $\Cv\times \Cw$ and $f_\eps$ is a linear form on $\Cw$, respectively. Throughout we assume that this abstract state equation admits a unique solution and 
that $u_\eps-u_0\in \Cw$ for all $\eps$.  Consider now a cost function 
\begin{equation}
    j(\eps) = J_\eps(u_\eps) \in \VR,
\end{equation}
where for all $\eps\ge0$ the functional $J_\eps:\Cv\to \VR$ is differentiable at $u_0$. In the following sections we review methods how to obtain an asymptotic expansion of $j(\eps)$ at $\eps =0$. For this purpose we introduce the Lagrangian function 
\[ 
    \Cl(\eps,u,v)= J_\eps(u) + a_\eps(u,v) - f_\eps(v), \quad u \in  \Cv, \; v\in \Cw.
\]

\subsection{Amstutz' method}
We first review the approach of Amstutz \cite{phd_AM_2003a}; see also \cite[Prop. 2.1]{a_AM_2006a}. This approach has been proved to be versatile and has 
been applied to a number of linear and nonlinear problems. For instance in \cite{a_AM_2006a} a linear transmission problem 
was examined and its first order topological derivative was computed. In \cite{a_AMNOANVA_2014a} the topological 
derivative of elliptic differentiation equations with $2m$ differential operator was derived. In \cite{a_AM_2006b} 
the topological derivative for a class of certain nonlinear equations has been studied.

\begin{proposition}[{\cite[Prop. 2.1]{a_AM_2006a}}]\label{thm:Amstutz_main}
 Assume the following hypotheses hold.
\begin{itemize}
    \item[(1)] There exist numbers $\delta a^{(1)}$ and $\delta f^{(1)}$ and a function $\ell_1:\VR^+\rightarrow\VR^+$ with $\displaystyle\lim_{\eps\searrow 0}\ell_1(\eps)=0$, such that
        \begin{align}
            (a_\eps-a_0)(u_0,p_\eps) & = \ell_1(\eps)\delta a^{(1)} + o(\ell_1(\eps)),  \\
            (f_\eps-f_0)(p_\eps) & = \ell_1(\eps)\delta f^{(1)} + o(\ell_1(\eps)),
        \end{align}
        where $p_\eps\in \Cw$ is the adjoint state satisfying
        \begin{equation}\label{eq:adjoint_am}
            a_\eps(\varphi,p_\eps) = - \partial J_\eps(u_0)(\varphi) \quad \text{ for all } \varphi\in \Cv.
        \end{equation}
\item[(2)] There exist two numbers $\delta J_1^{(1)}$ and $\delta J_2^{(1)}$, such that
    \begin{align}
        J_\eps(u_\eps) & = J_\eps(u_0) + \partial J_\eps(u_0)(u_\eps - u_0) + \ell_1(\eps)\delta J_1^{(1)} + o(\ell_1(\eps)), \\
        J_\eps(u_0) & = J_0(u_0) +  \ell_1(\eps) \delta J_2^{(1)} + o(\ell_1(\eps)).
    \end{align}
\end{itemize}
Then the following expansion holds
\begin{equation}
    j(\eps) = j(0) + \ell_1(\eps)(\delta a - \delta f^{(1)} + \delta J_1^{(1)} + \delta J_2^{(1)}) + o(\ell_1(\eps)).
\end{equation}
\end{proposition}

We will reformulate and generalise the previous result in terms of a Lagrangian function $\Cl(\eps,u,v)$ and additionally state a result for the second order derivative.

\begin{proposition}\label{prop:main_am}
\begin{itemize}
\item[(i)] Let $\li:\VR^+\to\VR^+$ be a function with $\displaystyle\lim_{\eps\searrow 0}\li(\eps)=0$. Furthermore, assume that the limits
 \begin{align}
     \mRi(u_0,p_0) & := \lim_{\eps\searrow0} \frac{\Cl(\eps,u_\eps,p_\eps) - \Cl(\eps,u_0,p_\eps)}{\li(\eps)}\label{eq:deriv_am_11},\\
     \pimL(0,u_0,p_0) & := \lim_{\eps\searrow0} \frac{\Cl(\eps,u_0,p_\eps) - \Cl(0,u_0,p_\eps) }{\li(\eps)}\label{eq:deriv_am_12},
 \end{align}
 exist.  Then we have the following expansion:
\begin{equation}
    j(\eps) = j(0) + \li(\eps)(\mRi(u_0,p_0)+\pimL(0,u_0,p_0)) + o(\li(\eps)).
\end{equation}
In particular $\mRi(u_0,p_0)+\pimL(0,u_0,p_0) = \delta a^{(1)} - \delta f^{(1)} + \delta J_1^{(1)} + \delta J_2^{(1)}$, where $\delta a^{(1)}, \delta f^{(1)}$, $\delta J_1^{(1)}, \delta J_2^{(1)}$ are as in Proposition~\ref{thm:Amstutz_main}.

\item[(ii)] Let $\lii:\VR^+\to\VR^+$ be a function with $\displaystyle\lim_{\eps\searrow 0}\frac{\lii(\eps)}{\li(\eps)}=0$. Furthermore, assume that the assumptions under (i) hold and that the limits
 \begin{align}
     \mRii(u_0,p_0) & := \lim_{\eps\searrow0} \frac{\Cl(\eps,u_\eps,p_\eps) - \Cl(\eps,u_0,p_\eps)-\li(\eps)\mRi(u_0,p_0)}{\lii(\eps)},\label{eq:deriv_am_21}\\
     \piimL(0,u_0,p_0) & := \lim_{\eps\searrow0} \frac{\Cl(\eps,u_0,p_\eps) - \Cl(0,u_0,p_\eps)-\li(\eps)\pimL(0,u_0,p_0)}{\lii(\eps)},\label{eq:deriv_am_22}
 \end{align}
 exist.  Then we have the following expansion
\begin{align*}
    j(\eps) =   j(0) + &  \li(\eps)(\mRi(u_0,p_0)+\pimL(0,u_0,p_0))\\
                     + &  \lii(\eps)(\mRii(u_0,p_0)+\piimL(0,u_0,p_0))+ o(\lii(\eps)).
\end{align*}

\end{itemize}
\end{proposition}
\begin{proof}
   ad (i): Using that $\Cl(\eps,u_\eps,0) = \Cl(\eps,u_\eps,p_\eps)$ and $\Cl(0,u_0,p_\eps) = \Cl(0,u_0,0)$ we get
\begin{align}
j(\eps)-j(0)=& \Cl(\eps,u_\eps,0) - \Cl(0,u_0,0) \\ =& \Cl(\eps,u_\eps,p_\eps) - \Cl(\eps,u_0,p_\eps) \\
                                       & + \Cl(\eps,u_0,p_\eps) - \Cl(0,u_0,p_\eps).
\end{align}
Now the result follows by dividing by $\ell_1(\eps)$ for $\eps >0$ and passing to the limit $\eps\searrow0$.\newline
ad (ii): This follows the same lines as the proof of item (i) and is left to the reader.

\end{proof}

\begin{remark}
\begin{itemize}
 \item[(i)] Checking the expansions \eqref{eq:deriv_am_12},\eqref{eq:deriv_am_22} in applications usually requires some regularity of the state $u_0$ and knowledge of the asymptotics of the adjoint state $p_\eps$ on a small domain of size $\eps$.
   \item[(ii)] The computation of the asymptotic expansions \eqref{eq:deriv_am_11},\eqref{eq:deriv_am_21}
    requires the study of the asymptotic behaviour of $u_\eps$ on the whole domain $\Dsf$. 
    This often causes problems, especially in dimension two. The reader will find an application of this method in Section \ref{sec:deriv_am}.
\end{itemize}
\end{remark}


\subsection{Averaged adjoint method}
Another approach to compute topological derivatives was proposed in \cite{a_ST_2020a} and applied 
to nonlinear problems in \cite{a_GAST_2020a,a_ST_2020a,a_GAST_2021a} and used for the optimisation on surfaces in \cite{a_GAST_2020b}. Recall the Lagrangian function
\begin{equation}\label{eq:definition_lagrangian}
    \Cl(\eps,u,v) = J_\eps(u) + a_\eps(u,v)- f_\eps(v), \quad u\in \Cv,\; v\in\Cw.
\end{equation}
We henceforth assume that for all $(\varphi,q)\in \Cv\times\Cw$ and $\eps\ge 0$ the function 
    \begin{equation}
	s \mapsto \partial_u \Cl(\eps, su_\eps + (1-s)u_0,q)(\varphi):[0,1]\to \VR
\end{equation}
is continuously differentiable. With the Lagrangian we can define the averaged adjoint equation associated with state variables $u_\eps$ (solution of \eqref{eq:abstract_state} for $\eps >0$) and $u_0$ (solution of \eqref{eq:abstract_state} for $\eps=0$): find $q_\eps \in \Cw$, such that
    \begin{equation}\label{eq:aa_equation}
	\int_0^1 \!\!\partial_u \Cl(\eps,su_\eps + (1-s)u_0 ,q_\eps)(\varphi)\, ds=0 \quad \text{ for all } \varphi\in \Cv. 
\end{equation}
In addition plugging $\varphi = u_\eps - u_0$ into \eqref{eq:aa_equation} one obtains $\Cl(\eps,u_\eps,0) = \Cl(\eps,u_0,q_\eps)$ for $\eps >0$, so the Lagrangian only depends on the unperturbed state $u_0$ and the averaged adjoint variable $q_\eps$. 
We henceforth assume that the averaged adjoint equation admits a unique solution.

\begin{proposition}\label{prop:main_av}
\begin{itemize}
\item[(i)] Let $\li:\VR^+\to\VR^+$ be a function with $\displaystyle\lim_{\eps\searrow 0}\li(\eps)=0$. Furthermore, assume that the limits
 \begin{align}
     \mRi(u_0,q_0) & := \lim_{\eps\searrow0} \frac{\Cl(\eps,u_0,q_\eps) - \Cl(\eps,u_0,q_0)}{\li(\eps)}\label{eq:deriv_av_11},\\
     \pimL(0,u_0,q_0) & := \lim_{\eps\searrow0} \frac{\Cl(\eps,u_0,q_0) - \Cl(0,u_0,q_0) }{\li(\eps)}\label{eq:deriv_av_12},
 \end{align}
 exist.  Then we have the following expansion
\begin{equation}
    j(\eps) = j(0) + \li(\eps)(\mRi(u_0,q_0)+\pimL(0,u_0,q_0)) + o(\li(\eps)).
\end{equation}

\item[(ii)] Let $\lii:\VR^+\to\VR^+$ be a function with $\displaystyle\lim_{\eps\searrow 0}\frac{\lii(\eps)}{\li(\eps)}=0$. Furthermore, assume that the assumption under (i) holds and the limits
 \begin{align}
     \mRii(u_0,q_0) & := \lim_{\eps\searrow0} \frac{\Cl(\eps,u_0,q_\eps) - \Cl(\eps,u_0,q_0)-\li(\eps)\mRi(u_0,q_0)}{\lii(\eps)}, \label{eq:deriv_av_21}\\
     \piimL(0,u_0,q_0) & := \lim_{\eps\searrow0} \frac{\Cl(\eps,u_0,q_0) - \Cl(0,u_0,q_0)-\li(\eps)\pimL(0,u_0,q_0)}{\lii(\eps)},\label{eq:deriv_av_22}
 \end{align}
 exist.  Then we have the following expansion
\begin{align*}
    j(\eps) = j(0) + &  \li(\eps)(\mRi(u_0,q_0)+\pimL(0,u_0,q_0))\\
                   + & \lii(\eps)(\mRii(u_0,q_0)+\piimL(0,u_0,q_0)) +o(\lii(\eps)).
\end{align*}

\end{itemize}
\end{proposition}

\begin{proof}
ad (i): Recalling $\Cl(\eps,u_\eps,0) = \Cl(\eps,u_0,q_\eps)$ we have
\begin{align*}
j(\eps)-j(0)&=\Cl(\eps,u_\eps,0)-\Cl(0,u_0,0)\\
&=\Cl(\eps,u_0,q_\eps)-\Cl(0,u_0,q_0)\\
&=\Cl(\eps,u_0,q_\eps)-\Cl(\eps,u_0,q_0) + \Cl(\eps,u_0,q_0)-\Cl(0,u_0,q_0).
\end{align*}
Dividing by $\li(\eps)$ for $\eps >0$ and passing to the limit $\eps\searrow 0$, yields the result.\newline
ad (ii): Similar to item (i).
\end{proof}

The previous result can be readily generalised to compute the nth order topological derivative as 
shown in the following proposition.

\begin{proposition}[nth topological derivative]\label{thm:Sturm_main}
 Assume the following hypotheses hold.
\begin{itemize}
    \item[(1)] There exist numbers $\delta a^{(i)}$ and $\delta f^{(i)}$, $i=1,2,\ldots, n$ and a function $\ell_1:\VR^+\rightarrow\VR^+$ with $\displaystyle\lim_{\eps\searrow 0}\ell_1(\eps)=0$, such that
        \begin{align}
            (a_\eps-a_0)(u_0,q_0) & = \ell_1(\eps)\sum_{i=0}^{n-1} \eps^i\delta a^{(i+1)} + o(\eps^n\ell_1(\eps)), \label{eq:expand_a0}  \\
            (f_\eps-f_0)(u_0) & = \ell_1(\eps)\sum_{i=0}^{n-1} \eps^i\delta f^{(i+1)} + o(\eps^n\ell_1(\eps)), \label{eq:expand_f0}\\
            (J_\eps-J_0)(u_0) & = \ell_1(\eps)\sum_{i=0}^{n-1} \eps^i\delta J^{(i+1)} + o(\eps^n\ell_1(\eps)),\label{eq:expand_J0}
        \end{align}
            \item[(2)] There exist numbers $\delta A^{(i)}$ and $\delta F^{(i)}$, $i=1,2,\ldots, n$, such that
    \begin{align}
        (a_\eps-a_0)(u_0,q_\eps - q_0) & = \ell_1(\eps)\sum_{i=0}^{n-1} \eps^i\delta A^{(i+1)} + o(\eps^n\ell_1(\eps)), \label{eq:expand_a0_} \\
        (f_\eps-f_0)(q_\eps - q_0) & =  \ell_1(\eps)\sum_{i=0}^{n-1} \eps^i\delta F^{(i+1)} + o(\eps^n\ell_1(\eps)), \label{eq:expand_f0_}
    \end{align}
where $q_\eps\in \Cv$ is the averaged adjoint state satisfying
        \begin{equation}
            a_\eps(\varphi,v_\eps) = - \int_0^1\partial J_\eps(su_\eps + (1-s)u_0)(\varphi)\; ds \quad \text{ for all } \varphi\in \Cw.
        \end{equation}

\end{itemize}
Then the following expansion holds
\begin{equation}\label{eq:expansion_av_adjoint}
    J_\eps(u_\eps) = J_0(u_0)  + \ell_1(\eps)\sum_{i=0}^{n-1} \eps^i (\delta a^{(i+1)} - \delta f^{(i+1)} + \delta A^{(i+1)} - \delta F^{(i+1)}) + o(\eps^n\ell_1(\eps)).
\end{equation}
\end{proposition}
\begin{proof}
Similar to the proof of Proposition \ref{prop:main_av} we write
    \begin{equation}
        J_\eps(u_\eps) - J_0(u_0) = \Cl(\eps,u_0,q_\eps) - \Cl(\eps,u_0,q_0) + \Cl(\eps,u_0,q_0)-\Cl(0,u_0,q_0).
    \end{equation}
    The second term on the right hand side reads
    \begin{equation}
        \Cl(\eps,u_0,q_0)-\Cl(0,u_0,q_0) = (J_\eps-J_0)(u_0) +  (a_\eps - a_0)(u_0,q_0) - (f_\eps-f_0)(u_0).
    \end{equation}
    So using \eqref{eq:expand_a0}-\eqref{eq:expand_J0}, we can expand each difference in this expression. As 
    for the first difference on right hand side one has 
    \begin{align*}
        \Cl(\eps,u_0,q_\eps) - \Cl(\eps,u_0,q_0) = & (a_\eps-a_0)(u_0,q_\eps-q_0)-(f_\eps-f_0)(q_\eps-q_0) \\
                                               & + \underbrace{a_0(u_0,q_\eps-q_0) - f_0(q_\eps-q_0)}_{=0}.
    \end{align*}
    Therefore employing \eqref{eq:expand_a0_},\eqref{eq:expand_f0_} we can also expand these two differences and obtain the claimed formula \eqref{eq:expansion_av_adjoint}.
\end{proof}

\begin{remark}
    \begin{itemize}
\item[(i)]
    Checking the expansions \eqref{eq:deriv_av_12},\eqref{eq:deriv_av_22} in applications usually requires some regularity of the state $u_0$ and adjoint state $q_0=p_0$. However, the computation of this expansion is a straight forward application of Taylor's formula. The reader will find an application in Section~\ref{sec:deriv_av}
\item[(ii)]
    The computation of the asymptotic expansions \eqref{eq:deriv_av_11},\eqref{eq:deriv_av_21}
    requires the study of the asymptotic behavior of $q_\eps$ and therefore also of $u_\eps$. 
    This is the most difficult part and can be done by the compounded layer expansion involving 
    corrector equations (see for instance \cite{b_MANA_2000a},\cite{b_MANA_2000b}) as is presented in Section~\ref{sec:asymptotic_av}
\end{itemize}
\end{remark}


\subsection{Delfour's method}
In this section we discuss a method proposed by M.C. Delfour in \cite[Thm.3.3]{c_DE_2018b}. The definite advantage is that it uses the unperturbed adjoint 
equation and only requires the asymptotic analysis of the state equation, but it seems to come with the shortcoming that it is only applicable to certain cost functions; see \cite{a_GAST_2020a}. As before 
we let $\Cl$ be a Lagrangian function and denote by $u_\eps$ the perturbed state equation (solution to \eqref{eq:abstract_state} for $\eps \ge 0$) and $p_0$ the unperturbed adjoint equation (solution to \eqref{eq:adjoint_am} for $\eps=0$). Using the perturbed state and the unperturbed adjoint equation, Delfour proposed the following result for computing the first topological derivative, where we also incorparate the second order topological derivative.  

\begin{proposition}[\cite{c_DE_2018b}]\label{prop:main_df}
\begin{itemize}
 \item[(i)]   Let $\li:\VR^+\to\VR^+$ be a function with $\li\ge 0$ and $\displaystyle\lim_{\eps\searrow 0}\li(\eps)=0$. Furthermore, assume that the limits
\begin{align}
        \mRi_1(u_0, p_0) :=& \underset{\eps \searrow 0}{\mbox{lim}} \; \frac{1}{\li(\eps)}  \left[\Cl(\eps,u_\eps, p_0) -\Cl(\eps,u_0, p_0)- \Cl(\eps, u_0, p_0)(u_\eps - u_0)\right],\label{eq:deriv_df_11}\\
        \mRi_2(u_0, p_0) :=& \underset{\eps \searrow 0}{\mbox{lim}}\; \frac{1}{\li(\eps)}(\partial_u \Cl(\eps, u_0,p_0) - \partial_u \Cl(0,u_0,p_0))(u_\eps - u_0),\label{eq:deriv_df_12}\\
        \pimL(0, u_0, p_0) :=&  \lim_{\eps\searrow 0} \;\frac{1}{\li(\eps)}(\Cl(\eps,u_0, p_0) - \Cl(0,u_0, p_0)),\label{eq:deriv_df_13}
    \end{align}
    exist. Then the following expansion holds:
    \begin{align} \label{eq_dJ1_direct}
        j(\eps) = j(0) +  \li(\eps)((\mRi_1(u_0, p_0) + \mRi_2(u_0, p_0) + \pimL(0, u_0, p_0)) + o(\li(\eps)).
    \end{align}

\item[(ii)]  Let $\lii:\VR^+\to\VR^+$ be a function with $\lii\ge 0$ and $\displaystyle\lim_{\eps\searrow 0}\frac{\lii(\eps)}{\li(\eps)}=0$. Furthermore, assume that the assumptions under (i9 hold and that the limits
\begin{align}
    \mRii_1(u_0,p_0) &:= \underset{\eps \searrow 0}{\mbox{lim }} \frac{1}{\lii(\eps)}\bigg[\Cl(\eps,u_\eps,p_0) -\Cl(\eps,u_0,p_0)- \Cl(\eps, u_0,p_0)) (u_\eps - u_0) \nonumber \\
     & \hspace{7cm} - \li(\eps) \mRi_1(u_0, p_0)\bigg],\label{eq:deriv_df_21} \\
    \mRii_2(u_0,p_0) &:=\underset{\eps \searrow 0}{\mbox{lim }} \frac{1}{\lii(\eps)}\bigg[(\partial_u \Cl(\eps,u_0,p_0) - \partial_u \Cl(0,u_0,p_0))(u_\eps-u_0) \nonumber \\
                     & \hspace{6cm} - \li(\eps) \mRi_2(u_0, p_0) \bigg],\label{eq:deriv_df_22} \\
    \piimL(0,u_0, p_0) &:=\underset{\eps \searrow 0}{\mbox{lim }} \frac{1}{\lii(\eps)} \left[ \Cl(\eps,u_0,p_0) - \Cl(0,u_0,p_0) - \li(\eps) \pimL(0, u_0, p_0) \right],\label{eq:deriv_df_23}
\end{align}
exist. Then we have the following expansion: 
 \begin{equation}
     \begin{split}
     j(\eps) = j(0) & + \ell_1(\eps)(\mRi_1(u_0, p_0) + \mRi_2(u_0, p_0) + \pimL(0, u_0, p_0))  \\
                    & + \ell_2(\eps)(\mRii_1(u_0, p_0) +  \mRii_2(u_0,p_0)  + \piimL(0,u_0,p_0))  + o(\lii(\eps)).
 \end{split}
 \end{equation}
\end{itemize}
\end{proposition}

\begin{proof}
ad (i): Firstly note that by definition the unperturbed adjoint state $p_0$ satisfies \[\partial_u\Cl(0,u_0,p_0)(\varphi)=0\quad \text{ for } \varphi \in \Cw.\] Thus, we can write $j(\eps)-j(0)$ in the following way:
\begin{align}
\begin{split}
j(\eps)-j(0)=&\Cl(\eps,u_\eps,0)-\Cl(0,u_0,0)\\
=&\Cl(\eps,u_\eps,p_0)-\Cl(0,u_0,p_0)\\
=&\Cl(\eps,u_\eps,p_0)-\Cl(\eps,u_0,p_0)-\partial_u\Cl(\eps,u_0,p_0)(u_\eps-u_0)\\
&+\partial_u\Cl(\eps,u_0,p_0)(u_\eps-u_0)-\partial_u\Cl(0,u_0,p_0)(u_\eps-u_0)\\
&+\Cl(\eps,u_0,p_0)-\Cl(0,u_0,p_0).
\end{split}
\end{align}
Now dividing by $\li(\eps)$, $\eps>0$ and passing to the limit $\eps\searrow 0$ yields the result.\newline
ad (ii): This can be shown similarly to (i).
\end{proof}

\begin{remark}
Similarly to Amstutz' method and the averaged adjoint method, Delfour's method requires the asymptotic behaviour of $u_\eps$ on the whole domain to compute \eqref{eq:deriv_df_11},\eqref{eq:deriv_df_21}. This may be challenging in the analysis in dimension two for some cost functionals. Additionally, \eqref{eq:deriv_df_12},\eqref{eq:deriv_df_22} can be checked by smoothness assumptions on $p_0$ and $u_0$ and the knowledge of the asymptotics of $u_\eps$ on a small subset of size $\eps$. The remaining terms \eqref{eq:deriv_df_13},\eqref{eq:deriv_df_23} usually are computed making use of Taylor's expansion of $u_0$ and $p_0$ respectively.
\end{remark}

\paragraph{Overview of the employed adjoint equations}
The methods reviewed in the previous sections make use of three different adjoint equations. The method of Amstutz \cite{a_AM_2006a} uses an adjoint equation which depends on the unperturbed state variable:
\[
p_\eps \in \Cw:\; \partial_u \Cl(\eps,u_0,p_\eps)(\varphi) =0 \quad \text{ for all } \varphi \in \Cv.
\]
Delfour's method uses the usual perturbed adjoint equation:
\[
 v_\eps  \in \Cw:\; \partial_u \Cl(\eps,u_\eps,v_\eps)(\varphi) =0 \quad \text{ for all } \varphi \in \Cv.
\]
Finally there is the averaged adjoint method, which employs the averaged adjoint equation \cite{a_ST_2015a} and \cite{c_DEST_2016a}:
\[
    q_\eps \in \Cw:\; \int_0^1 \partial_u \Cl(\eps,su_\eps + (1-s)u_0,q_\eps)(\varphi)\;ds =0 \quad \text{ for all } \varphi \in \Cv.
\]


\section{Analysis of the perturbed state equation}\label{sec:sens_an_elasticity}
Let $\Omega \subset D$ open, $\omega\subset \VR^d$ be a bounded domain containing the origin $0\in \omega$ and let $x_0 \in \Dsf$. Moreover, we define the perturbation $\omega_\eps:=x_0+ \eps\omega$ for $\eps \ge 0$ at $x_0$. Consider the perturbed state solution 
of \eqref{eq:elasticity_linear} for $\Omega = \omega_\eps$, that is, find $u_\eps\in H^1(\Dsf)^d$, such that $u_\eps|_\Gamma=u_D$ and
\begin{equation}\label{eq:elasticity_linear_per}
    \int_\Dsf \VC_{\omega_\eps} \epsb(u_\eps):\epsb(\varphi)\;dx = \int_\Dsf f_{\omega_\eps}\cdot \varphi\;dx +\int_{\Gamma^N}u_N\cdot\varphi \;dS\quad \text{ for all } \varphi\in H^1_\Gamma(\Dsf)^d.
\end{equation}
In the following sections we are going to derive the asymptotic expansion of $u_\eps$ using the compounded layer method; see  \cite{b_MANA_2000a},\cite{b_MANA_2000b}. We note that this expansion has already been computed in \cite{a_BOCO_2017a} by means of Green's function and earlier in \cite{a_AM_2002} for $f_{\Omega}=0$. In the following two sections we state some preliminary results regarding the scaling of inequalities and remainder estimates, which will be needed later on.

\subsection{Scaling of inequalities}

In this section, we discuss the influence of a parametrised affine transformation $\Phi_\eps:\VR^d\to\VR^d$ onto norms and the scaling behavior of some well-known inequalities with respect to that parameter.

\begin{definition}
For $\eps>0$ we define the inflation of $\Dsf$ by $\Dsf_\eps:=\Phi_\eps^{-1}(\Dsf)$, where the affine linear transformation $\Phi_\eps$ is given by $\Phi_\eps(x):=x_0+\eps x$, for a fixed point $x_0\in \Dsf$.
\end{definition}

For convenience, we denote the inflated boundary as $\Gamma_\eps:=\Phi_\eps^{-1}(\Gamma)$. Since $\Phi_\eps$ is a bi-Lipschitz continuous map, it holds $\varphi\in H^1_\Gamma(\Dsf)^d$ if and only if $\varphi\circ \Phi_\eps \in H^1_{\Gamma_\eps}(\Dsf_\eps)^d$; see \cite[p.52, Thm.2.2.2]{b_ZI_1989a}. Furthermore, we introduce the scaled $H^1$ norm on $\Dsf_\eps$ by
\begin{equation}
\|\varphi\|_\eps:=\eps\|\varphi\|_{L_2(\Dsf_\eps)^d}+\|\nabla \varphi\|_{L_2(\Dsf_\eps)^{d\times d}} \quad \text{ for all } \varphi\in H^1(\Dsf_\eps)^d.
\end{equation}

\begin{lemma}\label{lma:scaling1}
Let $\Dsf \subset \VR^d$ be a bounded Lipschitz domain and let $\eps>0$.
\begin{itemize}

\item[(a)] For $1\le p<\infty$ and $\varphi \in L_p(\Dsf_\eps)^d$ there holds
\begin{equation}
\eps^{\frac{d}{p}}\|\varphi\|_{L_p(\Dsf_\eps)^d}=\|\varphi\circ \Phi_\eps^{-1}\|_{L_p(\Dsf)^d}.
\end{equation}

\item[(b)] For $1\le p<\infty$ and $\varphi \in W^1_p(\Dsf_\eps)^d$ there holds
\begin{equation}
\eps^{\frac{d}{p}-1}\|\nabla \varphi\|_{L_p(\Dsf_\eps)^{d\times d}}=\|\nabla(\varphi\circ \Phi_\eps^{-1})\|_{L_p(\Dsf)^{d\times d}}.
\end{equation}

\item[(c)] For $\varphi \in H^1(\Dsf_\eps)^d$ there holds
\begin{equation}
\|\varphi\circ \Phi^{-1}_\eps\|_{H^1(\Dsf)^d} = \eps^{\frac{d}{2}-1}\|\varphi\|_{\eps}.
\end{equation}

\end{itemize}
\end{lemma}

\begin{proof}
\begin{itemize}
\item[(a)] A change of variables yields
\begin{equation}
\|\varphi\|_{L_p(\Dsf_\eps)^d}^p=\eps^{-d}\int_\Dsf|\varphi\circ \Phi_\eps^{-1}|^p\; dx=\eps^{-d}\|\varphi\circ \Phi_\eps^{-1}\|_{L_p(\Dsf)^d}^p,
\end{equation}
where we used $|\Det(\nabla \Phi_\eps^{-1})|=\eps^{-d}$.

\item[(b)] Taking into account that $\nabla (\varphi\circ \Phi_\eps^{-1})=\eps^{-1}\nabla \varphi\circ \Phi_\eps^{-1}$, a change of variables yields
\begin{align}
\begin{split}
\|\nabla \varphi\|_{L_p(\Dsf_\eps)^{d\times d}}^p&=\eps^{-d}\int_\Dsf|\nabla \varphi\circ \Phi_\eps^{-1}|^p\; dx\\
&=\eps^{-d}\eps^p\int_\Dsf|\nabla(\varphi\circ \Phi_\eps^{-1})|^p\; dx=\eps^{p-d}\|\nabla(\varphi\circ \Phi_\eps^{-1})\|_{L_p(\Dsf)^{d\times d}}^p.
\end{split}
\end{align} 

\item[(c)] This follows from item (a) and (b).

\end{itemize}
\end{proof}

\begin{lemma}\label{lma:scaling2}
Let $\Dsf \subset \VR^d$ be a bounded Lipschitz domain, $\Gamma \subset \partial \Dsf$ and let $\eps >0$. Recall the definitions $\Dsf_\eps =\Phi_\eps^{-1}(\Dsf)$ and $\Gamma_\eps = \Phi_\eps^{-1}(\Gamma)$.
\begin{itemize}

    \item[(a)] For $1\le p\le q\le\infty$ there exists a constant $C>0$, such that
\begin{equation}
\|\varphi\|_{L_p(\Dsf_\eps)^d} \le C \eps^{\frac{d}{q}-\frac{d}{p}}\|\varphi\|_{L_q(\Dsf_\eps)^d}.
\end{equation}

\item[(b)] Let $d\ge3$ and $2^\ast$ denote the Sobolev conjugate of $2$. There exists a constant $C>0$, such that
\begin{equation}
\|\varphi\|_{L_{2^\ast}(\Dsf_\eps)^d}\le C\|\varphi\|_\eps.
\end{equation}

\item[(c)] Let $d=2$ and $\alpha>0$ small. There exists a constant $C>0$ and $\delta>0$ small, such that
\begin{equation}
\|\varphi\|_{L_{(2-\delta)^\ast}(\Dsf_\eps)^d}\le C \eps^{-\alpha}\|\varphi\|_\eps.
\end{equation}

\item[(d)] For $\varphi\in H^1(\Dsf_\eps)^d$ we have
\begin{equation}
\|\varphi\|_{L_2(\Gamma_\eps)^d}\le C \eps^{-\frac{1}{2}}\|\varphi\|_\eps.
\end{equation}

\item[(e)] Given a smooth connected domain $\Gamma \subset \partial \Dsf$, there is a continuous extension operator\newline $Z_{\Gamma_\eps}:H^{\frac{1}{2}}(\Gamma_\eps)^d\rightarrow H^1(\Dsf_\eps)^d$, such that
\begin{equation}
\|Z_{\Gamma_\eps}(\varphi)\|_\eps \le C (\eps^{\frac{1}{2}}\|\varphi\|_{L_2(\Gamma_\eps)^d}+|\varphi|_{H^{\frac{1}{2}}(\Gamma_\eps)^d}),\quad \text{ for all }\varphi \in H^{\frac{1}{2}}(\Gamma_\eps)^d,
\end{equation}
where $C>0$ is independent of $\eps$.

\item[(f)] Let $\Gamma \subset \partial \Dsf$ have positive measure. There exists a constant $C>0$, such that
\begin{equation}
\|\varphi\|_{L_2(\Dsf_\eps)^d}\le C\eps^{-1}\|\nabla \varphi\|_{L_2(\Dsf_\eps)^{d \times d}},\quad \text{ for all }\varphi\in H_{\Gamma_\eps}^1(\Dsf_\eps)^d.
\end{equation}

\end{itemize}
\end{lemma}

\begin{proof}
\begin{itemize}

\item[(a)] This is a direct consequence of Lemma \ref{lma:scaling1} item (a).

\item[(b)] We use Lemma \ref{lma:scaling1} item (a) and (b), and apply the Gagliardo-Nirenberg inequality \cite[p. 265, Thm. 2]{b_EV_2010} to the bounded domain $\Dsf$.
\begin{align}
\begin{split}
\|\varphi\|_{L_{2^\ast}(\Dsf_\eps)^d}&=\eps^{-\frac{d}{2^\ast}}\|\varphi \circ \Phi^{-1}_\eps\|_{L_{2^\ast}(\Dsf)^d}\\
&\le C\eps^{-\frac{d}{2^\ast}}\|\varphi \circ \Phi^{-1}_\eps\|_{H^1(\Dsf)^d}\\
&=C\eps^{\frac{d}{2}-\frac{d}{2^\ast}-1}\|\varphi\|_\eps.
\end{split}
\end{align}
Now the result follows from $\frac{d}{2}-\frac{d}{2^\ast}=1$.

\item[(c)] This time we need to apply the Gagliardo-Nirenberg inequality with respect to $p:=2-\delta<2$ and use the continuous embedding $L_2(\Dsf) \hookrightarrow L_{2-\delta}(\Dsf)$ on the bounded domain $\Dsf$:
\begin{align}
\begin{split}
\|\varphi\|_{L_{(2-\delta)^\ast}(\Dsf_\eps)^d}&=\eps^{-\frac{2}{(2-\delta)^\ast}}\|\varphi \circ \Phi^{-1}_\eps\|_{L_{(2-\delta)^\ast}(\Dsf)^d}\\
&\le C\eps^{-\frac{2}{(2-\delta)^\ast}}(\|\varphi \circ \Phi^{-1}_\eps\|_{L_{(2-\delta)}(\Dsf)^d}+\|\nabla(\varphi \circ \Phi^{-1}_\eps)\|_{L_{(2-\delta)}(\Dsf)^{d \times d}})\\
&\le C\eps^{-\frac{2}{(2-\delta)^\ast}}(\|\varphi \circ \Phi^{-1}_\eps\|_{L_2(\Dsf)^d}+\|\nabla(\varphi \circ \Phi^{-1}_\eps)\|_{L_2(\Dsf)^{d \times d}})\\
&=C\eps^{-\frac{2}{(2-\delta)^\ast}}\|\varphi\|_\eps.
\end{split}
\end{align}
Since $(2-\delta)^\ast$ diverges to $\infty$ as $\delta \searrow 0$, the result follows.

\item[(d)] This follows from a change of variables and the continuity of the trace operator.

\item[(e)] From \cite[p. 129, Thm. 8.8]{b_WL_1987a} we know there exists a continuous extension operator\newline $Z_{\Gamma}:H^{\frac{1}{2}}(\Gamma)^d\rightarrow H^1(\Dsf)^d$. Thus, a scaling argument similar to the previous ones yields the result.

\item[(f)] Item  (a) and (b) of Lemma \ref{lma:scaling1} and an application of Friedrich's inequality yield the result.

\end{itemize}
\end{proof}

\subsection{Remainder estimates}

We begin this section with the following auxiliary result.
\begin{lemma}\label{lma:remainder_est}
Let $V:\VR^d\rightarrow\VR^d\in H^1_{loc}(\VR^d)^d$ satisfy
\begin{equation}
|V(x)|=c_1 |x|^{-m}+\mathcal{O}(|x|^{-m-1}),\quad |\nabla V(x)|=c_2 |x|^{-m-1}+\mathcal{O}(|x|^{-m-2}),
\end{equation}
for $x\in B_\delta(0)^\mathsf{c}$, where $\delta>0$ is fixed, $m\in \VR$ and $c_1,c_2>0$ are constants. Then there is a constant $C>0$, such that for $\Gamma\subset \partial \Dsf$ and $\eps>0$ sufficiently small the following estimates hold:

\begin{itemize}

\item[(i)] $\|V\|_{L_2(\Gamma_\eps)^d}\le C \eps^{\frac{2m+1-d}{2}}$.

\item[(ii)] $|V|_{H^{\frac{1}{2}}(\Gamma_\eps)^d}\le C \eps^{\frac{2m+2-d}{2}}$.

\item[(iii)] $\|\nabla V\|_{L_2(\Gamma_\eps)^{d\times d}}\le C \eps^{\frac{2m+3-d}{2}}$.

\end{itemize}

\end{lemma}

\begin{proof}

\begin{itemize}

\item[(i)] Let $\displaystyle M:=\min_{x\in \Gamma}|x-x_0|>0$ and $\eps$ sufficiently small, such that the leading term of $V$ dominates the remainder for $x\in \Gamma_\eps$. Then we conclude
\begin{equation}
\|V\|^2_{L_2(\Gamma_\eps)^d}=\int_{\Gamma_\eps}|V|^2\;dS\le |\Gamma_\eps|(\eps^{-1}M)^{-2m}\le C\eps^{1-d+2m}.
\end{equation}
Now taking the square root shows the result.

\item[(ii)] Let $0<r_1<r_2$ such that $\partial D\subset S$, where $S:=B_{r_2}(x_0)\setminus B_{r_1}(x_0)$. Additionally, let $\eps$ sufficiently small, such that $\rho<\eps^{-1}r_1$. Now we apply a change of variables to integrate over the fixed domain and split the norm into two terms, which are treated separately. Therefore fix some $\delta>0$ sufficiently small. Then
\begin{align}
    |V|^2_{H^{\frac{1}{2}}(\partial \Dsf_\eps)^d}=&\int_{\partial \Dsf_\eps}\int_{\partial \Dsf_\eps}\frac{|V(x)-V(y)|^2}{|x-y|^d}\;dS_y dS_x \nonumber \\
=&\eps^{2-2d}\int_{\partial D}\int_{\partial D}\frac{|V(\Phi_\eps^{-1}(x))-V(\Phi_\eps^{-1}(y))|^2}{|\Phi_\eps^{-1}(x)-\Phi_\eps^{-1}(y)|^d}\;dS_y dS_x \nonumber\\ \nonumber
=&\eps^{2-d}\int_{\partial D}\int_{\partial D}\frac{|V(\Phi_\eps^{-1}(x))-V(\Phi_\eps^{-1}(y))|^2}{|x-y|^d}\;dS_y dS_x\\
=&\eps^{2-d}\int_{\partial D}\int_{\partial D\setminus B_\delta(x)}\frac{|V(\Phi_\eps^{-1}(x))-V(\Phi_\eps^{-1}(y))|^2}{|x-y|^d}\;dS_y dS_x\label{eq:nU_main_1}\\
&+\eps^{2-d}\int_{\partial D}\int_{\partial D\cap B_\delta(x)}\frac{|V(\Phi_\eps^{-1}(x))-V(\Phi_\eps^{-1}(y))|^2}{|x-y|^d}\;dS_y dS_x.\label{eq:nU_main_2}
\end{align}
In order to compute the first term \eqref{eq:nU_main_1}, we consider for each pair $(x,y)\in\partial \Dsf \times\partial \Dsf$ a smooth path $\varphi_{x,y}: [0,1]\rightarrow S$ satisfying $\varphi_{x,y}(0)=x$ and $\varphi_{x,y}(1)=y$. Since $V$ is smooth in $\Phi_\eps^{-1}(S)$, we can apply the  mean value theorem to the function $F(t):=V(\Phi_\eps^{-1}(\varphi_{x,y}(t)))$ and consider $\nabla(\Phi_\eps^{-1})=\eps^{-1}I_d$ to get
\begin{equation}
V(\Phi_\eps^{-1}(y))-V(\Phi_\eps^{-1}(x))=\int_0^1\eps^{-1}\nabla V(\Phi_\eps^{-1}(\varphi_{x,y}(s)))\varphi_{x,y}^\prime(s) \;ds.
\end{equation}
Thus, by Hölder's inequality we conclude
\begin{equation}
|V(\Phi_\eps^{-1}(y))-V(\Phi_\eps^{-1}(x))|\le \eps^{-1}\|\nabla V(\Phi_\eps^{-1}(\varphi_{x,y}(\cdot)))\|_{L^\infty(0,1)^{d \times d}}\|\varphi_{x,y}^\prime\|_{L^1(0,1)^d}.
\end{equation}
Since this inequality holds for every smooth path $\varphi_{x,y}$ connecting $x$ and $y$, the estimate holds for $d_{S}(x,y):=\displaystyle\inf_{\varphi_{x,y}[0,1]\rightarrow S}\|\varphi_{x,y}^\prime\|_{L^1(0,1)^d}$. Furthermore, since $S$ is bounded and path connected, the following estimate holds (see \cite[Thm 5.8]{b_DEZO_2011a}).  
\begin{equation}
d_{S}(x,y)\le C|x-y|,\quad \text{ for }x,y \in \bar{S}
\end{equation}
for some constant $C>0$ that only depends on $S$.
Additionally, considering the representation formula of $V$, we have $\|\nabla V (x)\|= c_2 |x|^{-m-1}+\mathcal{O}(|x|^{-m-2})$. Hence, choosing $\eps$ small enough, such that the leading order term dominates the remainder, we get
\begin{equation}\label{eq:nU_est}
\|\nabla \Ui(\Phi_\eps^{-1}(\varphi_{x,y}(s)))\|_{L^\infty(0,1)^{d \times d}}\le \max_{z\in \bar{S}}|\nabla \Ui(\Phi_\eps^{-1}(z))|\le C\eps^{m+1}.
\end{equation}
As a result, we conclude
\begin{align}\label{eq:H12_dist}
\begin{split}
&\eps^{2-d}\int_{\partial \Dsf}\int_{\partial \Dsf\setminus B_\delta(x)}\frac{|V(\Phi_\eps^{-1}(y))-V(\Phi_\eps^{-1}(x))|^2}{|x-y|^d}\;dS_y dS_x\\&\le\eps^{-d}\int_{\partial \Dsf}\int_{\partial \Dsf\setminus B_\delta(x)}\frac{C\eps^{2m+2}|x-y|^2}{|x-y|^d}\;dS_y dS_x\\
&\le\eps^{-d}\int_{\partial \Dsf}\int_{\partial \Dsf\setminus B_\delta(x)}\frac{C\eps^{2m+2}}{\delta^{d-2}}\;dS_y dS_x\\
&\le C\eps^{2m+2-d}.
\end{split}
\end{align}
The key here was to choose the set $S$ such that $\Phi_\eps^{-1}\circ\varphi_{x,y}([0,1])\subset B_\rho(0)^\mathsf{c}$ for every path $\varphi_{x,y}$.\newline
The second term \eqref{eq:nU_main_2} can be estimated by using a straight line connecting $x\in \partial \Dsf$ and $y\in \partial \Dsf$. Therefore, let $\varphi_{x,y}(t):=x+t(y-x)$, for $t\in[0,1]$. Since this time we only need to consider $(x,y)\in \partial \Dsf \times \partial \Dsf$ such that $|x-y|<\delta$, $\Phi_\eps^{-1}\circ\varphi_{x,y}([0,1])\subset B_\rho(0)^\mathsf{c}$ can be guaranteed by choosing $\delta$ sufficiently small. Again, an application of the mean value theorem yields
\begin{equation}
|V(\Phi_\eps^{-1}(x))-V(\Phi_\eps^{-1}(y))|^2\le \eps^{-2}\max_{z\in\bar{S_\delta}}|\nabla V(\Phi_\eps^{-1}(z))|^2|x-y|^2,
\end{equation}
where $S_\delta:=\bigcup_{x\in\partial \Dsf}B_\delta(x)$. Furthermore, a similiar estimation to \eqref{eq:nU_est} yields
\begin{equation}
\max_{z\in\bar{S_\delta}}|\nabla V(\Phi_\eps^{-1}(z))|^2\le C\eps^{2m+2}.
\end{equation}
Plugging this estimate into \eqref{eq:nU_main_2}, yields
\begin{align}
\begin{split}\label{eq:H12_close}
&\eps^{2-d}\int_{\partial \Dsf}\int_{\partial  \Dsf\cap B_\delta(x)}\frac{|V(\Phi_\eps^{-1}(x))-V(\Phi_\eps^{-1}(y))|^2}{|x-y|^d}\;dS_y dS_x\\&\le\eps^{-d}\int_{\partial \Dsf}\int_{\partial \Dsf\cap B_\delta(x)}\frac{\displaystyle\max_{z\in\bar{S_\delta}}|\nabla V(\Phi_\eps^{-1}(z))|^2}{|x-y|^{d-2}}\;dS_y dS_x\\
&\le C\eps^{2m+2-d}\int_{\partial \Dsf}\int_{\partial \Dsf\cap B_\delta(x)}\frac{1}{|x-y|^{d-2}}\;dy dx.
\end{split}
\end{align}
To finish our proof, we need to show that the integral on the right hand side is finite. Therefore, let $A_j(x):=B_{2^{(1-j)}\delta}(x)\setminus B_{2^{-j}\delta}(x)$, for $j \in \VN$. Hence, \[B_\delta(x)=\displaystyle\bigcup_{j\ge 1}A_j(x).\]
Now we can split the inner integral into layers according to these sets:
\begin{align}
\begin{split}
\int_{\partial \Dsf\cap B_\delta(x)}\frac{1}{|x-y|^{d-2}}\;dy&=\sum_{j\ge 1}\int_{\partial \Dsf\cap A_j(x)}\frac{1}{|x-y|^{d-2}}\;dy\\
&\le\sum_{j\ge 1}\int_{\partial \Dsf\cap A_j(x)}\frac{1}{[2^{-j}\delta]^{d-2}}\; dy\\
&\le\sum_{j\ge 1}2^{jd-2j}\delta^{2-d}|A_j(x)|\\
&=\delta^{2-d}\sum_{j\ge 1}2^{jd-2j}[C(2^{(1-j)d}\delta^d-2^{-jd}\delta^d)]\\
&=\delta^{2}C\sum_{j\ge 1}2^{jd-2j-jd}[2-1] =C\sum_{i\ge 1}\left(\frac{1}{4}\right)^j<\infty.
\end{split}
\end{align}
Hence, combining \eqref{eq:H12_dist} and \eqref{eq:H12_close} and using $|V|^2_{H^{\frac{1}{2}}(A)^d}\le|V|^2_{H^{\frac{1}{2}}(B)^d}$ for $A\subset B$, the result follows.

\item[(iii)] The proof follows the lines of item (i) and is therefore left to the reader.

\end{itemize}

\end{proof}


\subsection{First order asymptotic expansion}

Let $u_0\in H^1(\Dsf)^d$ denote the unique solution of the state equation \eqref{eq:elasticity_linear} for $\eps=0$. We henceforth refer to $u_0$ as the unperturbed state variable. By definition $u_0$ satisfies $u_0|_\Gamma=u_D$ and
\begin{equation}\label{eq:elasticity_linear_unper}
\int_\Dsf \VC_2 \epsb(u_0):\epsb(\varphi)\;dx = \int_\Dsf f_2 \cdot\varphi\;dx+\int_{\Gamma^N}u_N\cdot\varphi\;dS \quad \text{ for all } \varphi\in H^1_{\Gamma}(\Dsf)^d.
\end{equation}
\begin{assumption}\label{ass:regularity}
We henceforth assume that the $u_0\in C^3(B_\delta(x_0))$ for a small radius $\delta>0$.
\end{assumption}
\begin{lemma}
There is a constant $C>0$, such that for all $\eps >0$ sufficiently small there holds
\begin{equation}
\|u_\eps-u_0\|_{H^1(\Dsf)^d}\le C\eps^{\frac{d}{2}}.
\end{equation}
\end{lemma}
\begin{proof}
Subtracting \eqref{eq:elasticity_linear_per} for $\eps >0$ and \eqref{eq:elasticity_linear_unper} yields
\begin{align}\label{eq:diff_fixed}
\begin{split}
\int_\Dsf \VC_{\omega_\eps} \epsb(u_\eps-u_0):\epsb(\varphi)\;dx =& \int_{\omega_\eps} (\VC_2-\VC_1)\epsb(u_0):\epsb(\varphi)\;dx\\
&+\int_{\omega_\eps} (f_1-f_2) \cdot \varphi\;dx \quad \text{ for all } \varphi\in H^1_{\Gamma}(\Dsf)^d.
\end{split}
\end{align}
Therefore, testing with $\varphi:=u_\eps-u_0\in H^1_{\Gamma}(\Dsf)^d$, applying Korn's inequality to the gradient term on the left hand side followed by Friedrich's inequality and using H\"older's inequality to estimate the right hand side, leads to
\begin{equation}\label{eq:proof_exp1}
\|u_\eps-u_0\|_{H^1(\Dsf)^d}^2\le C(\|(\VC_2-\VC_1)\epsb(u_0)\|_{L_2(\omega_\eps)^{d\times d}}+\|f_1-f_2\|_{L_2(\omega_\eps)^d})\|u_\eps-u_0\|_{H^1(\Dsf)^d},
\end{equation}
for a positive constant $C>0$. In view of Assumption~\ref{ass:regularity}, we have $u_0\in C^3(B_\delta(x_0))$ for $\delta>0$ small enough and thus \eqref{eq:proof_exp1} can be further estimated to obtain
\begin{equation}
\|u_\eps-u_0\|_{H^1(\Dsf)^d}\le C\sqrt{\omega_\eps}(\|(\VC_2-\VC_1)\epsb(u_0)\|_{C(\omega_\eps)^{d\times d}}+\|f_1-f_2\|_{C(\omega_\eps)^d}).
\end{equation}
Now the result follows from $\sqrt{|\omega_\eps|}=\sqrt{|\omega|}\eps^{\frac{d}{2}}$.
\end{proof}

\begin{definition}\label{def:variation_state}
        For almost every $x\in \Dsf$ we define the first variation of the state $u_\eps$ by
    \begin{equation}
    \Ui_\eps(x) := \left(\frac{u_\eps - u_0}{\eps}\right)\circ \Phi_\eps(x), \quad \eps >0.
    \end{equation}
    The second variation of $u_\eps$ is defined by
    \begin{equation}
        U^{(2)}_\eps(x) := \frac{U^{(1)}_\eps(x) - U^{(1)}(x) - \eps^{d-1}u^{(1)}\circ \Phi_\eps}{\eps}, \quad \eps >0.
    \end{equation}
    More generally we define the $i$-th variation of $u_\eps$ for $i\ge 2$ by
    \begin{equation}
        U^{(i+1)}_\eps(x) := \frac{U^{(i)}_\eps(x) - U^{(i)}(x) - \eps^{d-2}u^{(i)}\circ \Phi_\eps}{\eps}, \quad \eps >0.
    \end{equation}
    Here, $U^{(i)}:\VR^d\to \VR^d$ are so-called boundary layer correctors and $u^{(i)}:\Dsf\to \VR^d$ are
    regular correctors. The functions $U^{(i)}$ aim to approximate $U_\eps^{(i)}$, however, they
    introduce an error at the boundary of $\Dsf$, which is corrected with the help of $u^{(i)}$.
\end{definition}

By extending $u_\eps$ and $u_0$ outside of $\Dsf$ by a continuous extension operator $E:H^1(\Dsf)^d\rightarrow H^1(\VR^d)^d$, one can view $U^{(1)}_\eps$ as an element of the Beppo-Levi space $BL(\VR^d)^d$.\newline
In the following, we show that the first variation of the state converges to a function $U\in BL(\VR^d)^d$ and determine an equation satisfied by this limit. The next Lemma helps us to handle the inhomogeneous Dirichlet boundary condition on $\Gamma_\eps$.

\begin{lemma}\label{lma:aux}
Let $A:\VR^{d\times d}\rightarrow \VR^{d\times d}$ be uniformly positive definite, $F_\eps:H^1_{\Gamma_\eps}(\Dsf_\eps)^d\rightarrow \VR$ be a linear and continuous functional with respect to $\|\cdot\|_\eps$ and $g_\eps\in H^{\frac{1}{2}}(\Gamma_\eps)^d$. Then there exists a unique $V_\eps \in H^1(\Dsf_\eps)^d$, such that
\begin{equation}\label{eq:aux_vol}
\int_{\Dsf_\eps}A \epsb(V_\eps):\epsb(\varphi)\; dx=F_\eps(\varphi)\quad\text{ for all }\varphi\in H^1_{\Gamma_\eps}(\Dsf_\eps)^d,
\end{equation}
\begin{equation}\label{eq:aux_bd}
V_\eps|_{\Gamma_\eps}=g_\eps.
\end{equation}
Furthermore, there exists a constant $C>0$ such that
\begin{equation}\label{eq:aux_ineq}
\|V_\eps\|_\eps\le C(\|F_\eps\|+\eps^{\frac{1}{2}}\|g_\eps\|_{L_2(\Gamma_\eps)^d}+|g_\eps|_{H^{\frac{1}{2}}(\Gamma_\eps)^d}).
\end{equation}
\end{lemma}

\begin{proof}
    Let $a_\eps(u,v):=\int_{\Dsf_\eps}A \epsb(u):\epsb(v)\;dx$, $u,v\in H^1(\Dsf_\eps)^d$. Thanks to our assumption $A$ is uniformly positive definite and thus one readily checks that $a_\eps$ is an elliptic and continuous bilinear form on $H^1_{\Gamma_\eps}(\Dsf_\eps)^d$ endowed with the scaled norm $\|\cdot\|_\eps$. Furthermore, let $Z_{\Gamma_\eps}$ denote the right-inverse extension operator of the trace operator $T_{\Gamma_\eps}$ and define $G_\eps:=Z_{\Gamma_\eps}(g_\eps)\in H^1(\Dsf_\eps)^d$.\newline
Now consider $\tilde{F_\eps}(\varphi):=F_\eps(\varphi)-a_\eps(G_\eps,\varphi)$. Since
\begin{equation}
    \begin{split}
        |\tilde{F_\eps}(\varphi)| & \le |F_\eps(\varphi)|+|a_\eps(G_\eps,\varphi)| \\    
                                  & \le  C\|F_\eps\|\|\varphi\|_\eps+C\|G_\eps\|_\eps\|\varphi\|_\eps
                                                                                    \le C\|\varphi\|_\eps \quad \text{ for all }\varphi\in H_{\Gamma_\eps}^1(\Dsf_\eps)^d,
\end{split}
\end{equation}
for a constant $C>0$, $\tilde{F_\eps}$ is continuous wih respect to $\|\cdot\|_\eps$. Thus, by the Lax-Milgram theorem, there exists a unique $u_\eps\in H_{\Gamma_\eps}^1(\Dsf_\eps)^d$, such that
\begin{equation}
a_\eps(u_\eps,\varphi)=\tilde{F_\eps}(\varphi)\quad\text{ for all }\varphi\in H^1_{\Gamma_\eps}(\Dsf_\eps)^d.
\end{equation}
Hence we conclude that $V_\eps:=u_\eps+G_\eps$ satisfies \eqref{eq:aux_vol} and \eqref{eq:aux_bd}. Uniqueness is guaranteed by the ellipticity of $a_\eps$. Applying the triangle inequality and using the continuity of $Z_{\Gamma_\eps}$ to estimate $\|G_\eps\|_\eps$ yields
\begin{align*}
    \|V_\eps\|_\eps & \le \|u_\eps\|_\eps+\|G_\eps\|_\eps\le C(\|\tilde{F_\eps}\|+\|G_\eps\|_\eps) \le C (\|F_\eps\|+\|G_\eps\|_\eps) \\
                    & \le C(\|F_\eps\|+\eps^{\frac{1}{2}}\|g_\eps\|_{L_2(\Gamma_\eps)^d}+|g_\eps|_{H^{\frac{1}{2}}(\Gamma_\eps)^d}),
\end{align*}
which shows \eqref{eq:aux_ineq} and finishes the proof.
\end{proof}

\begin{lemma}\label{lma:U1_rep}
There exists a unique solution $[U]\in\dot{BL}(\VR^d)^d$ to
\begin{equation}\label{eq:first_U}
    \int_{\VR^d}\VC_{\omega} \epsb([U]):\epsb(\varphi)\;dx=\int_\omega(\VC_2-\VC_1)\epsb(u_0)(x_0):\epsb(\varphi)\;dx \quad \text{ for all } \varphi\in \dot{BL}(\VR^d)^d.
\end{equation}
Moreover, there exists a representative $\Ui \in [U]$, which satisfies pointwise for $|x|\rightarrow \infty$:
\begin{equation}\label{eq:U1_def}
\Ui(x)=\Ri(x)+\mathcal{O}(|x|^{-d}),
\end{equation}
where $\Ri:\VR^d\rightarrow \VR^d$ satisfies
\begin{equation}
|\Ri(x)|=\begin{cases}
 b_2 |x|^{-1} \quad &\text{ for }d=2,\\
b_3 |x|^{-2} \quad &\text{ for }d=3,
\end{cases}
\end{equation}
for some constants $b_2,b_3\in \VR$.
\end{lemma}

\begin{proof}
Unique solvability of \eqref{eq:first_U} follows directly from the Lemma of Lax-Milgram. Thus, the only thing left to show is the asymptotic behaviour \eqref{eq:U1_def} of $\Ui$. 
For this we first note that $\Ui$ can be characterised by the following set of equations:
\begin{align}
-\Div(\VC_1 \epsb(\Ui))&=0\quad &\text{ in }\omega,\label{eq:strong_U1_a}\\
-\Div(\VC_2 \epsb(\Ui))&=0\quad &\text{ in }\bar{\omega}^{\mathsf{c}},\\
[\Ui]^+&=[\Ui]^-\quad &\text{ on }\partial\omega,\\
[\VC_1 \epsb(\Ui)n]^+-[\VC_2 \epsb(\Ui)n]^- &=(\VC_2-\VC_1)\epsb(\Ui)(x_0)n\quad &\text{ on }\partial\omega\label{eq:strong_U1_o}.
\end{align}
By \cite[p.76, Thm. 3.3.8]{b_AM_2008a} there are $f,g \in L_2(\partial \omega)^d$, such that
\begin{align}
\begin{split}
[\mathcal{S}^1_\omega f]^+-[\mathcal{S}^2_\omega g]^-&=0,\quad \text{ on }\partial\omega\\
[\VC_1 \epsb(\mathcal{S}^1_\omega f)n]^+-[\VC_2 \epsb(\mathcal{S}^2_\omega g)n]^- &=(\VC_2-\VC_1)\epsb(\Ui)(x_0)n,\quad \text{ on }\partial\omega,
\end{split}
\end{align}
where $\mathcal{S}^i_\omega f$ denotes the single layer potential on $\partial\omega$ with respect to the fundamental solution $\Gamma_i$, i.e. $\mathcal{S}^i_\omega h(x):=\int_{\partial\omega}\Gamma_i(x-y)h(y)\; dS(y)$, $i\in\{1,2\}$. Additionally, since $\int_{\partial\omega}(\VC_2-\VC_1)\epsb(\Ui)(x_0)n\;dS=0$, it follows that $\int_{\partial\omega}g\;dS=0$. Thus,
\[
    \Ui:=\begin{cases}
                \mathcal{S}^1_\omega f\quad\text{ in }\omega,\\
            \mathcal{S}^2_\omega g\quad\text{ in }\bar{\omega}^{\mathsf{c}},
        \end{cases}
\]
satisfies \eqref{eq:strong_U1_a}-\eqref{eq:strong_U1_o}. Furthermore, considering $\int_{\partial\omega}g\;dS=0$, a Taylor expansion of $\Gamma_2(x-y)$ in $y=0$ yields the desired asymptotic behaviour  \eqref{eq:U1_def}.
\end{proof}

\begin{theorem}\label{thm:main_U1}
Let $U^{(1)}_\eps$ be defined as in Definition \ref{def:variation_state} and $\alpha \in (0,1)$.
There exists a constant $C>0$, such that
\begin{align}
\|U_\eps^{(1)}-U^{(1)}\|_\eps\le
\begin{cases}
 C\eps \quad &\text{ for }d=3,\\
C\eps^{1-\alpha} \quad &\text{ for }d=2, 
\end{cases}
\end{align}
for $\eps$ sufficiently small.
\end{theorem}

\begin{proof}
We start by deriving an equation for $U^{(1)}_\eps$. For this purpose, we change variables in \eqref{eq:diff_fixed} to obtain
\begin{align}\label{eq:diff_scaled}
\begin{split}
\int_{\Dsf_\eps} \VC_{\omega} \epsb(U_\eps^{(1)}):\epsb(\varphi)\;dx =& \int_{\omega} (\VC_2-\VC_1)\epsb(u_0)\circ \Phi_\eps:\epsb(\varphi)\;dx\\
&+\eps\int_{\omega} (f_1-f_2)\circ \Phi_\eps\cdot \varphi\;dx \quad \text{ for all } \varphi\in H^1_{\Gamma_\eps}(\Dsf_\eps)^d.
\end{split}
\end{align}
Splitting the integral on the left hand side of \eqref{eq:first_U}, integrating by 
parts and using $\Div(\VC_2\epsb(\Ui))=0$ in $\bar{\omega}^{\mathsf{c}}$ yields
\begin{align}\label{eq:first_U_bounded}
\begin{split}
    \int_{\Dsf_\eps} \VC_{\omega} \epsb(U^{(1)}):\epsb(\varphi)\;dx = & \int_{\omega} (\VC_2-\VC_1)\epsb(x_0):\epsb(\varphi)\;dx-\int_{\VR^d\setminus\Dsf_\eps} \VC_{2} \epsb(U^{(1)}):\epsb(\tilde{\varphi})\;dx\\
= & \int_{\omega} (\VC_2-\VC_1)\epsb(x_0):\epsb(\varphi)\;dx-\int_{\Gamma_\eps^N} \VC_{2} \epsb(U^{(1)})\tilde{n}\cdot\tilde{\varphi}\;dS\\
& + \int_{\VR^d\setminus\Dsf_\eps} \Div(\VC_{2} \epsb(U^{(1)}))\cdot\tilde{\varphi}\;dx\\
= & \int_{\omega} (\VC_2-\VC_1)\epsb(x_0):\epsb(\varphi)\;dx+\int_{\Gamma_\eps^N} \VC_{2} \epsb(U^{(1)})n\cdot\varphi\;dS,
\end{split}
\end{align}
where $\varphi\in H^1_{\Gamma_\eps}(\Dsf_\eps)^d$, $\tilde{\varphi}$ denotes an extension to the whole domain and $\bar{n}$ denotes the outer normal vector on $\bar{\Dsf_\eps^{\mathsf{c}}}$. Subtracting \eqref{eq:diff_scaled} and \eqref{eq:first_U_bounded} results in
\begin{align}\label{eq:diffdiff_scaled}
\begin{split}
    \int_{\Dsf_\eps} \VC_{\omega} \epsb(U_\eps^{(1)}-U^{(1)}):\epsb(\varphi)\; dx= &  \int_{\omega} (\VC_2-\VC_1)[\epsb(u_0)\circ \Phi_\eps-\epsb(u_0)(x_0)]:\epsb(\varphi)\;dx\\
&+\eps\int_{\omega} (f_1-f_2)\circ \Phi_\eps\cdot \varphi\;dx\\
&-\int_{\Gamma_\eps^N} \VC_{2} \epsb(U^{(1)})n\cdot\varphi\;dS
\end{split}
\end{align}
for all $\varphi\in H^1_{\Gamma_\eps}(\Dsf_\eps)^d$. Now, we apply Lemma~\ref{lma:aux} to $V_\eps:=U^{(1)}_\eps-U^{(1)}$, $g_\eps:=-U^{(1)}|_{\Gamma_\eps}$ and $F^1_\eps$ defined as the right hand side of \eqref{eq:diffdiff_scaled}. Thus, we conclude that there exists a constant $C>0$, such that
\begin{equation}\label{eq:norm_bd}
\|U^{(1)}_\eps-U^{(1)}\|_\eps\le C(\|F^1_\eps\|+\eps^{\frac{1}{2}}\|U^{(1)}\|_{L_2(\Gamma_\eps)^d}+|U^{(1)}|_{H^{\frac{1}{2}}(\Gamma_\eps)^d}).
\end{equation}
To finish our proof, we need to estimate the norms of $F^1_\eps$ and $U^{(1)}$, which appear in \eqref{eq:norm_bd}. For the sake of clarity, we split the functional $F^1_\eps$ according to \eqref{eq:diffdiff_scaled} and treat each term separately.\newline
Let $\varphi\in H^1_{\Gamma_\eps}(\Dsf_\eps)$.

\begin{itemize}

\item At first we consider $\int_{\omega} (\VC_2-\VC_1)[\epsb(u_0)\circ \Phi_\eps-\epsb(u_0)(x_0)]:\epsb(\varphi)\;dx$. Since $u_0\in C^3(B_\delta(x_0))$, we get
\begin{equation}
\epsb(u_0)(x_0+\eps x)=\epsb(u_0)(x_0)+\nabla \epsb(u_0)(x_0)\eps x + o(\eps x).
\end{equation}
Together with an application of Hölder's inequality, we conclude
\begin{align}
\begin{split}
\left|\int_{\omega} [\epsb(u_0)\circ \Phi_\eps-\epsb(u_0)(x_0)]:\epsb(\varphi)\;dx\right| &\le C\|\epsb(u_0)\circ \Phi_\eps-\epsb(u_0)(x_0)\|_{L_2(\omega)}\|\epsb(\varphi)\|_{L_2(\omega)^{d\times d}}\\
&\le C\eps\|\epsb(\varphi)\|_{L_2(\omega)^{d\times d}}\le C\eps\|\varphi\|_\eps.
\end{split}
\end{align}
\item Next we consider $\eps\int_{\omega} (f_1-f_2)\circ \Phi_\eps\cdot \varphi\;dx$. Since we want to apply the Gagliardo-Nirenberg inequality, we need to distinguish between dimensions $d=2$ and $d=3$.\newline
For $d=3$ an application of Hölder's inequality with respect to $p=2^\ast$ and Lemma \ref{lma:scaling2}, item (b) yield
\begin{align}
\begin{split}
\left|\eps\int_{\omega} (f_1-f_2)\circ \Phi_\eps \cdot\varphi\;dx\right| \le C\eps \|\varphi\|_\eps.
\end{split}
\end{align}
For $d=2$ we apply Hölder's inequality with respect to $p=(2-\delta)^\ast$ for $\delta>0$ sufficiently small and Lemma \ref{lma:scaling2}, item (c) to obtain
\begin{align}
\begin{split}
\left|\eps\int_{\omega} (f_1-f_2)\circ \Phi_\eps \cdot\varphi\;dx\right| \le C\eps^{1-\alpha} \|\varphi\|_\eps.
\end{split}
\end{align}
for a constant $C>0$.

\item Finally, the last term can be estimated using Hölder's inequality and the scaled trace inequality (Lemma \ref{lma:scaling2} item (d)):
\begin{align}
\left|\int_{\Gamma_\eps^N} \VC_{2} \epsb(U^{(1)})n\cdot\varphi\;dS\right|&\le C \|\epsb(\Ui)\|_{L_2(\partial \Dsf_\eps)^{d\times d}}\|\varphi\|_{L_2(\partial \Dsf_\eps)^d}\\
&\le C\eps^{-\frac{1}{2}} \|\epsb(\Ui)\|_{L_2(\partial \Dsf_\eps)^{d\times d}}\|\varphi\|_\eps.
\end{align}
Thus, Lemma~\ref{lma:remainder_est}, item (iii) with $m=d-1$ yields
\begin{equation}
\left|\int_{\Gamma_\eps^N} \VC_{2} \epsb(U^{(1)})n\cdot\varphi\;dS\right|\le C\eps^{\frac{d}{2}}\|\varphi\|_\eps.
\end{equation}
\end{itemize}
Combining these estimates results in
\begin{align}\label{eq:F_est}
\|F^1_\eps\|\le
\begin{cases}
 C\eps \quad &\text{ for }d=3,\\
C\eps^{1-\alpha} \quad &\text{ for }d=2, 
\end{cases}
\end{align}
for a constant $C>0$. Furthermore, Lemma \ref{lma:remainder_est} item (i) and (ii) with $m=d-1$ yield
\begin{equation}\label{eq:g_est}
\|U^{(1)}\|_{L_2(\Gamma_\eps)^d}\le C\eps^{\frac{d-1}{2}}, \quad |U^{(1)}|_{H^{\frac{1}{2}}(\Gamma_\eps)^d}\le C \eps^{\frac{d}{2}}.
\end{equation}
Now plugging \eqref{eq:F_est} and \eqref{eq:g_est} into \eqref{eq:norm_bd} finishes the proof.
\end{proof}
\begin{remark}
Rewriting $U_\eps^1-U^1$ leaves us with the first order expansion \[u_\eps(x)\approx u_0(x)+\eps U^1(\Phi_\eps^{-1}(x)),\] where
\begin{equation}
\|u_\eps - [u_0+\eps U^1\circ \Phi_\eps^{-1}]\|_{H^1(\Dsf)^d}=
\begin{cases}
 \mathcal{O}(\eps^{\frac{d}{2}+1}) \quad &\text{ for }d=3,\\
\mathcal{O}(\eps^{\frac{d}{2}+1-\alpha}) \quad &\text{ for }d=2, \alpha>0.
\end{cases}
\end{equation}
\end{remark}


\subsection{Second order asymptotic expansion}
As mentioned earlier, the boundary layer corrector $U^{(1)}$ introduces an error at the boundary of $\Dsf$. Therefore, we introduce the regular corrector $u^{(1)}$, which compensates the boundary error. Additionally, in order to obtain a second order expansion we introduce the second order approximations $U^{(2)}$ and $u^{(2)}$. In contrast to the first order approximation 
we need to split the boundary layer corrector $U^{(2)}$ into two terms, where one solves a lower order equation and the other solves an analogue to $U^{(1)}$. Furthermore, we need to add the regular corrector $\vii$ to compensate the error introduced by $\Uii$. The following lemma describes each corrector:

\begin{lemma}\label{lma:U2_def}
\begin{itemize}

\item There is a unique solution $\vi\in H^1(\Dsf)^d$ with $u^{(1)}(x)=-R^{(1)}(x-x_0)$ on $\Gamma$, such that
\begin{equation}\label{eq:vi}
\int_\Dsf \VC_2 \epsb(\vi):\epsb(\varphi)\; dx=-\int_{\Gamma^N}\VC_2 \epsb(\Ri)(x-x_0)n\cdot \varphi\; dS,
\end{equation}
for all $\varphi \in H^1_\Gamma(\Dsf)^d$.

\item There is a solution $[U]\in\dot{BL}_p(\VR^d)^d$, to
\begin{equation}\label{eq:Uiih}
\int_{\VR^d}\VC_\omega\epsb([U]):\epsb(\varphi)\; dx=\int_\omega (\VC_2-\VC_1)[\nabla\epsb(u_0)(x_0)x]:\epsb(\varphi)\; dx,
\end{equation}
for all $\varphi \in \dot{BL}_{p^\prime}(\VR^d)^d$, where 
\[
    p= \begin{cases}
           2+\delta\quad &\text{ for }d=2,\\ 
           2\quad &\text{ for }d=3,
       \end{cases}
\] 
and $\delta>0$ small. Moreover, there exists a representative $\hat{U}^{(2)} \in [U]$, which satisfies pointwise for $|x|\rightarrow \infty$:
\begin{equation}
\hat{U}^{(2)}(x)=\hat{R}^{(2)}(x)+\mathcal{O}(|x|^{1-d}),
\end{equation}
where $\hat{R}^{(2)}:\VR^d\rightarrow \VR^d$ satisfies
\begin{equation}
|\hat{R}^{(2)}(x)|=\begin{cases}
 \hat{c}_2 \ln(|x|) \quad &\text{ for }d=2,\\
 \hat{c}_3 |x|^{-1} \quad &\text{ for }d=3,
\end{cases}
\end{equation}
for some constants $\hat{c}_2,\hat{c}_3\in \VR$.

\item There exists a solution $[U]\in\dot{BL}_p(\VR^d)^d$ to
\begin{equation}\label{eq:Uiit}
\int_{\VR^d}\VC_\omega\epsb([U]):\epsb(\varphi)\; dx=\int_\omega [(f_1(x_0)-f_2(x_0))]\cdot\varphi\; dx,
\end{equation}
for all $\varphi \in C^1_c(\VR^d)^d$, where 
\[
    p=\begin{cases}
         2+\delta\quad &\text{ for }d=2,\\ 
         2\quad &\text{ for }d=3,
     \end{cases}
 \] 
 and $\delta>0$ small. Moreover, there exists a representative $\tilde{U}^{(2)} \in [U]$, which satisfies pointwise for $|x|\rightarrow \infty$:
\begin{equation}\label{eq:asym_Uiit}
\tilde{U}^{(2)}(x)=\tilde{R}^{(2)}(x)+\mathcal{O}(|x|^{1-d}),
\end{equation}
where $\tilde{R}^{(2)}:\VR^d\rightarrow \VR^d$ satisfies
\begin{equation}
|\tilde{R}^{(2)}(x)|=\begin{cases}
 \tilde{c}_2 \ln(|x|) \quad &\text{ for }d=2,\\
\tilde{c}_3 |x|^{-1} \quad &\text{ for }d=3,
\end{cases}
\end{equation}
for some constants $\tilde{c}_2,\tilde{c}_3\in \VR$.

\item There is a unique solution $\vii\in H^1(\Dsf)^d$ with $\vii(x)=-\Rii(x-x_0)$ on $\Gamma$, such that
\begin{equation}\label{eq:vii}
\int_\Dsf \VC_2 \epsb(\vii):\epsb(\varphi)\; dx=-\int_{\Gamma^N}\VC_2 \epsb(\Rii)(x-x_0)n\cdot\varphi\; dS,
\end{equation}
for all $\varphi \in H^1_\Gamma(\Dsf)^d$, where $\Rii:=\hat{R}^{(2)}+\tilde{R}^{(2)}$.

\end{itemize}
\end{lemma}

\begin{remark}
Note that the requirement for $p$ to be greater than $2$ in dimension two is necessary to guarantee that the gradient of $\tilde{U}^{(2)}$ and $\hat{U}^{(2)}$ is in $L_p(\VR^d)^{d\times d}$, which is not true for $p=2$. In fact, there is a solution $[U]\in\dot{BL}(\VR^2)^2$ of \eqref{eq:Uiih}, but no representative $U\in[U]$ has the desired asymptotic representation.
\end{remark}

\begin{proof}
    Unique solvability of \eqref{eq:vi} and \eqref{eq:vii} follows from the Lax-Milgram theorem. 
 In order to show the existence and the desired representation formula of $\tilde{U}^{(2)}$, we use single layer potentials.  Note that a solution $U\in BL_p(\VR^d)^d$  of \eqref{eq:Uiit} can be characterised by the following set of equations:
\begin{align}
-\Div(\VC_1 \epsb(U))&=[(f_1(x_0)-f_2(x_0))]\quad &\text{ in }\omega,\label{eq:strong_U2_a}\\
-\Div(\VC_2 \epsb(U))&=0\quad &\text{ in }\bar{\omega}^{\mathsf{c}},\\
[U]^+&=[U]^-\quad &\text{ on }\partial\omega,\\
[\VC_1 \epsb(U)n]^+&=[\VC_2 \epsb(U)n]^-\quad &\text{ on }\partial\omega\label{eq:strong_U2_o}.
\end{align}
Now consider the volume potential $u(x):=\int_\omega \Gamma_1(x-y)[(f_1(x_0)-f_2(x_0))]\;dy$, for $x\in \omega$, which satisfies the inhomogeneous equation inside $\omega$. By \cite[p.76, Thm. 3.3.8]{b_AM_2008a} there are $f,g \in L_2(\partial \omega)^d$, such that
\begin{align}
[\mathcal{S}^1_\omega f]^+-[\mathcal{S}^2_\omega g]^-&=-u|_{\partial\omega}\quad \text{ on }\partial\omega,\\
[\VC_1 \epsb(\mathcal{S}^1_\omega f)n]^+-[\VC_2 \epsb(\mathcal{S}^2_\omega g)n]^- &=-(\VC_1 \epsb(u)n)|_{\partial\omega}\quad \text{ on }\partial\omega.
\end{align}
Finally
\[
\tilde{U}^{(2)} :=  \begin{cases}
                      u+\mathcal{S}^1_\omega f,\quad&\text{ in }\omega,\\
                 \mathcal{S}^2_\omega g,\quad&\text{ in }\bar{\omega}^{\mathsf{c}}.
                    \end{cases}
\]
satisfies \eqref{eq:strong_U2_a}-\eqref{eq:strong_U2_o} and a Taylor expansion of $\mathcal{S}^2_\omega g$ shows the asymptotic representation of \eqref{eq:asym_Uiit}. The proof for $\hat{U}^{(2)}$ is similar and therefore omitted.
\end{proof}

\begin{remark}
As a consequence of the equivalence relation defining the Beppo-Levi space, the function $\Uii$ is defined up to a constant. Thus, we are allowed to add arbitrary constants to the boundary layer corrector $\Uii$. As a result of the additive property of the leading term $\Rii(x)=\ln(x)$, we need to add the $\eps$ dependent constant $c\ln(\eps)$, with a suitable constant $c\in \VR$ in dimension $d=2$. In dimension $d=3$ this problem does not appear since the leading term $|x|^{-1}$ is multiplicative and therefore can be compensated by the factor $\eps^{d-2}$ found in definition \ref{def:variation_state}.
\end{remark}

\begin{remark}
A possible approach to approximate the solution $\tilde{U}^{(2)}$ of $\eqref{eq:Uiit}$ numerically is to consider for each $\eps>0$ the unique solution $K_\eps\in\overset{\circ}{W^1_p}(\Dsf)^d$ satisfying
\begin{equation}
\eps^2\int_\Dsf\VC_{\omega_\eps}\epsb(K_\eps):\epsb(\varphi)\;dx=\int_{\omega_\eps} [(f_1(x_0)-f_2(x_0))]\varphi\; dx
\end{equation}
for all $\varphi\in\overset{\circ}{W^1_p}(\Dsf)^d$. Applying Hölder's inequality and the Gagliardo-Nirenberg inequality, we get
\begin{align}
\left|\int_{\omega_\eps} [(f_1(x_0)-f_2(x_0))]\varphi\; dx\right|\le |\omega_\eps|^{\frac{1}{((p^\prime)^\ast)^\prime}}\|\nabla \varphi\|_{L_{p^\prime}(\Dsf)^{d\times d}}
\end{align}
and thus follow
\[\eps^2\|\nabla K_\eps\|_{L_p(\Dsf)^{d\times d}}\le C \eps^{\frac{(p^\prime)^\ast-1}{(p^\prime)^\ast}},\]
for a constant $C>0$. Now a change of variables yields $\|\nabla (K_\eps\circ \Phi_\eps)\|_{L_p(\Dsf_\eps)}\le C$. Hence, $K_\eps\circ \Phi_\eps$ is bounded in $\dot{BL}_p$ and therefore has a weakly convergent subsequent with limit $[U]$ satisfying \eqref{eq:Uiit}.
\end{remark}

\begin{theorem}\label{thm:U2_main}
Let $\Uii_\eps$ be defined in Definition \ref{def:variation_state} and $\alpha \in (0,1)$.
\begin{itemize}
\item[(i)] There exists a constant $C>0$, such that
\begin{align}\label{eq:res_U2}
\|\Uii_\eps-\Uii-\eps^{d-2}\vii\circ \Phi_\eps(x)\|_\eps&\le C\eps \quad &\text{ for }d=3,\\
\|\Uii_\eps-\Uii-\eps^{d-2}\vii\circ \Phi_\eps(x)+c\ln(\eps)\|_\eps&\le C\eps^{1-\alpha} \quad &\text{ for }d=2, 
\end{align}
for $\eps>0$ sufficiently small and a suitable constant $c\in\VR$.
\item[(ii)] For $d\in \{2,3\}$, there holds $\displaystyle\lim_{\eps\searrow 0}\|\eps^{-1}(\nabla\Ui_\eps-\nabla\Ui)-\nabla\Uii\|_{L_2(\omega)^{d\times d}}=0$.
\end{itemize}
\end{theorem}


\begin{proof}
ad (i): Similar to the estimation of the first order expansion we aim to apply Lemma \ref{lma:aux} in order to handle the inhomogeneous Dirichlet boundary condition on $\Gamma_\eps$. Hence, we start by deriving an equation satisfied by $\Uii_\eps-\Uii-\eps^{d-2}\vii\circ \Phi_\eps(x)=\eps \Uiii_\eps$. Dividing \eqref{eq:diffdiff_scaled} by $\eps>0$, changing variables in \eqref{eq:vi} and \eqref{eq:vii}, and integrating by parts in the exterior domain of \eqref{eq:Uiih} and \eqref{eq:Uiit} yields

\begin{equation}
\int_{\Dsf_\eps}\VC_\omega \epsb(\eps \Uiii_\eps):\epsb(\varphi)\;dx=F_\eps^2(\varphi)+F_\eps^3(\varphi),\quad \text{ for all }\varphi \in H_{\Gamma_\eps}^1(\Dsf_\eps),
\end{equation}
where
\begin{align}
\begin{split}
    F_\eps^2:= & {}\int_\omega [(f_1(\Phi_\eps(x))-f_2(\Phi_\eps(x)))-(f_1(x_0)-f_2(x_0))]\cdot\varphi \;dx\\
&+\int_\omega (\VC_2-\VC_1)[\eps^{-1}(\epsb(u_0)\circ \Phi_\eps-\epsb(u_0)(x_0))-\nabla \epsb(u_0)(x_0)x]:\epsb(\varphi) \;dx\\
&+\eps^{d-1}\int_\omega (\VC_2-\VC_1)[\epsb(\vi(\Phi_\eps))+\epsb(\vii(\Phi_\eps))]:\epsb(\varphi) \;dx
\end{split}\\
\begin{split}
    F_\eps^3 := & {}-\eps^{-1}\int_{\Gamma_\eps^N} [\VC_2 \epsb(\Ui)-\eps^d\VC_2 \epsb(\Ri)(\eps x)]n\cdot\varphi \;dS\\
&-\int_{\Gamma_\eps^N} [\VC_2 \epsb(\Uii)-\eps^{d-1}\VC_2 \epsb(\Rii)(\eps x)]n\cdot\varphi \;dS.
\end{split}
\end{align}
Since the bilinear form only depends on the symmetrised gradient of $\Uiii_\eps$, one readily checks that $\eps\Uiii_\eps+c\ln(\eps)$ satisfies

\begin{equation}
\int_{\Dsf_\eps}\VC_\omega \epsb(\eps \Uiii_\eps+c\ln(\eps)):\epsb(\varphi)\;dx=F_\eps^2(\varphi)+F_\eps^3(\varphi),\quad \text{ for all }\varphi \in H_{\Gamma_\eps}^1(\Dsf_\eps).
\end{equation}
Now we can apply Lemma \ref{lma:aux} to 
 \[   V_\eps:= \begin{cases}
                   \eps \Uiii_\eps \quad &\text{ for }d=3,\\
                  \eps\Uiii_\eps+c\ln(\eps) \quad &\text{ for }d=2,
               \end{cases}
\]

\[F_\eps:=F_\eps^2+F_\eps^3\] and

\begin{equation}\label{eq:exp2_g}
g_\eps:=\begin{cases}(\eps^{d-2}\Ri(\eps x) -\eps^{-1}\Ui)|_{\Gamma_\eps}+(\eps^{d-2}\Rii(\eps x)-\Uii)|_{\Gamma_\eps} \quad &\text{ for }d=3,\\
(\eps^{d-2}\Ri(\eps x) -\eps^{-1}\Ui)|_{\Gamma_\eps}+(\eps^{d-2}\Rii(\eps x)-\Uii+c\ln(\eps))|_{\Gamma_\eps} \quad &\text{ for }d=2.
\end{cases}
\end{equation}
Hence, we get the apriori estimate
\begin{equation}\label{eq:apriori_U2}
\|V_\eps\|_\eps \le C(\|F_\eps\|+\eps^{\frac{1}{2}}\|g_\eps\|_{L_2(\Gamma_\eps)^d}+|g_\eps|_{H^{\frac{1}{2}}(\Gamma_\eps)^d}).
\end{equation}
Due to great similarity between $d=2$ and $d=3$, we will discuss both cases together and only highlight the terms that have to be treated separately. Thus, if not further specified, let $d=2,3$. Again, we start by estimating $\|F_\eps\|$. Let $\varphi \in H^1_{\Gamma_\eps}(\Dsf_\eps)$.
\begin{itemize}

\item A Taylor expansion of $(f_1(\Phi_\eps(x))-f_2(\Phi_\eps(x)))$ at $x_0$, Hölder's inequality and Lemma~\ref{lma:scaling2}, item (b), (c) yield
\begin{equation}
\left|\int_\omega [(f_1(\Phi_\eps(x))-f_2(\Phi_\eps(x)))-(f_1(x_0)-f_2(x_0))]\cdot\varphi \;dx\right|\le \begin{cases}C\eps \|\varphi\|_\eps \quad &\text{ for }d=3,\\
C\eps^{1-\alpha}\|\varphi\|_\eps \quad &\text{ for }d=2, \alpha>0,
\end{cases}
\end{equation}
for a constant $C>0$.

\item Since $u_0$ is three times differentiable in a neighbourhood of $x_0$, there is a constant $C>0$, such that $|\eps^{-1}(\epsb(u_0)\circ \Phi_\eps-\epsb(u_0)(x_0))-\nabla \epsb(u_0)(x_0)x|\le C\eps$, for $x\in\omega$. Hence, Hölder's inequality yields
\begin{equation}
\left|\int_\omega (\VC_2-\VC_1)[\eps^{-1}(\epsb(u_0)\circ \Phi_\eps-\epsb(u_0)(x_0))-\nabla \eps(u_0)(x_0)x]:\epsb(\varphi) \;dx\right|\le C\eps \|\varphi\|_\eps.
\end{equation}

\item Furthermore, by Hölder's inequality we get
\begin{equation}
\left|\eps^{d-1}\int_\omega (\VC_2-\VC_1)[\epsb(\vi(\Phi_\eps))+\epsb(\vii(\Phi_\eps))]:\epsb(\varphi) \;dx\right|\le C\eps \|\varphi\|_\eps,
\end{equation}
for a constant $C>0$.
\end{itemize}

Next we consider the boundary integral terms:

\begin{itemize}

\item Here we note that $\epsb(\Ui)-\eps^d\epsb(\Ri)(\eps x)$ cancels out the leading term of $\Ui$ on $\partial \Dsf$. Thus, we can apply Hölder's inequality, Lemma~\ref{lma:remainder_est}, item (iii) with $m=d$ and the scaled trace inequality to conclude
\begin{equation}
\left|\eps^{-1}\int_{\Gamma_\eps^N} [\VC_2 \epsb(\Ui)-\eps^d\VC_2 \epsb(\Ri)(\eps x)]n\cdot\varphi \;dS\right|\le C\eps^{\frac{d}{2}}\|\varphi\|_\eps,
\end{equation}
for a constant $C>0$.

\item Similarly, we deduce from Lemma~\ref{lma:remainder_est}, item (iii) with $m=d-1$ that there is a constant $C>0$, such that
\begin{equation}
\left|\int_{\Gamma_\eps^N} [\VC_2 \epsb(\Uii)-\eps^{d-1}\VC_2 \epsb(\Rii)(\eps x)]n\cdot\varphi \;dS\right|\le C\eps^{\frac{d}{2}}\|\varphi\|_\eps.
\end{equation}

\end{itemize}
Combining the previous estimates yields
\begin{equation}\label{eq:exp2_F_est}
\|F_\eps\|\le 
\begin{cases}
    C\eps\quad&\text{ for }d=3,\\
C\eps^{1-\alpha}\quad&\text{ for }d=2,
\end{cases}
\end{equation}
for a constant $C>0$.
Finally, we recall that $g_\eps$ is defined in \eqref{eq:exp2_g}. At this point we choose the constant $c\in\VR$, such that
\[
    \Rii(x)=\Rii(\eps x)+c\ln(\eps)\quad \text{ in }d=2.
\]
Then, by Lemma~\ref{lma:remainder_est}, item (i), (ii) with $m=d$ and $m=d-1$ respectively, there is a constant $C>0$, such that

\begin{equation}\label{eq:exp2_g_est}
\eps^{\frac{1}{2}}\|g_\eps\|_{L_2(\Gamma_\eps)^d}+|g_\eps|_{H^{\frac{1}{2}}(\Gamma_\eps)^d}\le C \eps^{\frac{d}{2}}.
\end{equation}
Now we can plug \eqref{eq:exp2_F_est} and \eqref{eq:exp2_g_est} into the apriori estimate \eqref{eq:apriori_U2}, which shows \eqref{eq:res_U2}.\newline
ad (ii): By the triangle inequality, we have
\begin{align}
\begin{split}
    \|\eps^{-1}(\nabla(\Ui_\eps)-\nabla(\Ui))-&\nabla(\Uii)\|_{L_2(\omega)^{d\times d}} \\
    \le &  \|\nabla(\Uii_\eps-\Uii-\eps^{d-2}\vii\circ \Phi_\eps)\|_{L_2(\omega)^{d\times d}}\\
&+\eps^{d-1}\|\nabla(\vi)\circ \Phi_\eps\|_{L_2(\omega)^{d\times d}}+\eps^{d-1}\|\nabla(\vii)\circ \Phi_\eps\|_{L_2(\omega)^{d\times d}}\\
    \le &  C\eps^{1-\alpha},
\end{split}
\end{align}
for a positive constant $C$. This shows (ii) and therefore finishes the proof.

\end{proof}

\begin{remark}\label{rem:U2}
Note that by the triangle inequality one has
\begin{align*}
    \|\eps^{-1}(\Ui_\eps-\Ui)-\Uii\|_\eps  \le & \|\Uii_\eps-\Uii-\eps^{d-2}\vii\circ \Phi_\eps(x)\|_\eps+\|\eps^{d-2}\vi\circ \Phi_\eps(x)\|_\eps \\
                                              & +\|\eps^{d-2}\vii\circ \Phi_\eps(x)\|_\eps \\
                                          \le & C(\eps+\eps^{\frac{1}{2}}),
\end{align*}
for $d=3$ and a constant $C>0$.
Thus in dimension $d=3$, the correction of $\vii$ is not necessary to achieve convergence of $\Uii_\eps$. In fact, sparing the corrector results in a slower convergence of order $\eps^{\frac{1}{2}}$ compared to the corrected order $\eps$.
\end{remark}

\begin{remark}
In order to give a better understanding of the scheme of the asymptotic expansion, we would like to point out the main difference between the first and second order expansion, which is the slower decay of the boundary layer corrector $\Uii$ compared to $\Ui$. As a result, there was no necessity to introduce the regular corrector $\vi$ in the first order expansion, whereas $\vii$ was needed to obtain the desired order of at least $\eps^{1-\alpha}$, for $\alpha>0$ small.\newline
Additionally, one should note that boundary layer correctors appearing in higher order expansions have asymptotics similar to $\Uii$ and therefore demand a correction of the associated regular correctors. Thus, the scheme of the asymptotic expansion of arbitrary order resembles the second order expansion given in this chapter, rather than the first order expansion.
\end{remark}


\section{Analysis of the perturbed adjoint equation}\label{sec:sens_adjoint}
In this section we study the asymptotic analysis of the Amstutz' adjoint equation and 
the averaged adjoint equation for our elasticity model problem. We shall first exam 
Amstutz' adjoint and derive its asymptotic expansion up to order two. 

\subsection{Amstutz' adjoint equation}\label{sec:asymptotic_am}
The adjoint state $p_\eps$ ,$\eps\ge0$ satisfies
\begin{equation}
  p_\eps\in H^1_\Gamma(\Dsf)^d,\quad  \partial_u G(\eps,u_0,p_\eps)(\varphi) =0 \quad \text{ for all } \varphi \in H^1_\Gamma(\Dsf)^d.
\end{equation}
With the cost function defined in \eqref{eq:cost_func} this equation reads explicitly
\begin{align}\label{eq:adj_per}
\begin{split}
    \int_\Dsf \VC_{\omega_\eps} \epsb(\varphi):\epsb(p_\eps) \;dx= & -\gamma_f\int_\Dsf f_{\omega_\eps}\cdot\varphi\;dx-2\gamma_m\int_{\Gamma^m}(u_0-u_m)\cdot\varphi\;dS \\
                                                                   & -2\gamma_g\int_{\Dsf}[\nabla u_0-\nabla u_d]:\nabla\varphi\; dx,
\end{split}
\end{align}
for all $\varphi\in H^1_\Gamma(\Dsf)^d$. Similarly, the unperturbed adjoint equation reads
\begin{align}\label{eq:adj_unper}
\begin{split}
    \int_\Dsf \VC_{2} \epsb(\varphi):\epsb(p_0) \;dx=& -\gamma_f\int_\Dsf f_{2}\cdot\varphi\;dx-2\gamma_m\int_{\Gamma^m}(u_0-u_m)\cdot\varphi\;dS \\
                                                    & -2\gamma_g\int_{\Dsf}[\nabla u_0-\nabla u_d]:\nabla\varphi\; dx,
\end{split}
\end{align}
for all $\varphi\in H^1_\Gamma(\Dsf)^d$.
Note that the $\eps$ dependence of $p_\eps$ is only via the coefficients $\VC_{\omega_\eps}$ and $f_{\omega_\eps}$. This is a definite advantage over 
the averaged adjoint method, where also the perturbed state variable $u_\eps$ appears.

We now compute an asymptotic expansion of $p_\eps$ in a similar fashion to the direct state $u_\eps$. Therefore we define the variation of the adjoint state $P^{(i)}_\eps$ for $i\ge1$ in analogy to the definition of the variation of the direct state (Defintion \ref{def:variation_state}), where we replace the boundary layer correctors $U^{(i)}$ by similar correctors $P^{(i)}$ adapted to the new inhomogeneity and the regular correctors $u^{(i)}$ are replaced by correctors $p^{(i)}$ matching $P^{(i)}$.
\begin{lemma}\label{lma:P1_rep}
There exists a solution $[P]\in\dot{BL}(\VR^d)^d$ to
\begin{equation}\label{eq:first_P}
    \int_{\VR^d}\VC_\omega \epsb(\varphi):\epsb([P])\;dx=\int_\omega(\VC_2-\VC_1)\epsb(\varphi):\epsb(p_0)(x_0)\;dx,
\end{equation}
for all $\varphi\in \dot{BL}(\VR^d)^d$.
Moreover, there exists a representative $\PPi \in [P]$, which satisfies pointwise for $|x|\rightarrow \infty$:
\begin{equation}\label{eq:P1_def}
\PPi(x)=\Si(x)+\mathcal{O}(|x|^{-d}),
\end{equation}
where $\Si:\VR^d\rightarrow \VR^d$ satisfies
\begin{equation}
|\Si(x)|=\begin{cases}
    b_2 |x|^{-1} \quad &\text{ for }d=2,\\
    b_3 |x|^{-2} \quad &\text{ for }d=3,
\end{cases}
\end{equation}
for some constants $b_2,b_3\in \VR$.
\end{lemma}

\begin{proof}
Using the adjoint tensor $\VC^\top_\omega:\VR^{d\times d}\rightarrow \VR^{d\times d}$, we can rewrite \eqref{eq:first_P} to get
\begin{equation}\label{eq:first_P_adj}
    \int_{\VR^d}\epsb(\varphi):\VC_\omega^\top\epsb([P])\;dx=\int_\omega\epsb(\varphi):(\VC_2^\top-\VC_1^\top)\epsb(p_0)(x_0)\;dx.
\end{equation}
Thus, using single layer potentials, the proof follows the lines of Lemma \ref{lma:U1_rep}.
\end{proof}

\begin{theorem}\label{thm:main_P1}
For $\alpha \in (0,1)$ and $\eps>0$ sufficiently small there is a constant $C>0$, such that
\begin{align}\label{eq:exp1_est_adj}
\|P_\eps^{(1)}-P^{(1)}\|_\eps\le
\begin{cases}
 C\eps \quad &\text{ for }d=3,\\
C\eps^{1-\alpha} \quad &\text{ for }d=2.
\end{cases}
\end{align}
\end{theorem}

\begin{proof}
Similarly to the analysis of the direct state, we derive an equation of the form
\begin{align}\label{eq:P1_rhs}
    \int_{\Dsf_\eps} \VC_{\omega}\epsb(\varphi):\epsb(\PPi_\eps-\PPi)\; dx=G^1_\eps(\varphi) \quad \text{ for all } \varphi\in H^1_{\Gamma_\eps}(\Dsf_\eps)^d,
\end{align}
where the right hand side satisfies
\begin{equation}\label{eq:F_est_adj}
\|G^1_\eps\|\le
\begin{cases}
C \eps\quad &\text{ for }d=3,\\
C\eps^{1-\alpha}\quad &\text{ for }d=2.
\end{cases}
\end{equation}
A detailed derivation and estimation of the functional $G^1_\eps$ can be found in the Appendix (Section \ref{sec:ad_P1}).
In view of Lemma~\ref{lma:aux}, we now estimate the boundary integral terms. Since $\PPi_\eps|_{\Gamma_\eps}=0$ we follow from Lemma \ref{lma:remainder_est} item (i), (ii) with $m=d-1$ that there is a constant $C>0$, such that

\begin{equation}\label{eq:g_est_adj}
\eps^{\frac{1}{2}}\|\PPi_\eps-\PPi\|_{L_2(\Gamma_\eps)^d}+|\PPi_\eps-\PPi|_{H^{\frac{1}{2}}(\Gamma_\eps)^d}\le C\eps^{\frac{d}{2}}.
\end{equation}
Thus, considering \eqref{eq:F_est_adj} and \eqref{eq:g_est_adj}, an application of Lemma \ref{lma:aux} with $A=\VC_\omega^\top$ shows \eqref{eq:exp1_est_adj}, which finishes our proof.

\end{proof}

We now continue with the second order expansion. Similar to the state variable expansion, we therefore introduce a number of correctors in the following Lemma, which approximate the first order expansion inside $\omega_\eps$ and on the boundary $\partial \Dsf$ respectively.

\begin{lemma}\label{lma:P2_def}
\begin{itemize}

\item There is a unique solution $\wi\in H^1(\Dsf)^d$ with $p^{(1)}(x)=-S^{(1)}(x-x_0)$ on $\Gamma$, such that
\begin{equation}\label{eq:wi}
\int_\Dsf \VC_2 \epsb(\varphi):\epsb(\wi)=-\int_{\Gamma^N}\VC_2^\top \epsb(\Si)(x-x_0)n \cdot\varphi\; dS,
\end{equation}
for all $\varphi \in H^1_\Gamma(\Dsf)^d$.

\item There is a solution $[P]\in\dot{BL}_p(\VR^d)^d$, to
\begin{align}\label{eq:Piih}
\begin{split}
\int_{\VR^d}\VC_\omega\epsb(\varphi):\epsb([P])\; dx&=\int_\omega (\VC_2-\VC_1)\epsb(\varphi):[\nabla\epsb(p_0)(x_0)x]\; dx,
\end{split}
\end{align}
for all $\varphi \in \dot{BL}_{p^\prime}(\VR^d)^d$, where 
\[
    p=\begin{cases}
     2+\delta\quad &\text{ for }d=2,\\ 
  2\quad &\text{ for }d=3,
  \end{cases}
\] 
  and $\delta>0$ small. Moreover, there exists a representative $\hat{P}^{(2)} \in [P]$, which satisfies pointwise for $|x|\rightarrow \infty$:
\begin{equation}
\hat{P}^{(2)}(x)=\hat{S}^{(2)}(x)+\mathcal{O}(|x|^{1-d}),
\end{equation}
where $\hat{S}^{(2)}:\VR^d\rightarrow \VR^d$ satisfies
\begin{equation}
|\hat{S}^{(2)}(x)|=\begin{cases}
 \hat{c}_2 \ln(|x|) \quad &\text{ for }d=2,\\
\hat{c}_3 |x|^{-1} \quad &\text{ for }d=3,
\end{cases}
\end{equation}
for some constants $\hat{c}_2,\hat{c}_3\in \VR$.

\item There is a solution $[P]\in\dot{BL}_p(\VR^d)^d$ to
\begin{equation}\label{eq:Piit}
\int_{\VR^d}\VC_\omega\epsb(\varphi):\epsb([P])\; dx=\gamma_f\int_\omega [f_2(x_0)-f_1(x_0)]\cdot\varphi\; dx,
\end{equation}
for all $\varphi \in C^1_c(\VR^d)^d$, where 
\[
    p=\begin{cases}
    2+\delta\quad &\text{ for }d=2,\\ 
2\quad &\text{ for }d=3,
\end{cases}
\] 
and $\delta>0$ small. Moreover, there exists a representative $\tilde{P}^{(2)} \in [P]$, which satisfies pointwise for $|x|\rightarrow \infty$:
\begin{equation}
\tilde{P}^{(2)}(x)=\tilde{S}^{(2)}(x)+\mathcal{O}(|x|^{1-d}),
\end{equation}
where $\tilde{S}^{(2)}:\VR^d\rightarrow \VR^d$ satisfies
\begin{equation}
|\tilde{S}^{(2)}(x)|=\begin{cases}
 \tilde{c}_2 \ln(|x|) \quad &\text{ for }d=2,\\
\tilde{c}_3 |x|^{-1} \quad &\text{ for }d=3,
\end{cases}
\end{equation}
for some constants $\tilde{c}_2,\tilde{c}_3\in \VR$.

\item There is a unique solution $\wii\in H^1(\Dsf)^d$ with $\wii(x)=-\Sii(x-x_0)$ on $\Gamma$, such that
\begin{equation}\label{eq:wii}
\int_\Dsf \VC_2 \epsb(\varphi):\epsb(\wii)=-\int_{\Gamma^N}\VC_2^\top \eps(\Sii)(x-x_0)n\cdot\varphi\; dS
\end{equation}
for all $\varphi \in H^1_\Gamma(\Dsf)^d$, where $\Sii:=\hat{S}^{(2)}+\tilde{S}^{(2)}$.
\end{itemize}
\end{lemma}
\begin{proof}
Rewriting these equations with the help of the adjoint operator $\VC_\omega^\top$ leads to a proof similar to Lemma \ref{lma:U2_def}.
\end{proof}

Now we are able to state our main result regarding the second order expansion of the adjoint state variable $p_\eps$:

\begin{theorem}\label{thm:main_P2}
\begin{itemize}
\item[(i)] There exists a constant $C>0$, such that
\begin{align}\label{eq:P2_diff_est}
\|\Pii_\eps-\Pii-\eps^{d-2}\wii\circ \Phi_\eps\|_\eps&\le C\eps \quad &\text{ for }d=3,\\
\|\Pii_\eps-\Pii-\eps^{d-2}\wii\circ \Phi_\eps+c\ln(\eps)\|_\eps&\le C\eps^{1-\alpha} \quad &\text{ for }d=2, 
\end{align}
for $\eps$ sufficiently small and a suitable constant $c\in\VR$.
\item[(ii)] For $d\in \{2,3\}$, there holds $\displaystyle\lim_{\eps\searrow 0}\|\eps^{-1}(\nabla(\PPi_\eps)-\nabla(\PPi))-\nabla(\Pii)\|_{L_2(\omega)^{d\times d}}=0$.
\end{itemize}
\end{theorem}

\begin{proof}
 ad (i):For the sake of clarity, we restrict ourselves to the case of $d=3$. Dimension $d=2$ can be treated in a similar fashion.
In view of the auxiliary result Lemma \ref{lma:aux}, we seek a governing equation for $\eps\Piii_\eps=\Pii_\eps-\Pii-\eps^{d-2}\wii\circ \Phi_\eps$. Such an equation can be found using similar techniques to the analysis of the direct state. Thus, we refer to the appendix (Section \ref{sec:ad_P2}) for more details regarding the exact computation and only mention that there are functionals $G^2_\eps$, $G^3_\eps$, such that

\begin{align}\label{eq:P2_rhs}
    \int_{\Dsf_\eps} \VC_{\omega}\epsb(\varphi):\epsb(\eps\Piii_\eps)\; dx=G^2_\eps(\varphi)+G^3_\eps(\varphi) \quad\text{ for all } \varphi\in H^1_{\Gamma_\eps}(\Dsf_\eps)^d,
\end{align}
where $G^k_\eps$, $k=2,3$, satisfy
\begin{equation}\label{eq:P2_F_est}
\|G^k_\eps\|\le C\eps,
\end{equation}
for $k\in\{2,3\}$ and a constant $C>0$. Since \[\eps\Piii_\eps|_{\Gamma_\eps}=(\eps^{d-2}\Si(\eps x) -\eps^{-1}\PPi)|_{\Gamma_\eps}+(\eps^{d-2}\Sii(\eps x)-\Pii)|_{\Gamma_\eps},\] it follows by Lemma \ref{lma:remainder_est} item(i), (ii) with $m=d-1$ and $m=d$ respectively that

\begin{equation}\label{eq:P2_g_est}
\eps^{\frac{1}{2}}\|\eps\Piii_\eps\|_{L_2(\Gamma_\eps)^d}+|\eps\Piii_\eps|_{H^{\frac{1}{2}}(\Gamma_\eps)^d}\le C\eps^{\frac{d}{2}}.
\end{equation}

Hence, considering \eqref{eq:P2_F_est} and \eqref{eq:P2_g_est}, Lemma \ref{lma:aux} shows \eqref{eq:P2_diff_est}.\newline
ad (ii): By the triangle inequality we have
\begin{align}
\begin{split}
    \|\eps^{-1}(\nabla(\PPi_\eps)-\nabla(\PPi))& -\nabla(\Pii)\|_{L_2(\omega)^{d\times d}} \\
    \le & \|\nabla(\Pii_\eps-\Pii-\eps^{d-2}\wii\circ \Phi_\eps)\|_{L_2(\omega)^{d\times d}}\\
&+\eps^{d-1}\|\nabla(\wi)\circ \Phi_\eps\|_{L_2(\omega)^{d\times d}}+\eps^{d-1}\|\nabla(\wii)\circ \Phi_\eps\|_{L_2(\omega)^{d\times d}}\\
    \le &  C\eps^{1-\alpha},
\end{split}
\end{align}
for a positive constant $C$. This shows (ii) and therefore finishes the proof.
\end{proof}


\subsection{Averaged adjoint equation}\label{sec:asymptotic_av}

The averaged adjoint state $q_\eps$ satisfies
\begin{equation}
  q_\eps\in H^1_\Gamma(\Dsf)^d,\quad  a_\eps(\varphi,q_\eps) = - \int_0^1\partial J_\eps(su_\eps + (1-s)u_0)(\varphi)\; ds \quad \text{ for all } \varphi \in H^1_\Gamma(\Dsf)^d.
\end{equation}
With the cost function defined in \eqref{eq:cost_func} this equation reads explicitly
\begin{align}\label{eq:av_adj_per}
\begin{split}
    \int_\Dsf \VC_{\omega_\eps} \epsb(\varphi):\epsb(q_\eps) \;dx= & -\gamma_m\int_{\Gamma^m}(u_0+u_\eps-2 u_m)\cdot\varphi\;dS-\gamma_f\int_\Dsf f_{\omega_\eps}\cdot\varphi\;dx \\
                                                                   & -\gamma_g\int_{\Dsf}[\nabla u_0+\nabla u_\eps-2 \nabla u_d]:\nabla \varphi\;dx,
\end{split}
\end{align}
for all $\varphi\in H^1_\Gamma(\Dsf)^d$. Similarly, the unperturbed adjoint equation reads
\begin{align}\label{eq:av_adj_unper}
\begin{split}
    \int_\Dsf \VC_{2} \epsb(\varphi):\epsb(q_0) \;dx=& -2\gamma_m\int_{\Gamma^m}(u_0- u_m)\cdot\varphi\;dS-\gamma_f\int_\Dsf f_{2}\cdot\varphi\;dx \\
                                                     &-2\gamma_g\int_{\Dsf}[\nabla u_0-\nabla u_d]:\nabla \varphi\;dx,
\end{split}
\end{align}
for all $\varphi\in H^1_\Gamma(\Dsf)^d$. Considering \eqref{eq:adj_unper}, we like to point out that $p_0$ and $q_0$ satisfy the same equation and due to unique solvability it follows $p_0=q_0$.  Note that for the sake of simplicity we have chosen $\gamma_g=\gamma_m=0$ in $d=2$, as these terms lead to a more complicated analysis of the asymptotic expansion of $q_\eps$.

We now introduce the first terms of the asymptotic expansion:

\begin{lemma}\label{lma:Q1_rep}
\begin{itemize}
\item There exists a solution $[Q]\in\dot{BL}(\VR^d)^d$ to
\begin{align}\label{eq:first_Qh}
\begin{split}
    \int_{\VR^d}\VC_\omega \epsb(\varphi):\epsb([Q])\;dx&=\int_\omega(\VC_2-\VC_1)\epsb(\varphi):\epsb(q_0)(x_0)\;dx,
\end{split}
\end{align}
for all $\varphi\in \dot{BL}(\VR^d)^d$.
Moreover, there exists a representative $\hat{Q}^{(1)} \in [Q]$, which satisfies pointwise for $|x|\rightarrow \infty$:
\begin{equation}
\hat{Q}^{(1)}(x)=\hat{T}^{(1)}(x)+\mathcal{O}(|x|^{-d}),
\end{equation}
where $\hat{T}^{(1)}:\VR^d\rightarrow \VR^d$ satisfies
\begin{equation}
|\hat{T}^{(1)}(x)|=\begin{cases}
    \hat{b}_2 |x|^{-1} \quad &\text{ for }d=2,\\
\hat{b}_3 |x|^{-2} \quad &\text{ for }d=3,
\end{cases}
\end{equation}
for some constants $\hat{b}_2,\hat{b}_3\in \VR$.

\item There exists a solution $[Q]\in\dot{BL}(\VR^3)^3$ to
\begin{align}\label{eq:first_Qt}
\begin{split}
    \int_{\VR^d}\VC_\omega \epsb(\varphi):\epsb([Q])\;dx&=-\gamma_g\int_{\VR^d}\nabla\Ui:\nabla\varphi\;dx,
\end{split}
\end{align}
for all $\varphi\in \dot{BL}(\VR^3)^3$.
Moreover, there exists a representative $\tilde{Q}^{(1)} \in [Q]$, which satisfies pointwise for $|x|\rightarrow \infty$:
\begin{equation}
\tilde{Q}^{(1)}(x)=\tilde{T}^{(1)}(x)+\mathcal{O}(\ln(|x|)|x|^{-2}),
\end{equation}
where $\tilde{T}^{(1)}:\VR^d\rightarrow \VR^d$ satisfies
\begin{equation}
|\tilde{T}^{(1)}(x)|=\tilde{b}_3 |x|^{-1}
\end{equation}
for a constant $\tilde{b}_3\in \VR$.

\end{itemize}

Now let \[\Qi:=\begin{cases}\hat{Q}^{(1)}+\tilde{Q}^{(1)}\quad&\text{ for }d=3,\\\hat{Q}^{(1)}\quad&\text{ for }d=2, \end{cases}\]
and similarly \[\Ti:=\begin{cases}\hat{T}^{(1)}+\tilde{T}^{(1)}\quad&\text{ for }d=3,\\\hat{T}^{(1)}\quad&\text{ for }d=2. \end{cases}\]

\end{lemma}

\begin{proof}
Similar to Lemma \ref{lma:P1_rep}, where this time, due to the inhomogeneity in the exterior domain, we use a Newton potential to represent the solution.
\end{proof}

\begin{theorem}\label{thm:main_Q1}
For $\alpha \in (0,1)$ and $\eps>0$ sufficiently small there is a constant $C>0$, such that
\begin{align}\label{eq:exp1_est_av_adj}
\|Q_\eps^{(1)}-Q^{(1)}\|_\eps\le
\begin{cases}
C\eps^{\frac{1}{2}} \quad &\text{ for }d=3,\\
C\eps^{1-\alpha} \quad &\text{ for }d=2.
\end{cases}
\end{align}
\end{theorem}

\begin{proof}
We only show the case of $d=3$, since the proof for $d=2$ follows the same lines.
Again, we start by deriving an equation of the form

\begin{align}\label{eq:Q1_rhs}
\int_{\Dsf_\eps} \VC_{\omega}\epsb(\varphi):\epsb(\Qi_\eps-\Qi)\; dx=G^4_\eps(\varphi),
\end{align}
for all $\varphi\in H^1_{\Gamma_\eps}(\Dsf_\eps)^d$, where 

\begin{equation}\label{eq:F_est_av_adj}
\|G^4_\eps\|\le C\eps^{\frac{1}{2}}.
\end{equation}
A detailed derivation and estimation of the functional $G^4_\eps$ can be found in the appendix (Section \ref{sec:ad_Q1}).
In view of Lemma \ref{lma:aux} we now estimate the boundary integral terms. Since $\Qi_\eps|_{\Gamma_\eps}=0$ we deduce from Lemma \ref{lma:remainder_est} item (i), (ii) with $m=d-1$ that there is a constant $C>0$ satisfying

\begin{equation}\label{eq:g_est_av_adj}
\eps^{\frac{1}{2}}\|\Qi_\eps-\Qi\|_{L_2(\Gamma_\eps)^d}+|\Qi_\eps-\Qi|_{H^{\frac{1}{2}}(\Gamma_\eps)^d}\le C\eps^{\frac{1}{2}}.
\end{equation}
Thus, considering \eqref{eq:F_est_av_adj} and \eqref{eq:g_est_av_adj}, an application of Lemma \ref{lma:aux} shows \eqref{eq:exp1_est_av_adj}.
\end{proof}
We now continue with the second order expansion. Similar to the previous asymptotic expansions, we therefore introduce each component in the following lemma. Note that this time the regular correctors aim to approximate $\Ui$ in addition to their approximation of the occurring boundary layer correctors $\Qi,\Qii$. This is a result of the appearance of $\Ui_\eps$ on the right hand side of \eqref{eq:Q1_rhs}, \eqref{eq:G5}.
\begin{lemma}\label{lma:Q2_def}
\begin{itemize}

\item There is a unique solution $\mi\in H^1(\Dsf)^d$ with \newline$\mi(x)=-\Ti(x-x_0)$ on $\Gamma$, such that
\begin{align}\label{eq:mi}
\begin{split}
\int_\Dsf \VC_2 \epsb(\varphi):\epsb(\mi)\;dx&=-\int_{\Gamma^N}\VC_2^\top \epsb(\Ti)(x-x_0)n \cdot\varphi\; dS,
\end{split}
\end{align}
for all $\varphi \in H^1_\Gamma(\Dsf)^d$.

\item There is a solution $[Q]\in\dot{BL}_p(\VR^d)^d$ to
\begin{align}\label{eq:Qiih}
\begin{split}
    \int_{\VR^d}\VC_\omega\epsb(\varphi):\epsb([Q])\; dx=&\int_\omega (\VC_2-\VC_1)\epsb(\varphi):[\nabla\epsb(q_0)(x_0)x]\; dx\\
&-\gamma_g\int_{\VR^d}\nabla\Uii:\nabla\varphi\; dx,
\end{split}
\end{align}
for all $\varphi \in \dot{BL}_{p^\prime}(\VR^d)^d$, where 
\[
    p=\begin{cases}
    2+\delta\quad &\text{ for }d=2,\\ 
2\quad &\text{ for }d=3,
\end{cases}
\] 
and $\delta>0$ small. Moreover, there exists a representative $\hat{Q}^{(2)} \in [Q]$, which satisfies pointwise for $|x|\rightarrow \infty$:
\begin{equation}
\hat{Q}^{(2)}(x)=\hat{T}^{(2)}(x)+\begin{cases}
 \mathcal{O}(|x|^{-1}) \quad &\text{ for }d=2,\\
 \mathcal{O}(\ln(|x|)|x|^{-2}) \quad &\text{ for }d=3,
\end{cases}
\end{equation}
where $\hat{T}^{(2)}:\VR^d\rightarrow \VR^d$ satisfies
\begin{equation}
|\hat{T}^{(2)}(x)|=\begin{cases}
 \hat{c}_2 \ln(|x|) \quad &\text{ for }d=2,\\
\hat{c}_3 \ln(|x|)|x|^{-1} \quad &\text{ for }d=3,
\end{cases}
\end{equation}
for some constants $\hat{c}_2,\hat{c}_3\in \VR$.

\item There is a solution $[Q]\in\dot{BL}_p(\VR^d)^d$ to
\begin{equation}\label{eq:Qiit}
\int_{\VR^d}\VC_\omega\epsb(\varphi):\epsb([Q])\; dx=\gamma_f\int_\omega [f_2(x_0)-f_1(x_0)]\cdot\varphi\; dx,
\end{equation}
for all $\varphi \in C^1_c(\VR^d)^d$, where 
\[
    p=\begin{cases}
    2+\delta\quad &\text{ for }d=2,\\ 
2\quad &\text{ for }d=3,
\end{cases}
\] and $\delta>0$ small. Moreover, there exists a representative $\tilde{Q}^{(2)} \in [Q]$, which satisfies pointwise for $|x|\rightarrow \infty$:
\begin{equation}
\tilde{P}^{(2)}(x)=\tilde{T}^{(2)}(x)+\mathcal{O}(|x|^{1-d}),
\end{equation}
where $\tilde{T}^{(2)}:\VR^d\rightarrow \VR^d$ satisfies
\begin{equation}
|\tilde{T}^{(2)}(x)|=\begin{cases}
 \tilde{c}_2 \ln(|x|) \quad &\text{ for }d=2,\\
\tilde{c}_3 |x|^{-1} \quad &\text{ for }d=3,
\end{cases}
\end{equation}
for some constants $\tilde{c}_2,\tilde{c}_3\in \VR$.

\item There is a unique solution $\mii_U\in H_\Gamma^1(\Dsf)^d$, such that
\begin{align}\label{eq:miiU}
\begin{split}
    \int_\Dsf \VC_2 \epsb(\varphi):\epsb(\mii_U)\; dx= & -\gamma_m\int_{\Gamma^m} \Ri(x-x_0)\cdot \varphi\;dS-\gamma_m\int_{\Gamma^m} \vi\cdot\varphi\;dS\\
&-\gamma_m\int_{\Gamma^m}\Rii(x-x_0)\cdot\;dS-\gamma_m\int_{\Gamma^m}\vii\cdot\varphi\;dS\\
&-\gamma_g\int_\Dsf\nabla \vi:\nabla\varphi\; dx-\gamma_g\int_\Dsf\nabla \vii:\nabla\varphi\; dx\\
& -\gamma_g\int_{\Gamma^N}\nabla \Ri (x-x_0) n\cdot\varphi\; dS \\
& -\gamma_g\int_{\Gamma^N}\nabla \Rii (x-x_0) n\cdot\varphi\; dS
\end{split}
\end{align}
for all $\varphi \in H^1_\Gamma(\Dsf)^d$.
\item There is a unique solution $\mii_Q\in H^1(\Dsf)^d$ with $\mii(x)=-\Tii(x-x_0)$ on $\Gamma$, such that
\begin{align}\label{eq:miiQ}
\begin{split}
\int_\Dsf \VC_2 \epsb(\varphi):\epsb(\mii_Q)\; dx&=-\int_{\Gamma^N}\VC_2^\top \epsb(\Tii)(x-x_0)n\cdot \varphi\; dS
\end{split}
\end{align}
for all $\varphi \in H^1_\Gamma(\Dsf)^d$, where 
\[
    \Tii:=\begin{cases}
 \hat{T}^{(2)}+\tilde{T}^{(2)} \quad &\text{ for }d=2,\\
\hat{T}^{(2)} \quad &\text{ for }d=3.
\end{cases}
\]
Furthermore, we define $\mii:=\mii_U+\mii_Q$.
\end{itemize}
\end{lemma}

\begin{proof}
Similar to Lemma~\ref{lma:U2_def} and Lemma~\ref{lma:Q1_rep}.
\end{proof}

Now we are able to state our main result regarding the second order expansion of the averaged adjoint state variable $q_\eps$:

\begin{theorem}\label{thm:main_Q2}
\begin{itemize}
\item[(i)] Let $\alpha\in(0,1)$. There exists a constant $C>0$, such that for $d=3$ and $d=2$, we have respectively:
\begin{align}\label{eq:Q2_diff_est}
\|\eps^{-1}[\Qi_\eps-\Qi-\eps^{d-2}\mi\circ \Phi_\eps]-\Qii-\eps^{d-2}\mii \circ \Phi_\eps\|_\eps&\le C\eps^{\frac{1}{2}}\ln(\eps^{-1})\\
\|\eps^{-1}[\Qi_\eps-\Qi-\eps^{d-1}\mi\circ \Phi_\eps]-\Qii-\eps^{d-2}\mii\circ \Phi_\eps+c\ln(\eps)\|_\eps&\le C\eps^{1-\alpha} 
\end{align}
for $\eps$ sufficiently small and a suitable constant $c\in \VR$.
\item[(ii)] For $d\in \{2,3\}$, there holds $\displaystyle\lim_{\eps\searrow 0}\|\eps^{-1}(\nabla\Qi_\eps-\nabla\Qi)-\nabla\Qii\|_{L_2(\omega)^{d \times d}}=0$.
\end{itemize}
\end{theorem}

\begin{proof}
ad (i): Similar to the proof of the second expansion of the adjoint state variable, we restrict ourselves to the case of $d=3$. The proof for dimension $d=2$ follows the same lines and is therefore omitted.
In view of the auxiliary result Lemma~\ref{lma:aux}, we seek a governing equation for $V_\eps:=\eps^{-1}[\Qi_\eps-\Qi]-\eps^{d-3}\mi\circ \Phi_\eps-\Qii-\eps^{d-2}\mii \circ \Phi_\eps$. Such an equation can be found using similar techniques to the analysis of the direct state and the derivation will be discussed in detail in the appendix (see Section \ref{sec:ad_Q2}). We just note that there are functionals $G^5_\eps$, $G^6_\eps$, such that

\begin{align}\label{eq:Q2_rhs}
    \int_{\Dsf_\eps} \VC_{\omega}\epsb(\varphi):\epsb(V_\eps)\; dx=G^5_\eps(\varphi)+G^6_\eps(\varphi) \quad\text{ for all } \varphi\in H^1_{\Gamma_\eps}(\Dsf_\eps)^d,
\end{align}
where $G^k_\eps$, $k=5,6$, satisfy
\begin{equation}\label{eq:Q2_F_est}
\|G^k_\eps\|\le C\eps^{\frac{1}{2}}\ln(\eps^{-1}),
\end{equation}
for $k\in\{5,6\}$ and a constant $C>0$. Since \[V_\eps|_{\Gamma_\eps}=(\eps^{d-2}\Ti(\eps x) -\eps^{-1}\Qi)|_{\Gamma_\eps}-\Qii|_{\Gamma_\eps},\] it follows with a similar argument to Lemma \ref{lma:remainder_est} that there is a constant $C>0$ satisfying
\begin{equation}\label{eq:Q2_g_est}
\eps^{\frac{1}{2}}\|V_\eps\|_{L_2(\Gamma_\eps)^d}+|V_\eps|_{H^{\frac{1}{2}}(\Gamma_\eps)^d}\le C\eps^{\frac{1}{2}}\ln(\eps^{-1}).
\end{equation}
In view of \eqref{eq:Q2_F_est} and \eqref{eq:Q2_g_est}, Lemma~\ref{lma:aux} shows \eqref{eq:Q2_diff_est}.\newline
ad (ii): Let $d=3$. By the triangle inequality we have
\begin{align}
\begin{split}
    \|\eps^{-1}(\nabla(\Qi_\eps)-\nabla(\Qi))& -\nabla(\Qii)\|_{L_2(\omega)^{d\times d}} \\
                                             & \le \|\eps^{-1}\nabla(\Qi_\eps-\Qi)-\nabla(\eps^{d-3}\mi\circ \Phi_\eps)-\nabla(\eps^{d-2}\mii\circ \Phi_\eps)\|_{L_2(\omega)^{d\times d}}\\
&+\eps^{d-2}\|\nabla(\mi)\circ \Phi_\eps\|_{L_2(\omega)^{d\times d}}+\eps^{d-1}\|\nabla(\mii)\circ \Phi_\eps\|_{L_2(\omega)^{d\times d}}\\
&\le C\eps^{\frac{1}{2}}\ln(\eps^{-1}),
\end{split}
\end{align}
for a positive constant $C$, which shows (ii). The proof for $d=2$ follows the same lines and is therefore omitted.

\end{proof}


\section{Computation of the topological derivatives for linear elasticity problem}\label{sec:topo_deriv}

In this section we compute the first and second topological derivative of our elasticity problem introduced in the introduction, namely, 
\begin{equation}
    \mathcal J(\Omega) = J(\Omega,u) =\gamma_m \int_{\Gamma^m} |u - u_m|^2\;dS + \gamma_f\int_\Dsf f_{\omega_\eps}\cdot u\;dx+\gamma_g\int_\Dsf|\nabla u-\nabla u_d|^2\;dx,
\end{equation}
subject to $u\in H^1(\Dsf)^d$ solves $u|_\Gamma=u_D$ and 
\begin{equation}
    \int_\Dsf \VC_\Omega \epsb(u):\epsb(\varphi)= \int_\Dsf f_\Omega \cdot\varphi\;dx +\int_{\Gamma^N}u_N\cdot\varphi\; dS\quad \text{ for all } \varphi\in H^1_{\Gamma}(\Dsf)^d.
\end{equation}

\begin{definition}
For $\eps\ge 0$ let $\Omega_\eps\subset \Dsf$ be a singularly perturbed domain with perturbation shape $\omega$ and $\Omega:=\Omega_0$. Additionally let $\li,\lii :\VR^+\rightarrow \VR^+$ be two functions converging to $0$ for $\eps\searrow 0$ and $\frac{\li(\eps)}{\lii(\eps)}\rightarrow 0$ for $\eps\searrow 0$. Then the first order topological derivative is defined as
\[d\mathcal{J}(\Omega,\omega)=\lim_{\eps\searrow 0}\frac{\mathcal{J}(\Omega_\eps)-\mathcal{J}(\Omega)}{\li(\eps)}.\]
Similarly, the second order topological derivative is given as
\[d^2\mathcal{J}(\Omega,\omega)=\lim_{\eps\searrow 0}\frac{\mathcal{J}(\Omega_\eps)-\mathcal{J}(\Omega)-\li(\eps)d\mathcal{J}(\Omega)}{\lii(\eps)}.\]
\end{definition}
More specifically, since we considered $\Omega=\emptyset$, we compute the topological derivative at a point $x_0\in \Dsf$ and derive 
the asymptotics of $\mathcal J(\omega_\eps)$ with $\omega_\eps := \Phi_\eps(\omega)$, $\Phi_\eps(x):= x_0+\eps x$. Recall the Lagrangian function
\[ 
    \Cl(\eps,u,v)= J_\eps(u) + a_\eps(u,v) - f_\eps(v), \quad u \in  \Cv, \; v\in \Cw.
\]
with $\Cv = \{u \in H^1({\Dsf})^d| u =u_D \text{ on }\Gamma\}$, $\Cw = H^1_\Gamma({\Dsf})^d$ and 
\begin{align}
    J_\eps(u)   &= \gamma_m \int_{\Gamma^m} |u - u_m|^2\;dS + \gamma_f\int_\Dsf f_{\omega_\eps}\cdot u\;dx+\gamma_g\int_\Dsf|\nabla u -\nabla u_d|^2\;dx,\\
    a_\eps(u,v) &= \int_{\Dsf} \VC_{\omega_\eps} \epsb(u) :  \epsb(v)\;dx,\\
    f_\eps(v)   &= \int_{\Dsf} f_{\omega_\eps} \cdot v\;dx.
\end{align}
We compute the derivative using Proposition~\ref{prop:main_am} (Amstutz' method), Proposition~\ref{prop:main_av} (averaged adjoint) and Proposition~\ref{prop:main_df} (Delfour's method).

\begin{remark}
We would like to point out that, contrary to the setting in Section \ref{sec:abstract_setting}, $\Cv$ is a affine space spanned by $\Cw$ and not a vector space itself. Yet, it can easily be verified that the Lagrangian techniques can still be applied, as the construction of $\Cv$ allows derivatives of $\Cl(\eps,u,v)$ with respect to the second variable in direction $\Cw$.
\end{remark}
\begin{remark}
Note that the more general case $\Omega\neq\emptyset$, $x_0\in\Dsf\setminus\bar{\Omega}$ and $\Omega_\eps=\Omega\cup\omega_\eps$ can be treated in a similar fashion. The main difference is that the unperturbed state equation and unperturbed adjoint state equation respectively depend on $\Omega$ and therefore $u_0$ and $p_0$ vary. Furthermore, as the boundary layer correctors coincide in both cases, the regular correctors need to compensate in $\Omega$.\newline
At last, the computation of the topological derivative for $x_0\in\Omega$ and $\Omega_\eps=\Omega\setminus\omega_\eps$  can be done analogously to the presented one and only results in a change of sign.
\end{remark}

\subsection{Amstutz' method}\label{sec:deriv_am}

In order to compute the first order topological derivative let $\li(\eps):=|\omega_\eps|$. By Proposition~\ref{prop:main_am}, item (i), we have

\begin{equation}
d\mathcal{J}(\emptyset,\omega)(x_0)=\mRi(u_0,p_0)+\pimL(0,u_0,p_0),
\end{equation}
where
\begin{align}
    \mRi(u_0,p_0) & =\lim_{\eps\searrow 0}\frac{\Cl(\eps,u_\eps,p_\eps)-\Cl(\eps,u_0,p_\eps)}{\li(\eps)},\\
    \pimL(0,u_0,p_0)& =\lim_{\eps\searrow 0}\frac{\Cl(\eps,u_0,p_\eps)-\Cl(0,u_0,p_\eps)}{\li(\eps)},
\end{align}
if the above limits exist. Thus, we start with the first quotient $\mRi(u_0,p_0)$:

\begin{align}\label{eq:r1_am}
\begin{split}
\frac{\Cl(\eps,u_\eps,p_\eps)-\Cl(\eps,u_0,p_\eps)}{\li(\eps)}=&\frac{1}{|\omega_\eps|}[J_\eps(u_\eps)+a_\eps(u_\eps,p_\eps)-f_\eps(p_\eps)-J_\eps(u_0)-a_\eps(u_0,p_\eps)+f_\eps(p_\eps)]\\
=\frac{1}{|\omega_\eps|}&[J_\eps(u_\eps)-J_\eps(u_0)+a_\eps(u_\eps-u_0,p_\eps)]\\
=\frac{1}{|\omega_\eps|}\gamma_m&\int_{\Gamma^m}[|u_\eps-u_m|^2-|u_0-u_m|^2-2(u_0-u_m)(u_\eps-u_0)]\;dS\\
+\frac{1}{|\omega_\eps|}\gamma_g&\int_{\Dsf}\left[|\nabla u_\eps-\nabla u_d|^2-|\nabla u_0-\nabla u_d|^2-2(\nabla u_0-\nabla u_d)(\nabla u_\eps-\nabla u_0)\right]\;dx\\
=\frac{1}{|\omega_\eps|}\gamma_m&\int_{\Gamma^m}|u_\eps-u_0|^2\;dS+\frac{1}{|\omega_\eps|}\gamma_g\int_{\Dsf}|\nabla u_\eps-\nabla u_0|^2\;dx.
\end{split}
\end{align}
Now, a change of variables leads to $\frac{\eps}{|\omega|}\gamma_m\|\Ui_\eps\|_{L_2(\Gamma_\eps^m)^d}^2+\frac{1}{|\omega|}\gamma_g\|\nabla \Ui_\eps\|_{L_2(\Dsf_\eps)^{d\times d}}^2$.
On the one hand, we have
\begin{equation}
    \begin{split}
        \frac{\eps}{|\omega|}\gamma_m\|\Ui_\eps\|_{L_2(\Gamma_\eps^m)^d}^2 & \le\frac{\gamma_m}{|\omega|}(\eps\|\Ui_\eps-\Ui\|_{L_2(\Gamma_\eps^m)^d}^2+\eps\|\Ui\|_{L_2(\Gamma_\eps^m)^d}^2) \\
                                                                           & \le C (\|\Ui_\eps-\Ui\|_\eps^2+\eps^d) \le C \eps^{2-\alpha},
\end{split}
\end{equation}
for $\alpha$ arbitrarily small and a constant $C>0$. Here we used Lemma~\ref{lma:scaling2}, item (d), Lemma \ref{lma:remainder_est} item(i) with $m=d-1$ and Theorem~\ref{thm:main_U1}. \newline
On the other hand, Theorem \ref{thm:main_U1} shows that $\nabla\Ui_\eps \rightarrow \nabla\Ui$ in $L_2(\VR^d)^{d\times d}$ for $\eps \searrow 0$.
Now passing to the limit in \eqref{eq:r1_am} yields \[\mRi(u_0,p_0)=\frac{1}{|\omega|}\gamma_g\int_{\VR^d}|\nabla \Ui|^2\;dx.\]
Next, we consider $\pimL(0,u_0,p_0)$. Splitting the quotient, one observes
\begin{align}\label{eq:l1_am}
\begin{split}
\frac{\Cl(\eps,u_0,p_\eps)-\Cl(0,u_0,p_\eps)}{\li(\eps)}=&\frac{1}{|\omega_\eps|}[J_\eps(u_0)+a_\eps(u_0,p_\eps)-f_\eps(p_\eps)-J_0(u_0)-a_0(u_0,p_\eps)+f_0(p_\eps)]\\
=&\frac{1}{|\omega_\eps|}\int_{\omega_\eps}\left[\gamma_f(f_1-f_2)u_0+(\VC_1-\VC_2)\epsb(u_0):\epsb(p_\eps)-(f_1-f_2)p_\eps\right]\;dx\\
=&\gamma_f\dashint_{\omega}(f_1\circ\Phi_\eps-f_2\circ\Phi_\eps)\cdot u_0\circ\Phi_\eps\;dx\\
&+\dashint_{\omega}(\VC_1-\VC_2)\epsb(u_0)\circ\Phi_\eps:\epsb(\PPi_\eps)\;dx\\
&+\dashint_{\omega}(\VC_1-\VC_2)\epsb(u_0)\circ\Phi_\eps:\epsb(p_0)\circ\Phi_\eps\;dx\\
&-\dashint_{\omega}(f_1\circ\Phi_\eps-f_2\circ\Phi_\eps)\cdot(\PPi_\eps)\;dx\\
&-\dashint_{\omega}(f_1\circ\Phi_\eps-f_2\circ\Phi_\eps)\cdot p_0\circ\Phi_\eps\;dx.
\end{split}
\end{align}
By Hölder's inequality, Lemma~\ref{lma:scaling2}, item (b), (c) and Theorem \ref{thm:main_Q1} one readily checks that $\PPi_\eps\rightarrow\PPi$ in $L_1(\omega)^d$ and $\epsb(\Qi_\eps)\rightarrow \epsb(\Qi)$ in $L_2(\omega)^{d\times d}$. Hence, we deduce

\begin{align}
\begin{split}
\pimL(0,u_0,p_0)=&\dashint_{\omega}(\VC_1-\VC_2)\epsb(u_0)(x_0):\epsb(\PPi)\;dx\\
&+\gamma_f(f_1(x_0)-f_2(x_0))\cdot u_0(x_0)\\
&+(\VC_1-\VC_2)\epsb(u_0)(x_0):\epsb(p_0)(x_0)\\
&-(f_1(x_0)-f_2(x_0))\cdot p_0(x_0).
\end{split}
\end{align}
Therefore, the first order topological derivative is given by
\begin{align}\label{eq:first_der_am}
\begin{split}
d\mathcal{J}(\emptyset,\omega)(x_0)=&\dashint_{\omega}(\VC_1-\VC_2)\epsb(u_0)(x_0):\epsb(\PPi)\;dx+(\VC_1-\VC_2)\epsb(u_0)(x_0):\epsb(p_0)(x_0)\\
&+\gamma_f(f_1(x_0)-f_2(x_0))\cdot u_0(x_0)-(f_1(x_0)-f_2(x_0))\cdot p_0(x_0)\\
&+\frac{1}{|\omega|}\gamma_g\int_{\VR^d}|\nabla \Ui|^2\;dx,
\end{split}
\end{align}
with $\PPi$ defined in \eqref{eq:first_P} and $\Ui$ defined in \eqref{eq:first_U}.
Next, we compute the second order topological derivative. Therefore let $\lii(\eps):=\eps\li(\eps)$. By Proposition~\ref{prop:main_am}, item (ii), we have

\begin{equation}
d^2\mathcal{J}(\emptyset,\omega)(x_0)=\mRii(u_0,p_0)+\piimL(0,u_0,p_0),
\end{equation}
where
\begin{align*}
    \mRii(u_0,p_0) & =\lim_{\eps\searrow 0}\frac{\Cl(\eps,u_\eps,p_\eps)-\Cl(\eps,u_0,p_\eps)-\li(\eps)\mRi(u_0,p_0)}{\lii(\eps)}, \\
    \piimL(0,u_0,p_0)& =\lim_{\eps\searrow 0}\frac{\Cl(\eps,u_0,p_\eps)-\Cl(0,u_0,p_\eps)-\li(\eps)\piimL(0,u_0,p_0)}{\lii(\eps)},
\end{align*}
if the above limits exist. Dividing \eqref{eq:r1_am} by $\eps$, it follows that 
\begin{align}\label{eq:r2_am}
\begin{split}
\mRii(u_0,p_0)&=\lim_{\eps \searrow 0}\frac{\gamma_g}{\eps}\left[\int_{\Dsf_\eps}|\nabla \Ui_\eps|^2-|\nabla\Ui|^2\;dx-\int_{\VR^d\setminus\Dsf_\eps}|\nabla\Ui|^2\;dx\right]\\
&=\lim_{\eps \searrow 0}\gamma_g\left[\int_{\Dsf_\eps}(\Ui_\eps+\Ui):(\eps^{-1}[\Ui_\eps-\Ui])\;dx-\eps^{-1}\int_{\VR^d\setminus\Dsf_\eps}|\nabla\Ui|^2\;dx\right]\\
&=2\gamma_g\frac{1}{|\omega|}\int_{\VR^d}\nabla\Ui:\nabla\Uii\;dx,
\end{split}
\end{align}
where we used that $\nabla\Ui_\eps\rightarrow\nabla\Ui$ in $L_2(\VR^d)^{d\times d}$ (see Theorem \ref{thm:main_U1}) and $\eps^{-1}(\nabla\Ui_\eps-\nabla\Ui)\rightarrow\nabla\Uii$ in $L_2(\VR^d)^{d\times d}$ (see Remark \ref{rem:U2}). The integral term over the exterior domain vanishes due to the asymptotic behaviour of $\Ui$.

In order to compute $\piimL(0,u_0,p_0)$, we use \eqref{eq:l1_am} to get

\begin{align}\label{eq:l2_am}
\begin{split}
&\frac{\Cl(\eps,u_0,p_\eps)-\Cl(0,u_0,p_\eps)-\li(\eps)\pimL(0,u_0,p_0)}{\lii(\eps)}=\\
&\gamma_f\dashint_{\omega}\eps^{-1}[(f_1\circ\Phi_\eps-f_2\circ\Phi_\eps)-(f_1(x_0)-f_2(x_0))]\cdot u_0\circ\Phi_\eps\;dx\\
&+\gamma_f\dashint_{\omega}(f_1(x_0)-f_2(x_0))\cdot\eps^{-1}[u_0\circ\Phi_\eps-u_0(x_0)]\;dx\\
&+\dashint_{\omega}(\VC_1-\VC_2)\epsb(u_0)\circ\Phi_\eps:\epsb(\eps^{-1}[\PPi_\eps-\PPi])\;dx\\
&+\dashint_{\omega}(\VC_1-\VC_2)\eps^{-1}[\epsb(u_0)\circ\Phi_\eps-\epsb(u_0)(x_0)]:\epsb(\PPi)\;dx\\
&+\dashint_{\omega}(\VC_1-\VC_2)\epsb(u_0)\circ\Phi_\eps:\eps^{-1}[\epsb(p_0)\circ\Phi_\eps-\epsb(p_0)(x_0)]\;dx\\
&+\dashint_{\omega}(\VC_1-\VC_2)\eps^{-1}[\epsb(u_0)\circ\Phi_\eps-\epsb(u_0)(x_0)]:\epsb(p_0)(x_0)\;dx\\
&-\dashint_{\omega}(f_1\circ\Phi_\eps-f_2\circ\Phi_\eps)\cdot(\PPi_\eps)\;dx\\
&-\dashint_{\omega}(f_1\circ\Phi_\eps-f_2\circ\Phi_\eps)\cdot\eps^{-1}[p_0\circ\Phi_\eps-p_0(x_0)]\;dx\\
&-\dashint_{\omega}\eps^{-1}[(f_1\circ\Phi_\eps-f_2\circ\Phi_\eps)-(f_1(x_0)-f_2(x_0))]\cdot p_0(x_0)\;dx.
\end{split}
\end{align}
Now considering Theorem~\ref{thm:main_P2}, item (ii), we have $\eps^{-1}[\epsb(\PPi_\eps)-\epsb(\PPi)]\rightarrow \epsb(\Pii)$ in $L_2(\omega)^{d\times d}$ and by Theorem \ref{thm:main_P1} and the Gagliardo-Nirenberg inequality we get $\PPi_\eps\rightarrow \PPi$ in $L_1(\omega)^d$. Thus, passing to the limit $\eps\searrow 0$ in \eqref{eq:l2_am} we conclude
\begin{align}
\begin{split}
\piimL(0,u_0,p_0)=&\gamma_f\dashint_{\omega}\nabla[f_1(x_0)-f_2(x_0)]x\cdot u_0(x_0)\; dx+\gamma_f\dashint_{\omega}[f_1(x_0)-f_2(x_0)]\cdot\nabla u_0(x_0)x\; dx\\
&+\dashint_{\omega}(\VC_1-\VC_2)\epsb(u_0)(x_0):\epsb(\Pii)\; dx+\dashint_{\omega}(\VC_1-\VC_2)[\nabla\epsb(u_0)(x_0)x]:\epsb(\PPi)\; dx\\
&+\dashint_{\omega}(\VC_1-\VC_2)\epsb(u_0)(x_0):[\nabla(\epsb(p_0))(x_0)x]\; dx \\
& +\dashint_{\omega}(\VC_1-\VC_2)[\nabla(\epsb(u_0))(x_0)x]:\epsb(p_0)(x_0)\; dx\\
&-\dashint_{\omega}[f_1(x_0)-f_2(x_0)]\cdot\nabla p_0(x_0)x\; dx- \dashint_{\omega}\nabla[f_1(x_0)-f_2(x_0)]x\cdot p_0(x_0)\; dx\\
&-\dashint_{\omega}(f_1(x_0)-f_2(x_0))\cdot(\PPi)\;dx.
\end{split}
\end{align}
Thus, the second order topological derivative is given by
\begin{align}
\begin{split}
d^2\mathcal{J}(\emptyset,\omega)(x_0)=&\gamma_f\dashint_{\omega}\nabla[f_1(x_0)-f_2(x_0)]x\cdot u_0(x_0)\; dx+\gamma_f\dashint_{\omega}[f_1(x_0)-f_2(x_0)]\cdot\nabla u_0(x_0)x\; dx\\
&+\dashint_{\omega}(\VC_1-\VC_2)\epsb(u_0)(x_0):\epsb(\Pii)\; dx+\dashint_{\omega}(\VC_1-\VC_2)[\nabla\epsb(u_0)(x_0)x]:\epsb(\PPi)\; dx\\
&+\dashint_{\omega}(\VC_1-\VC_2)\epsb(u_0)(x_0):[\nabla(\epsb(p_0))(x_0)x]\; dx \\
& +\dashint_{\omega}(\VC_1-\VC_2)[\nabla(\epsb(u_0))(x_0)x]:\epsb(p_0)(x_0)\; dx\\
&-\dashint_{\omega}[f_1(x_0)-f_2(x_0)]\cdot\nabla p_0(x_0)x\; dx- \dashint_{\omega}\nabla[f_1(x_0)-f_2(x_0)] x\cdot p_0(x_0)\; dx\\
&-\dashint_{\omega}(f_1(x_0)-f_2(x_0))\cdot(\PPi)\;dx+2\gamma_g\frac{1}{|\omega|}\int_{\VR^d}\nabla\Ui:\nabla\Uii\;dx,
\end{split}
\end{align}
with $\PPi$ defined in \eqref{eq:first_P}, $\Ui$ defined in \eqref{eq:first_U}, $\Pii$ defined in \eqref{eq:Piih},\eqref{eq:Piit} and $\Uii$ defined in \eqref{eq:Uiih},\eqref{eq:Uiit}.


\subsection{Averaged adjoint method}\label{sec:deriv_av}

We start with the first order topological derivative. Therefore let $\li(\eps):=|\omega_\eps|$. By Proposition \ref{prop:main_av} item (i) we have

\begin{equation}
d\mathcal{J}(\emptyset,\omega)(x_0)=\mRi(u_0,q_0)+\pimL(0,u_0,q_0),
\end{equation}
where
\begin{align*}
    \mRi(u_0,q_0)& =\lim_{\eps\searrow 0}\frac{\Cl(\eps,u_0,q_\eps)-\Cl(\eps,u_0,q_0)}{\li(\eps)}, \\
    \pimL(0,u_0,q_0)& =\lim_{\eps\searrow 0}\frac{\Cl(\eps,u_0,q_0)-\Cl(0,u_0,q_0)}{\li(\eps)},
\end{align*}
if the above limits exist. Thus, we start computing $\mRi(u_0,q_0)$:

\begin{align}\label{eq:r1_av}
\begin{split}
\frac{\Cl(\eps,u_0,q_\eps)-\Cl(\eps,u_0,q_0)}{\li(\eps)}&=\frac{1}{|\omega_\eps|}[J_\eps(u_0)+a_\eps(u_0,q_\eps)-f_\eps(q_\eps)-J_\eps(u_0)-a_\eps(u_0,q_0)+f_\eps(q_0)]\\
=\frac{1}{|\omega_\eps|}&[a_\eps(u_0,q_\eps-q_0)-f_\eps(q_\eps-q_0)]\\
=\frac{1}{|\omega_\eps|}&\left[\int_{\omega_\eps}(\VC_1-\VC_2)\epsb(u_0):\epsb(q_\eps-q_0)\; dx-\int_{\omega_\eps}(f_1-f_2)\cdot(q_\eps-q_0)\;dx\right]\\
=\dashint_{\omega}&(\VC_1-\VC_2)\epsb(u_0)\circ\Phi_\eps:\epsb(\Qi_\eps)\; dx-\eps\dashint_{\omega}(f_1-f_2)\circ\Phi_\eps(\Qi_\eps)\;dx.
\end{split}
\end{align}
Since $u_0\in C^3(B_\delta(x_0))$ for $\delta>0$ small and by Theorem \ref{thm:main_Q1} $\epsb(\Qi_\eps)\rightarrow \epsb(\Qi)$ in $L_2(\omega)^{d\times d}$ as $\eps$ tends to zero, we have
\begin{equation}
\lim_{\eps\searrow 0}\int_{\omega}(\VC_1-\VC_2)\epsb(u_0)\circ\Phi_\eps:\epsb(\Qi_\eps)\; dx=\int_{\omega}(\VC_1-\VC_2)\epsb(u_0)(x_0):\epsb(\Qi)\; dx.
\end{equation}
Furthermore, applying Hölder's inequality, Lemma \ref{lma:scaling2} item (b), (c) and Theorem \ref{thm:main_Q1}, one readily checks that $\Qi_\eps\rightarrow\Qi$ in $L_1(\omega)^d$. Thus, we deduce 
\begin{equation}
    \lim_{\eps\searrow 0}\eps\int_{\omega}[(f_1-f_2)\circ\Phi_\eps]\cdot\Qi_\eps\;dx=0.
\end{equation}
It follows that $\mRi(u_0,q_0)=\dashint_{\omega}(\VC_1-\VC_2)\epsb(u_0)(x_0):\epsb(\Qi)\; dx$. Next we compute $\pimL(0,u_0,q_0)$. 
For this we note for $\eps >0$:

\begin{align}\label{eq:l1_av}
\begin{split}
\frac{\Cl(\eps,u_0,q_0)-\Cl(0,u_0,q_0)}{\li(\eps)}=&\frac{1}{|\omega_\eps|}[(J_\eps(u_0)-J_0(u_0))+(a_\eps(u_0,q_0)-a_0(u_0,q_0))-(f_\eps(q_0)-f_0(q_0))]\\
=&\frac{1}{|\omega_\eps|}\bigg[\gamma_f\int_{\omega_\eps}(f_1-f_2)\cdot u_0\;dx+\int_{\omega_\eps}(\VC_1-\VC_2)\epsb(u_0):\epsb(q_0)\; dx \\
 & \hspace{5cm}-\int_{\omega_\eps}(f_1-f_2)\cdot q_0\;dx\bigg]\\
=&\gamma_f\dashint_{\omega}(f_1-f_2)\circ \Phi_\eps \cdot u_0\circ \Phi_\eps\;dx -\dashint_{\omega}(f_1-f_2)\circ \Phi_\eps \cdot q_0\circ \Phi_\eps\;dx
\\
& +\dashint_{\omega}(\VC_1-\VC_2)\epsb(u_0)\circ \Phi_\eps:\epsb(q_0)\circ \Phi_\eps\; dx.
\end{split}
\end{align}
Now, since $u_0,q_0,f_1,f_2$ are smooth in a neighbourhood of $x_0$, we get
\begin{equation}
    \pimL(0,u_0,q_0)=\left[\gamma_f (f_1-f_2)\cdot u_0+(\VC_1-\VC_2)\epsb(u_0):\epsb(q_0)-(f_1-f_2)\cdot q_0\right](x_0).
\end{equation}
Hence, the first topological derivative is given by
\begin{align}\label{eq:first_der_av}
\begin{split}
d\mathcal{J}(\emptyset,\omega)(x_0)=&\dashint_{\omega}(\VC_1-\VC_2)\epsb(u_0)(x_0):\epsb(\Qi)\; dx+(\VC_1-\VC_2)\epsb(u_0)(x_0):\epsb(q_0)(x_0)\\
&+\gamma_f (f_1(x_0)-f_2(x_0))\cdot u_0(x_0)-(f_1(x_0)-f_2(x_0))\cdot q_0(x_0),
\end{split}
\end{align}
 with $\Qi$ defined in \eqref{eq:first_Qh},\eqref{eq:first_Qt}.
\begin{remark}
An elegant way to represent the topological derivative is by the use of a polarisation tensor (see \cite{b_NOSO_2013a}, \cite{a_AMKAKI_2005a}). For this note that the mappings
\[\mathcal{F}^1: \VR^{d\times d}\rightarrow \VR^{d \times d},\quad\zeta\mapsto\dashint_\omega \epsb(\hat{Q}^{(1)}_\zeta)\; dx,\]
\[\mathcal{F}^2: \VR^{d\times d}\rightarrow \VR^{d \times d},\quad \zeta\mapsto\dashint_\omega \epsb(\tilde{Q}^{(1)}_\zeta)\; dx,\]
are linear, where $\hat{Q}^{(1)}_\zeta$ solves
\[\int_{\VR^d}\VC_\omega \epsb(\varphi):\epsb(\hat{Q}^{(1)}_\zeta)\; dx=\int_\omega (\VC_2-\VC_1)\epsb(\varphi):\zeta\; dx\quad\text{ for all }\varphi\in \dot{BL}(\VR^d)^d,\]
and $\tilde{Q}^{(1)}_\zeta$ solves
\[\int_{\VR^d}\VC_\omega \epsb(\varphi):\epsb(\tilde{Q}^{(1)}_\zeta)\; dx=-\gamma_g\int_{\VR^d} \nabla \Ui_\zeta:\nabla\varphi\; dx\quad\text{ for all }\varphi\in \dot{BL}(\VR^d)^d,\]
with $\Ui_\zeta$ satisfying
\[\int_{\VR^d}\VC_\omega \epsb(\Ui_\zeta):\epsb(\varphi)\; dx=\int_\omega (\VC_2-\VC_1)\zeta:\epsb(\varphi)\; dx\quad\text{ for all }\varphi\in \dot{BL}(\VR^d)^d.\]
Hence, there are tensors $\mathcal{P}^1,\mathcal{P}^2$ representing $\mathcal{F}^1,\mathcal{F}^2$ respectively, which we refer to as polarisation tensors. With their help we are able to rewrite \eqref{eq:first_der_av} the following way:
\begin{align}
d\mathcal{J}(\emptyset,\omega)(x_0)=&(\VC_1-\VC_2)\epsb(u_0)(x_0):\mathcal{P}^1\epsb(q_0)(x_0)+(\VC_1-\VC_2)\epsb(u_0)(x_0):\mathcal{P}^2\epsb(u_0)(x_0)\\
&+\left[\gamma_f (f_1-f_2)\cdot u_0+(\VC_1-\VC_2)\epsb(u_0):\epsb(q_0)-(f_1-f_2)\cdot q_0\right](x_0).
\end{align}
\end{remark}
Next we compute the second order topological derivative. Therefore let $\lii(\eps):=\eps\li(\eps)$. By Proposition \ref{prop:main_av} item (ii) we have
\begin{equation}
d^2\mathcal{J}(\emptyset,\omega)(x_0)=\mRii(u_0,q_0)+\piimL(0,u_0,q_0),
\end{equation}
where
\begin{align*}
    \mRii(u_0,q_0)& =\lim_{\eps\searrow 0}\frac{\Cl(\eps,u_0,q_\eps)-\Cl(\eps,u_0,q_0)-\li(\eps)\mRi(u_0,q_0)}{\lii(\eps)}, \\
    \piimL(0,u_0,q_0) & =\lim_{\eps\searrow 0}\frac{\Cl(\eps,u_0,q_0)-\Cl(0,u_0,q_0)-\li(\eps)\piimL(0,u_0,q_0)}{\lii(\eps)},
\end{align*}
if the above limits exist. We start computing $\mRii(u_0,q_0)$. Using \eqref{eq:r1_av} we get 
\begin{align}\label{eq:r2_av}
\begin{split}
\mRii(u_0,q_0)=&\lim_{\eps\searrow 0}\bigg[\dashint_{\omega}(\VC_1-\VC_2)\epsb(u_0)\circ\Phi_\eps:\epsb(\eps^{-1}\Qi_\eps)\; dx -\dashint_{\omega}(f_1-f_2)\circ\Phi_\eps\cdot(\Qi_\eps)\;dx\\
&-\dashint_{\omega}(\VC_1-\VC_2)\epsb(u_0)(x_0):\epsb(\eps^{-1}\Qi)\; dx\bigg]\\
=&\lim_{\eps\searrow 0}\bigg[\dashint_{\omega}(\VC_1-\VC_2)[\eps^{-1}(\epsb(u_0)\circ\Phi_\eps-\epsb(u_0)(x_0))]:\epsb(\Qi_\eps)\; dx\\
&\qquad +\dashint_{\omega}(\VC_1-\VC_2)\epsb(u_0)(x_0):\epsb(\eps^{-1}\Qi_\eps-\eps^{-1}\Qi)\; dx\\
&\qquad +\dashint_{\omega}(f_1-f_2)\circ\Phi_\eps\cdot(\Qi_\eps)\;dx\bigg]\\
=&\dashint_{\omega}(\VC_1-\VC_2)[\nabla\epsb(u_0)(x_0)x]:\epsb(\Qi)\; dx+\dashint_{\omega}(\VC_1-\VC_2)\epsb(u_0)(x_0):\epsb(\Qii)\; dx\\
&-\dashint_{\omega}(f_1(x_0)-f_2(x_0))\cdot(\Qi)\;dx,
\end{split}
\end{align}
where we used $\Qi_\eps\rightarrow \Qi$ in $L_1(\omega)^d$, $\epsb(\Qi_\eps)\rightarrow \epsb(\Qi)$ in $L_2(\omega)^d$ and by Theorem~\ref{thm:main_Q2}\newline $\eps^{-1}(\epsb(\Qi_\eps)-\epsb(\Qi))\rightarrow \epsb(\Qii)$ in $L_2(\omega)^d$. Next we compute $\piimL(0,u_0,q_0)$:

\begin{align}\label{eq:l2_av}
\begin{split}
\piimL(0,u_0,q_0)=&\lim_{\eps\searrow 0}\bigg[\gamma_f\dashint_{\omega}\eps^{-1}[(f_1-f_2)\circ \Phi_\eps \cdot u_0\circ \Phi_\eps - (f_1(x_0)-f_2(x_0)) \cdot u_0(x_0)]\;dx\\
&+\dashint_{\omega}\eps^{-1}[(\VC_1-\VC_2)\epsb(u_0)\circ \Phi_\eps:\epsb(q_0)\circ \Phi_\eps-(\VC_1-\VC_2)\epsb(u_0)(x_0):\epsb(q_0)(x_0)]\; dx\\
&-\dashint_{\omega}\eps^{-1}[(f_1-f_2)\circ \Phi_\eps \cdot q_0\circ \Phi_\eps-(f_1(x_0)-f_2(x_0)) \cdot q_0(x_0)]\;dx\bigg]\\
    =&\gamma_f\dashint_{\omega}\nabla[(f_1-f_2)](x_0) x\cdot u_0(x_0)\; dx+\gamma_f\dashint_{\omega}[f_1(x_0)-f_2(x_0)]\cdot\nabla u_0(x_0)x\; dx\\
&+\dashint_{\omega}(\VC_1-\VC_2)[\nabla(\epsb(u_0))(x_0)x]:\epsb(q_0)(x_0)\; dx \\
& +\dashint_{\omega}(\VC_1-\VC_2)\epsb(u_0)(x_0):[\nabla(\epsb(q_0))(x_0)x]\; dx\\
&-\dashint_{\omega}[\nabla(f_1-f_2)(x_0)] x\cdot q_0(x_0)\; dx-\dashint_{\omega}[f_1(x_0)-f_2(x_0)]\cdot\nabla q_0(x_0)x\; dx,
\end{split}
\end{align}
where again we used the smoothness of $u_0,q_0,f_1,f_2$ in a neighbourhood of $x_0$ in the last step. This is the claimed formula \eqref{eq:l1_av}. Furthermore, combining \eqref{eq:r2_av} and \eqref{eq:l2_av}, we obtain the final formula for the second order topological derivative:
\begin{align}
\begin{split}
    d^2\mathcal{J}(\emptyset,\omega)(x_0)=&\gamma_f\dashint_{\omega} [\nabla(f_1-f_2)(x_0)] x\cdot u_0(x_0)\; dx+\gamma_f\dashint_{\omega}[f_1(x_0)-f_2(x_0)]\cdot\nabla u_0(x_0)x\; dx\\
&+\dashint_{\omega}(\VC_1-\VC_2)[\nabla(\epsb(u_0))(x_0)x]:\epsb(q_0)(x_0)\; dx \\
& +\dashint_{\omega}(\VC_1-\VC_2)\epsb(u_0)(x_0):[\nabla(\epsb(q_0))(x_0)x]\; dx\\
&-\dashint_{\omega}[\nabla(f_1-f_2)](x_0) x\cdot q_0(x_0)\; dx-\dashint_{\omega}[f_1(x_0)-f_2(x_0)]\cdot\nabla q_0(x_0)x\; dx\\
&+\dashint_{\omega}(\VC_1-\VC_2)[\nabla\epsb(u_0)(x_0)x]:\epsb(\Qi)\; dx+\dashint_{\omega}(\VC_1-\VC_2)\epsb(u_0)(x_0):\epsb(\Qii)\; dx\\
&-\dashint_{\omega}(f_1(x_0)-f_2(x_0))\cdot(\Qi)\;dx,
\end{split}
\end{align}
 with $\Qi$ defined in \eqref{eq:first_Qh},\eqref{eq:first_Qt} and $\Qii$ defined in \eqref{eq:Qiih},\eqref{eq:Qiit}.


 \subsection{Delfour's method}\label{sec:deriv_df}
At last we consider Delfour's method to compute the topological derivative. Therefore recall that by Proposition \ref{prop:main_df} item (i) we have

\begin{align}
d\mathcal{J}(\emptyset,\omega)(x_0)=\mRi_1(u_0, p_0) + \mRi_2(u_0, p_0) + \pimL(0, u_0, p_0),
\end{align}
where we let $\li(\eps):=|\omega_\eps|$ and assume that the limits

\begin{align}
        \mRi_1(u_0, p_0) :=& \underset{\eps \searrow 0}{\mbox{lim}} \; \frac{1}{\li(\eps)}  \int_0^1 \left(\partial_u \Cl(\eps,s u_\eps + (1-s)u_0, p_0) -   \partial_u \Cl(\eps, u_0, p_0)\right)(u_\eps - u_0) \; ds,\\
        \mRi_2(u_0, p_0) :=& \underset{\eps \searrow 0}{\mbox{lim}}\; \frac{1}{\li(\eps)}(\partial_u \Cl(\eps, u_0,p_0) - \partial_u \Cl(0,u_0,p_0))(u_\eps - u_0),\\
        \pimL(0, u_0, p_0) :=&  \lim_{\eps\searrow 0} \;\frac{1}{\li(\eps)}(\Cl(\eps,u_0, p_0) - \Cl(0,u_0, p_0)),
\end{align}
exist. We now compute the limit of each term. Plugging in the definition of $\Cl(\eps,u,v)$, we get for $\eps >0$, 
\begin{align}\label{eq:r11_df}
\begin{split}
\frac{1}{\li(\eps)}  \int_0^1 &\left(\partial_u \Cl(\eps,s u_\eps + (1-s)u_0, p_0) -   \partial_u \Cl(\eps, u_0, p_0)\right)(u_\eps - u_0) \; ds\\
&=\frac{1}{|\omega_\eps|}\int_0^1 \partial_u J(\eps,s u_\eps + (1-s)u_0, p_0)(u_\eps-u_0) -   \partial_u J(\eps, u_0, p_0)(u_\eps - u_0) \; ds\\
&=\frac{\eps}{|\omega|}\gamma_m\|\Ui_\eps\|_{L_2(\Gamma_\eps^m)^d}^2+\frac{1}{|\omega|}\gamma_g\|\nabla \Ui_\eps\|_{L_2(\Dsf_\eps)^{d\times d}}^2.
\end{split}
\end{align}
Hence, passing to the limit $\eps\searrow 0$ yields $\mRi_1(u_0,p_0)=\frac{\gamma_g}{|\omega|}\|\nabla\Ui\|_{L_2(\VR^d)^{d\times d}}^2$ (see \eqref{eq:r1_am}). Furthermore, we have
\begin{align}\label{eq:r12_df}
\begin{split}
\frac{1}{\li(\eps)}(\partial_u \Cl(\eps, u_0,p_0) - \partial_u \Cl(0,u_0,p_0))(u_\eps - u_0)=&\frac{1}{|\omega_\eps|}\gamma_f\int_{\omega_\eps}(f_1-f_2)\cdot(u_\eps-u_0)\; dx\\
&+\frac{1}{|\omega_\eps|}\int_{\omega_\eps}(\VC_1-\VC_2)\epsb(u_\eps-u_0):\epsb(p_0)\; dx\\
=&\eps\gamma_f\dashint_{\omega}(f_1\circ \Phi_\eps-f_2\circ \Phi_\eps)\cdot(\Ui_\eps)\; dx\\
&+\dashint_{\omega}(\VC_1-\VC_2)\epsb(\Ui_\eps):\epsb(p_0)\circ \Phi_\eps\; dx.
\end{split}
\end{align}
Since $\Ui_\eps\rightarrow\Ui$ in $L_1(\omega)^d$ and $\epsb(\Ui_\eps)\rightarrow \epsb(\Ui)$ in $L_2(\omega)^{d\times d}$ for $\eps\searrow 0$, which can be seen similarly to the analogous results for $\PPi$ and $\Qi$, it follows
\begin{align}
\mRi_2(u_0,p_0)=\dashint_{\omega}(\VC_1-\VC_2)\epsb(\Ui):\epsb(p_0)(x_0)\; dx.
\end{align}
A similar computation yields
\begin{align}\label{eq:l1_df}
\begin{split}
    \pimL(0, u_0, p_0) =&\left[\gamma_f (f_1-f_2)\cdot u_0+(\VC_1-\VC_2)\epsb(u_0):\epsb(p_0)-(f_1-f_2)\cdot p_0\right](x_0).
\end{split}
\end{align}
A more detailed derivation of $\pimL(0,u_0,p_0)$ can be found in \eqref{eq:l1_av} by substituting $p_0$ for $q_0$. Combining these limits, yields
\begin{align}
\begin{split}
d\mathcal{J}(\emptyset,\omega)(x_0)=&\dashint_{\omega}(\VC_1-\VC_2)\epsb(\Ui):\epsb(p_0)(x_0)\; dx+(\VC_1-\VC_2)\epsb(u_0)(x_0):\epsb(p_0)(x_0)\\
&+\gamma_f (f_1(x_0)-f_2(x_0))\cdot u_0(x_0)-(f_1(x_0)-f_2(x_0))\cdot p_0(x_0)\\
&+\gamma_g\frac{1}{|\omega|}\int_{\VR^d}|\nabla\Ui|^2\;dx,
\end{split}
\end{align}
with $\Ui$ defined in \eqref{eq:first_U}.
Next we compute the second order topological derivative. In view of Proposition~\ref{prop:main_df}, item (ii), we first show that the following limits exist:
\begin{align}
    \mRii_1(u_0,p_0) &:= \underset{\eps \searrow 0}{\mbox{lim }} \frac{1}{\lii(\eps)}\int_0^1 ( \partial_u \Cl(\eps,su_\eps + (1-s)u_0,p_0)\nonumber  \\
                     & \hspace{4cm}- \partial_u\Cl(\eps, u_0,p_0)) (u_\eps - u_0)\; ds - \li(\eps) \mRi_1(u_0, p_0), \\
    \mRii_2(u_0,p_0) &:=\underset{\eps \searrow 0}{\mbox{lim }} \frac{1}{\lii(\eps)}\left[(\partial_u \Cl(\eps,u_0,p_0) - \partial_u \Cl(0,u_0,p_0))(u_\eps-u_0) - \li(\eps) \mRi_2(u_0, p_0) \right], \\
    \piimL(0,u_0, p_0) &:=\underset{\eps \searrow 0}{\mbox{lim }} \frac{1}{\lii(\eps)} \left[ \Cl(\eps,u_0,p_0) - \Cl(0,u_0,p_0) - \li(\eps) \pimL(0, u_0, p_0) \right],
\end{align}
where $\lii(\eps):=\eps\li(\eps)$. Then the topological derivative is given by \[d^2\mathcal{J}(\Omega)(x_0)=\mRii_1(u_0,p_0)+\mRii_2(u_0,p_0)+\piimL(0,u_0, p_0).\]
Similar to \eqref{eq:r2_am} it follows that \[\mRii_1(u_0,p_0)=2\gamma_g\frac{1}{|\omega|}\int_{\VR^d}\nabla\Ui:\nabla\Uii\;dx.\] Furthermore, we have

\begin{align}\label{eq:r22_df}
\begin{split}
\frac{1}{\lii(\eps)}(\partial_u \Cl(\eps, u_0,p_0) - \partial_u \Cl(0,u_0,p_0))&(u_\eps - u_0) - \li(\eps) \mRi_2(u_0, p_0)\\
=&\gamma_f\dashint_{\omega}(f_1\circ \Phi_\eps-f_2\circ \Phi_\eps)\cdot(\Ui_\eps)\; dx\\
&+\dashint_{\omega}(\VC_1-\VC_2)\epsb(\eps^{-1}[\Ui_\eps-\Ui]):\epsb(p_0)\circ \Phi_\eps\; dx\\
&+\dashint_{\omega}(\VC_1-\VC_2)\epsb(\Ui):[\eps^{-1}(\epsb(p_0)\circ \Phi_\eps-\epsb(p_0)(x_0))]\; dx.
\end{split}
\end{align}
Hence, passing to the limit $\eps\searrow 0$, we deduce
\begin{align}
\begin{split}
\mRii_2(u_0, p_0) =&\dashint_{\omega}(\VC_1-\VC_2)\epsb(\Uii):\epsb(p_0)(x_0)\; dx+\dashint_{\omega}(\VC_1-\VC_2)\epsb(\Ui):[\nabla\epsb(p_0)(x_0)x]\; dx\\
&+\gamma_f\dashint_{\omega}(f_1(x_0)-f_2(x_0))\cdot\Ui\; dx,
\end{split}
\end{align}
where we used $\eps^{-1}[\epsb(\Ui_\eps)-\epsb(\Ui)]\rightarrow\epsb(\Uii)$ in $L_2(\omega)^{d\times d}$ and $\Ui_\eps\rightarrow \Ui$ in $L_1(\omega)^d$. Additionally, from \eqref{eq:l2_av} we get
\begin{align}
\begin{split}
\piimL(0,u_0,p_0)=&\gamma_f\dashint_{\omega}\nabla[f_1(x_0)-f_2(x_0)]x\cdot u_0(x_0)\; dx+\gamma_f\dashint_{\omega}[f_1(x_0)-f_2(x_0)]\cdot\nabla u_0(x_0)x\; dx\\
&+\dashint_{\omega}(\VC_1-\VC_2)[\nabla(\epsb(u_0))(x_0)x]:\epsb(p_0)(x_0)\; dx \\
& +\dashint_{\omega}(\VC_1-\VC_2)\epsb(u_0)(x_0):[\nabla(\epsb(p_0))(x_0)x]\; dx\\
&-\dashint_{\omega}\nabla[f_1(x_0)-f_2(x_0)]x\cdot p_0(x_0)\; dx-\dashint_{\omega}[f_1(x_0)-f_2(x_0)]\cdot\nabla p_0(x_0)x\; dx.
\end{split}
\end{align}
We obtain
\begin{align}
\begin{split}
d^2\mathcal{J}(\emptyset,\omega)(x_0)=&\gamma_f\dashint_{\omega}(f_1(x_0)-f_2(x_0))\cdot\Ui\; dx+2\gamma_g\frac{1}{|\omega|}\int_{\VR^d}\nabla\Ui:\nabla\Uii\;dx\\
&+\dashint_{\omega}(\VC_1-\VC_2)\epsb(\Uii):\epsb(p_0)(x_0)\; dx+\dashint_{\omega}(\VC_1-\VC_2)\epsb(\Ui):[\nabla\epsb(p_0)(x_0)x]\; dx\\
&+\gamma_f\dashint_{\omega}\nabla[f_1(x_0)-f_2(x_0)]x\cdot u_0(x_0)\; dx+\gamma_f\dashint_{\omega}[f_1(x_0)-f_2(x_0)]\cdot\nabla u_0(x_0)x\; dx\\
&+\dashint_{\omega}(\VC_1-\VC_2)[\nabla(\epsb(u_0))(x_0)x]:\epsb(p_0)(x_0)\; dx \\
& +\dashint_{\omega}(\VC_1-\VC_2)\epsb(u_0)(x_0):[\nabla(\epsb(p_0))(x_0)x]\; dx\\
&- \dashint_{\omega}\nabla[f_1(x_0)-f_2(x_0)]x\cdot p_0(x_0)\; dx-\dashint_{\omega}[f_1(x_0)-f_2(x_0)]\cdot\nabla p_0(x_0)x\; dx,
\end{split}
\end{align}
with $\Ui$ defined in \eqref{eq:first_U} and $\Uii$ defined in \eqref{eq:Uiih},\eqref{eq:Uiit}. This finishes the proof of the computation of the topological derivative using Delfour's method.

\begin{remark}
We like to point out that using the defining equations of the boundary layer correctors, one can show that all three expressions of the topological derivative coincide and therefore all methods lead to the same result. To get an idea, we show that the first topological derivative of Amstutz' approach and the averaged adjoint method are the same. Plugging in $\varphi=\Qi$ in \eqref{eq:first_U} yields
\[\int_\omega(\VC_1-\VC_2)\epsb(u_0)(x_0):\epsb(\Qi)\; dx=-\int_{\VR^d}\VC_\omega \epsb(\Ui):\epsb(\Qi)\; dx.\]
Additionally, by choosing $\varphi=\PPi$ in\eqref{eq:first_P} and $\varphi=\Ui$ in \eqref{eq:first_Qh},\eqref{eq:first_Qt} we get
\begin{align*}
\int_\omega(\VC_1-\VC_2)&\epsb(u_0)(x_0):\epsb(\PPi)\; dx+\int_{\VR^d}|\nabla \Ui|^2\; dx=\\
&\underbrace{-\int_{\VR^d}\VC_\omega\epsb(\Ui):\epsb(\PPi)\; dx+\int_\omega (\VC_2-\VC_1)\epsb(q_0)(x_0):\epsb(\Ui)\; dx}_{=0}\\
&-\int_{\VR^d}\VC_\omega \epsb(\Ui):\epsb(\Qi)\; dx.
\end{align*}
Now using $p_0=q_0$ it follows that both results \eqref{eq:first_der_am},\eqref{eq:first_der_av} coincide.
\end{remark}


\section{Conclusion}

In the present work we review three different methods to compute the second order topological derivative and illustrate their methodologies by applying them to a linear elasticity model. To give a better insight into the differences of these methods, the cost functional consists of three terms: the compliance, a $L_2$ tracking type over a part of the Neumann boundary and a gradient tracking type over the whole domain, whereas the first one is linear and the latter two are quadratic.\newline
Amstutz' method to compute the topological derivative requires beside the analysis of the direct state $u_\eps$ also the analysis of the adjoint state $p_\eps$. Even though this seems to lead to additional work, we like to point out that, due to the $\eps$-dependence of the defining equation of the adjoint state variable, the analysis of $p_\eps$ resembles the analysis of the direct state and can be done in a similar way. The computation of the topological derivative for the compliance term is straight forward, whereas checking the occurring limits for the nonlinearities requires the asymptotic analysis of $u_\eps$ on the whole domain.\newline
The averaged adjoint method shifts the work from the computation of the topological derivative to the asymptotic analysis of the averaged adjoint variable $q_\eps$. Since the defining equation depends on the state variable $u_\eps$, the asymptotic analysis of $p_\eps$ does not resemble the analysis of $u_\eps$ and therefore needs to be treated differently. In fact, we like to mention that again the nonlinearities of the cost functional are the reason for additional work during this process.
When it comes to the computation of the topological derivative, the averaged adjoint method simplifies the procedure as it only requires convergence of $q_\eps$ on a small subdomain of size $\eps$.\newline
Finally, Delfour's method resembles Amstutz' method as it requires the asymptotic analysis of $u_\eps$ on the whole domain, yet it does not need the analysis of the adjoint state $p_\eps$. This advantage seems to come with the shortcoming, that this method is only applicable to a selective set of cost functions.\newline
To recapitulate, each method proposed in this work has some advantages and disadvantages over the others. The decision, which method fits the actual problem setting the best, greatly depends on the actual cost function as well as the properties of the underlying partial differential equation.

\subsection*{Acknowledgements}
Phillip Baumann has been funded by the Austrian Science Fund (FWF) project P 32911.


\section{Appendix}

Here we derive equations satisfied by the variations of the adjoint and average adjoint variable respectively.

\subsection{Derivation and estimation of equation \eqref{eq:P1_rhs}}\label{sec:ad_P1}
In order to compute $G_\eps^1$, we subtract \eqref{eq:adj_per} and \eqref{eq:adj_unper} to obtain
\begin{align}
\begin{split}
\int_\Dsf \VC_{\omega_\eps}\epsb(\varphi):\epsb(p_\eps-p_0)\; dx=&\int_{\omega_\eps}(\VC_2-\VC_1)\epsb(\varphi):\epsb(p_0)\; dx\\
&+\gamma_f\int_{\omega_\eps}(f_2-f_1)\cdot\varphi\; dx,
\end{split}
\end{align}
for all $\varphi\in H^1_\Gamma(\Dsf)^d$.
Next we change variables according to the transformation $y=\Phi_\eps(x)$, multiply with $\eps^{1-d}$ and subtract
\begin{align}
\begin{split}
\int_{\Dsf_\eps}\VC_\omega \epsb(\varphi):\epsb(\PPi)\;dx=&\int_\omega(\VC_2-\VC_1)\epsb(\varphi):\epsb(p_0)(x_0)\;dx+\int_{\Gamma_\eps^N} \VC_2^\top\epsb(\PPi)n\cdot\varphi\;dS
\end{split}
\end{align}
to conclude
\begin{align}\label{eq:diff_P1}
\int_{\Dsf_\eps} \VC_{\omega}\epsb(\varphi):\epsb(\PPi_\eps-\PPi)\; dx=G^1_\eps(\varphi),
\end{align}
for all $\varphi\in H^1_{\Gamma_\eps}(\Dsf_\eps)^d$ with the notation

\begin{align}
\begin{split}
G^1_\eps(\varphi)=&\int_{\omega}(\VC_2-\VC_1)\epsb(\varphi):[\epsb(p_0)\circ \Phi_\eps-\epsb(p_0)(x_0)]\; dx\\
&+\eps\gamma_f\int_{\omega}(f_2-f_1)\circ \Phi_\eps\cdot\varphi\; dx\\
&-\int_{\Gamma_\eps^N} \VC_2^\top\epsb(\PPi)n\cdot\varphi\;dS.
\end{split}
\end{align}

Now we can find a constant $C>0$, such that the following estimates hold:
\begin{itemize}

\item $\left|\int_{\omega}(\VC_2-\VC_1)\epsb(\varphi):[\epsb(p_0)\circ \Phi_\eps-\epsb(p_0)(x_0)]\; dx\right|\le C \eps \|\varphi\|_\eps$, which follows from a Taylor's expansion of $\epsb(p_0)\circ \Phi_\eps$ in $x_0$ and Hölder's inequality.

\item $\left|\eps\gamma_f\int_{\omega}(f_2-f_1)\circ \Phi_\eps\cdot\varphi\; dx\right|\le C \eps\|\varphi\|_\eps$, for $d=3$ and $\left|\eps\gamma_f\int_{\omega}(f_2-f_1)\circ \Phi_\eps\cdot\varphi\; dx\right|\le C \eps\|\varphi\|_\eps^{1-\alpha}$, for $d=2$ which is a consequence of Hölder's inequality and Lemma \ref{lma:scaling2} item (b) and (c).

\item $\left|\int_{\Gamma_\eps^N} \VC_2^\top\epsb(\PPi)n\cdot\varphi\;dS\right|\le C\eps^{\frac{d}{2}}\|\varphi\|_\eps$, which follows from Hölder's inequality, Lemma \ref{lma:scaling2} item (d) and Lemma~\ref{lma:remainder_est}, item (iii) with $m=d-1$.

\end{itemize}
Combining the previous estimates yields
\begin{equation}
\|G^1_\eps\|\le\begin{cases}C \eps\quad &\text{ for }d=3,\\C\eps^{1-\alpha}\quad &\text{ for }d=2.\end{cases}
\end{equation}


\subsection{Derivation and estimation of equation \eqref{eq:P2_rhs}}\label{sec:ad_P2}
We start by dividing \eqref{eq:diff_P1} by $\eps$ and subtract \eqref{eq:wi}, \eqref{eq:wii}, which can be formulated on the domain $\Dsf_\eps$ by a change of variables. Next we subtract \eqref{eq:Piih} \eqref{eq:Piit}, whereas these equations can be restricted to the domain $\Dsf_\eps$ by splitting the integral over $\VR^d$ and integrating by parts in the exterior domain. These operations leave us with
\begin{align}
\int_{\Dsf_\eps} \VC_{\omega}\epsb(\varphi):\epsb(\eps\Piii_\eps)\; dx=G^2_\eps(\varphi)+G^3_\eps(\varphi),
\end{align}
for all $\varphi\in H^1_{\Gamma_\eps}(\Dsf_\eps)^d$ where
\begin{align}
\begin{split}
G^2_\eps(\varphi)=&\int_{\omega}(\VC_2-\VC_1)\epsb(\varphi):[\eps^{-1}(\epsb(p_0)\circ \Phi_\eps-\epsb(p_0)(x_0))-\nabla\eps(p_0)(x_0)x]\; dx\\
&+\gamma_f\int_{\omega}[(f_2-f_1)\circ \Phi_\eps-(f_2(x_0)-f_1(x_0))]\cdot\varphi\; dx\\
&+\eps^{d-1}\int_\omega (\VC_2-\VC_1)\epsb(\varphi):[\epsb(\wi)\circ \Phi_\eps+\epsb(\wii)\circ \Phi_\eps]\;dx,
\end{split}
\end{align}
\begin{align}
\begin{split}
G^3_\eps(\varphi)=&-\eps^{-1}\int_{\Gamma_\eps^N} [\VC_2^\top\epsb(\PPi)-\eps^d\VC_2^\top\epsb(\Si)(\eps x)]n\cdot\varphi\;dS\\
&-\int_{\Gamma_\eps^N} [\VC_2^\top\epsb(\Pii)-\eps^{d-1}\VC_2^\top\epsb(\Sii)(\eps x)]n\cdot\varphi\;dS.
\end{split}
\end{align}
Now we want to estimate the norm of $G_\eps^k$, $k\in\{2,3\}$. Therefore let $\varphi\in H^1_{\Gamma_\eps}(\Dsf_\eps)^d$.
\begin{itemize}

\item Since $p_0$ is three times differentiable in a neighbourhood of $x_0$, there is a constant $C>0$, such that $|\eps^{-1}(\epsb(p_0)\circ \Phi_\eps-\epsb(p_0)(x_0))-\nabla \epsb(p_0)(x_0)x|\le C\eps$, for $x\in\omega$. Hence, Hölder's inequality yields
\begin{equation}
\left|\int_{\omega}(\VC_2-\VC_1)\epsb(\varphi):[\eps^{-1}(\epsb(p_0)\circ \Phi_\eps-\epsb(p_0)(x_0))-\nabla\epsb(p_0)(x_0)x]\; dx\right|\le C\eps \|\varphi\|_\eps.
\end{equation}

\item A Taylor expansion of $(f_2-f_1)\circ \Phi_\eps$ at $x_0$, Hölder's inequality and Lemma~\ref{lma:scaling2} item (b) yield
\begin{equation}
\left|\gamma_f\int_{\omega}[(f_2-f_1)\circ \Phi_\eps-(f_2(x_0)-f_1(x_0))]\cdot\varphi\; dx\right|\le C\eps \|\varphi\|_\eps,
\end{equation}
for a constant $C>0$.

\item Furthermore, by Hölder's inequality we get
\begin{equation}
\left|\eps^{d-1}\int_\omega (\VC_2-\VC_1)\epsb(\varphi):[\epsb(\wi)\circ \Phi_\eps+\epsb(\wii)\circ \Phi_\eps]\;dx\right|\le C\eps \|\varphi\|_\eps,
\end{equation}
for a constant $C>0$.
\end{itemize}
Combining the above results leaves us with $\|G^2_\eps\|\le C \eps$ for a constant $C>0$.
Next we consider the boundary integral terms:

\begin{itemize}

\item From Hölder's inequality, Lemma~\ref{lma:remainder_est} item (iii) with $m=d$ and the scaled trace inequality we get
\begin{equation}
\left|\eps^{-1}\int_{\Gamma_\eps^N} [\VC_2^\top\epsb(\PPi)-\eps^d\VC_2^\top\epsb(\Si)(\eps x)]n\cdot\varphi\;dS\right|\le C\eps^{\frac{d}{2}}\|\varphi\|_\eps,
\end{equation}
for a constant $C>0$.

\item Similarly, we deduce from Lemma~\ref{lma:remainder_est} item (iii) with $m=d-1$ that there is a constant $C>0$, such that
\begin{equation}
\left|\int_{\Gamma_\eps^N} [\VC_2^\top\epsb(\Pii)-\eps^{d-1}\VC_2^\top\epsb(\Sii)(\eps x)]n\cdot\varphi\;dS\right|\le C\eps^{\frac{d}{2}}\|\varphi\|_\eps.
\end{equation}

\end{itemize}
Thus, these estimates result in $\|G_\eps^3\|\le C\eps$ for a constant $C>0$.


\subsection{Derivation and estimation of equation \eqref{eq:Q1_rhs}}\label{sec:ad_Q1}

In order to compute a governing equation for $\Qi_\eps-\Qi$ we start by subtracting \eqref{eq:av_adj_per} and \eqref{eq:av_adj_unper}. Rearranging these terms leaves us with

\begin{align}
\begin{split}
    \int_\Dsf \VC_{\omega_\eps}\epsb(\varphi):\epsb(q_\eps-q_0)\; dx= & \int_{\omega_\eps}(\VC_2-\VC_1)\epsb(\varphi):\epsb(q_0)\; dx\\
&+\gamma_f\int_{\omega_\eps}(f_2-f_1)\cdot\varphi\; dx\\
&-\gamma_m\int_{\Gamma^m}(u_\eps-u_0)\cdot\varphi\;dS\\
&-\gamma_g\int_{\Dsf}[\nabla u_\eps-\nabla u_0]:\nabla\varphi\;dx,
\end{split}
\end{align}
for all $\varphi\in H^1_\Gamma(\Dsf)^d$. Thus, considering the definition of $\Ui_\eps$, a change of variables followed by subtracting 

\begin{align}
\begin{split}
    \int_{\Dsf_\eps}\VC_\omega \epsb(\varphi):\epsb(\Qi)\;dx=& \int_\omega(\VC_2-\VC_1)\epsb(\varphi):\epsb(q_0)(x_0)\;dx-\gamma_g\int_{\Dsf_\eps}\nabla\Ui:\nabla\varphi\; dx\\
&+\int_{\Gamma^N_\eps}\VC_2^\top \epsb(\Qi)n\cdot\varphi\; dS+\int_{\Gamma^N_\eps}\nabla \Ui n\cdot\varphi\; dS\\
&+\int_{\VR^d\setminus\Dsf_\eps}\underbrace{[\Div(\VC_2^\top \epsb(\Qi))+\gamma_g\Delta \Ui]}_{=0}\cdot\varphi\; dx,
\end{split}
\end{align}

yields

\begin{equation}\label{eq:lhs_G5}
\int_{\Dsf_\eps}\VC_\omega\epsb(\varphi):\epsb(\eps\Qii_\eps)\;dx=G_\eps^4(\varphi),
\end{equation}
where
\begin{align}\label{eq:G5}
\begin{split}
    G_\eps^4(\varphi)=& \int_{\omega}(\VC_2-\VC_1)\epsb(\varphi):[\epsb(q_0)\circ \Phi_\eps-\epsb(q_0)(x_0)]\; dx\\
&+\eps\gamma_f\int_{\omega}(f_2\circ \Phi_\eps-f_1\circ \Phi_\eps)\cdot\varphi\; dx\\
&-\eps\gamma_m\int_{\Gamma^m_\eps}(\Ui_\eps)\cdot\varphi\;dS\\
&-\int_{\Gamma^N_\eps}\VC_2^\top \epsb(\Qi)n\cdot\varphi\; dS\\
&-\gamma_g\int_{\Gamma^N_\eps}\nabla \Ui n\cdot\varphi\; dS\\
&-\gamma_g\int_{\Dsf_\eps}(\nabla\Ui_\eps-\nabla\Ui):\nabla\varphi\;dx,
\end{split}
\end{align}
for all $\varphi\in H^1_{\Gamma_\eps}(\Dsf_\eps)^d$.

Now let $\varphi\in H^1_{\Gamma_\eps}(\Dsf_\eps)^d$. There is a constant $C>0$ independent of $\eps$ and $\varphi$, such that

\begin{itemize}

\item $|\int_{\omega}(\VC_2-\VC_1)\epsb(\varphi):[\epsb(q_0)\circ \Phi_\eps-\epsb(q_0)(x_0)]\; dx|\le C \eps \|\varphi\|_\eps$, which can be seen by a Taylor's expansion of $q_0$ in $x_0$ and Hölder's inequality.

\item $|\eps\gamma_f\int_{\omega}(f_2\circ \Phi_\eps-f_1\circ \Phi_\eps)\cdot\varphi\; dx|\le C \eps\|\varphi\|_\eps$, which is a consequence of Hölder's inequality and Lemma \ref{lma:scaling2} item (b).

\item $|\int_{\Gamma^N_\eps}\VC_2^\top \epsb(\Qi)n\cdot\varphi\; dS|\le C\eps^{\frac{1}{2}}\|\varphi\|_\eps$, which is a consequence of Hölder's inequality, Lemma \ref{lma:remainder_est} item (iii) with $m=d-2$ and the scaled trace inequality.

\item $|\gamma_g\int_{\Gamma^N_\eps}\nabla \Ui n\cdot\varphi\; dS|\le C\eps^{\frac{d}{2}}\|\varphi\|_\eps$, which can be seen similarly.

\item $|\eps\gamma_m\int_{\Gamma^m_\eps}(\Ui_\eps)\cdot\varphi\;dS|\le C \eps\|\varphi\|_\eps$, which follows from Hölder's inequality, splitting $\|\Ui_\eps\|_{L_2(\Gamma^m_\eps)^d}\le\|\Ui_\eps-\Ui\|_{L_2(\Gamma^m_\eps)^d}+\|\Ui\|_{L_2(\Gamma^m_\eps)^d}$, the scaled trace inequality, Theorem \ref{thm:main_U1} and Lemma \ref{lma:remainder_est} item (i) with $m=d-1$.

\item $|\gamma_g\int_{\Dsf_\eps}(\nabla\Ui_\eps-\nabla\Ui):\nabla\varphi\;dx|\le C\eps \|\varphi\|_\eps$, which is a consequence of Hölder's inequality and Theorem \ref{thm:main_U1}.

\end{itemize}
Combining the above results leaves us with \[\|G_\eps^4\|\le C \eps^{\frac{1}{2}}.\]


\subsection{Derivation and estimation of equation \eqref{eq:Q2_rhs}}\label{sec:ad_Q2}

Due to the high number of terms, we derive the governing equation in more detail this time. Therefore, we formulate \eqref{eq:mi}, \eqref{eq:miiU}, \eqref{eq:miiQ}, \eqref{eq:Qiih} and \eqref{eq:Qiit} on the domain $\Dsf_\eps$ by scaling arguments and splitting of the integral domain respectively, to get
\begin{align}\label{eq:mi_fit}
\begin{split}
    \int_{\Dsf_\eps} \VC_\omega \epsb(\varphi):\epsb(\eps^{d-3}\mi\circ \Phi_\eps)\;dx=& -\eps^{d-2}\int_{\Gamma^N_\eps}\VC_2^\top\epsb(\Ti)(\eps x)n\cdot\varphi\;dS\\
&+\eps^{d-2}\int_\omega (\VC_1-\VC_2)\epsb(\varphi):\epsb(\mi)\circ\Phi_\eps\;dx.
\end{split}
\end{align}
\begin{align}\label{eq:Qiih_fit}
\begin{split}
    \int_{\Dsf_\eps}\VC_\omega\epsb(\varphi):\epsb(\hat{Q}^{(2)})\; dx=& \int_\omega (\VC_2-\VC_1)\epsb(\varphi):[\nabla\epsb(q_0)(x_0)x]\; dx\\
&-\gamma_g\int_{\Dsf_\eps}\nabla\Uii:\nabla\varphi\; dS\\
&+\gamma_g\int_{\Gamma^N_\eps}\nabla\Uii n\cdot\varphi\; dS\\
&+\int_{\Gamma^N_\eps}\VC_2^\top \epsb(\hat{Q}^{(2)})n\cdot\varphi\; dS\\
\end{split}
\end{align}
\begin{align}\label{eq:Qiit_fit}
\begin{split}
    \int_{\Dsf_\eps}\VC_\omega\epsb(\varphi):\epsb(\tilde{Q}^{(2)})\; dx=& \gamma_f\int_\omega [f_2(x_0)-f_1(x_0)]\cdot\varphi\; dx\\
&+\int_{\Gamma^N_\eps}\VC_2^\top \epsb(\tilde{Q}^{(2)})n\cdot\varphi\; dS\\
\end{split}
\end{align}
\begin{align}\label{eq:mii_fit}
\begin{split}
    \int_{\Dsf_\eps} \VC_\omega   \epsb(\varphi):&\epsb(\eps^{d-2}\mii\circ \Phi_\eps)\;dx=  -\eps^{d-1}\gamma_m\int_{\Gamma^m_\eps}\Ri(\eps x)\cdot\varphi\;dS-\eps\gamma_m\int_{\Gamma^m_\eps}(\eps^{d-2}\vi\circ \Phi_\eps)\cdot\varphi\;dS\\
&-\eps^{d-1}\gamma_m\int_{\Gamma^m_\eps}\Rii(\eps x)\cdot\varphi\;dS-\eps\gamma_m\int_{\Gamma^m_\eps} (\eps^{d-2}\vii\circ \Phi_\eps)\cdot\varphi\;dS\\
&-\eps^{d-2}\gamma_g\int_{\Dsf_\eps}\nabla(\vi\circ\Phi_\eps):\nabla\varphi\; dx-\eps^{d-2}\gamma_g\int_{\Dsf_\eps}\nabla(\vii\circ\Phi_\eps):\nabla\varphi\; dx\\
&-\eps^{d-1}\gamma_g\int_{\Gamma^N_\eps}\nabla(\Ri)(\eps x)n \cdot\varphi\; dS-\eps^{d-1}\gamma_g\int_{\Gamma^N_\eps}\nabla(\Rii)(\eps x)n\cdot \varphi\; dS\\
&+\eps^{d-1}\int_\omega (\VC_1-\VC_2)\epsb(\varphi):\epsb(\mii)\circ\Phi_\eps\;dx.
\end{split}
\end{align}
Now dividing \eqref{eq:lhs_G5},\eqref{eq:G5} by $\eps$ and subtracting \eqref{eq:mi_fit}-\eqref{eq:mii_fit} leaves us with

\begin{equation}
\int_{\Dsf_\eps} \VC_\omega \epsb(\varphi):\epsb(V_\eps)\;dx=G_\eps^5(\varphi)+G_\eps^6(\varphi),
\end{equation}
where we remember the simplified notation $V_\eps:=\eps^{-1}[\Qi_\eps-\Qi]-\eps^{d-3}\mi\circ \Phi_\eps-\Qii-\eps^{d-2}\mii \circ \Phi_\eps$ and
\begin{align}\label{eq:G6}
\begin{split}
    G_\eps^5(\varphi)=&\int_{\omega}(\VC_2-\VC_1)\epsb(\varphi):[\eps^{-1}(\epsb(q_0)\circ \Phi_\eps-\epsb(q_0)(x_0))-\nabla\eps(q_0)(x_0)x]\; dx\\
&+\gamma_f\int_{\omega}[(f_2\circ \Phi_\eps-f_1\circ \Phi_\eps)-(f_2(x_0)-f_1(x_0))]\cdot\varphi\; dx\\
&-\gamma_g\int_{\Dsf_\eps}[\eps^{-1}(\nabla\Ui_\eps-\nabla\Ui)-\nabla(\eps^{d-2}\vi\circ\Phi_\eps)-\nabla\Uii-\nabla(\eps^{d-2}\vii\circ\Phi_\eps)]:\nabla\varphi\;dx\\
&+\eps^{d-2}\int_\omega(\VC_1-\VC_2) \epsb(\varphi):\epsb(\mi)\circ \Phi_\eps\; dx\\
&+\eps^{d-1}\int_\omega(\VC_1-\VC_2) \epsb(\varphi):\epsb(\mii)\circ \Phi_\eps\; dx,
\end{split}
\end{align}
\begin{align}\label{eq:G8}
\begin{split}
    G_\eps^6(\varphi)=&-\eps\gamma_m\int_{\Gamma^m_\eps}(\eps^{-1}\Ui_\eps-\eps^{d-2}\Ri(\eps x)-\eps^{d-2}\vi\circ T_\eps-\eps^{d-2}\Rii(\eps x)-\eps^{d-2}\vii\circ T_\eps)\cdot\varphi\;dS\\
&-\eps^{-1}\int_{\Gamma^N_\eps}\VC_2^\top [\epsb(\Qi)-\eps^{d-1}\epsb(\Ti)(\eps x)]n\cdot\varphi\; dS\\
&-\int_{\Gamma^N_\eps}\VC_2^\top \epsb(\Qii)n\cdot\varphi\; dS\\
&-\gamma_g\eps^{-1}\int_{\Gamma^N_\eps}[\nabla \Ui-\eps^d\nabla(\Ri)(\eps x)] n\cdot\varphi\; dS\\
&-\gamma_g\int_{\Gamma^N_\eps}[\nabla \Uii-\eps^{d-1}\nabla(\Rii)(\eps x)] n\cdot\varphi\; dS.
\end{split}
\end{align}
In the following let $\varphi \in H^1_{\Gamma_\eps}(\Dsf_\eps)^d$ and $C$ denote a sufficiently large constant independent of $\varphi$ and $\eps$. We now want to estimate the operator norm of $G_\eps^k$, $k\in\{5,6\}$ with respect to $\|\cdot\|_\eps$:

\begin{itemize}
\item A Taylor's expansion and Hölder's inequality yield \[\left|\int_{\omega}(\VC_2-\VC_1)\epsb(\varphi):[\eps^{-1}(\epsb(q_0)\circ \Phi_\eps-\epsb(q_0)(x_0))-\nabla\eps(q_0)(x_0)x]\; dx\right| \le C \eps \|\varphi\|_\eps.\]
\item A Taylor's expansion followed by an application of Hölder's inequality with respect to $p=2^\ast$ and Lemma \ref{lma:scaling2} item (b) yield 
\[\left|\gamma_f\int_{\omega}[(f_2\circ \Phi_\eps-f_1\circ \Phi_\eps)-(f_2(x_0)-f_1(x_0))]\cdot\varphi\; dx\right|\le C \eps \|\varphi\|_\eps.\]
\item From Theorem \ref{thm:U2_main} we deduce \[
        \begin{split}
            \bigg|\gamma_g\int_{\Dsf_\eps} & [\eps^{-1}(\nabla\Ui_\eps-\nabla\Ui)-\nabla(\eps^{d-2}\vi\circ\Phi_\eps) \\
     & \hspace{4cm} -\nabla\Uii-\nabla(\eps^{d-2}\vii\circ\Phi_\eps)]:\nabla\varphi\;dx\bigg|\le C\eps\|\varphi\|_\eps.
        \end{split}
        \]
\item Furthermore, from Hölder's inequality it follows
\[\left|\eps^{d-2}\int_\omega(\VC_1-\VC_2) \epsb(\varphi):\epsb(\mi)\circ \Phi_\eps\; dx\right|\le C \eps \|\varphi\|_\eps.\]
\item Similarly, one gets
\[\left|\eps^{d-1}\int_\omega(\VC_1-\VC_2) \epsb(\varphi):\epsb(\mii)\circ \Phi_\eps\; dx\right|\le C \eps \|\varphi\|_\eps.\]
\end{itemize}
Combining these estimates, we get $\|G_\eps^5\|\le C\eps$. Next we consider the boundary integral terms.
\begin{itemize}
\item By smuggling in $\eps^{-1}\Ui$ and $\Uii$ we get
\begin{align*}
&\left|\eps\gamma_m\int_{\Gamma^m_\eps}(\eps^{-1}\Ui_\eps-\eps^{d-2}\Ri(\eps x)-\eps^{d-2}\vi\circ \Phi_\eps-\eps^{d-2}\Rii(\eps x)-\eps^{d-2}\vii\circ \Phi_\eps)\cdot\varphi\;dS\right|\le\\
&\left|\eps\gamma_m\int_{\Gamma^m_\eps}(\eps^{-1}\Ui_\eps-\eps^{-1}\Ui-\eps^{d-2}\vi\circ \Phi_\eps-\Uii-\eps^{d-2}\vii\circ \Phi_\eps)\cdot\varphi\;dS\right|\\
&+\left|\eps\gamma_m\int_{\Gamma^m_\eps}(\eps^{-1}\Ui-\eps^{d-2}\Ri(\eps x))\cdot\varphi\;dS\right|\\
&+\left|\eps\gamma_m\int_{\Gamma^m_\eps}(\Uii-\eps^{d-2}\Rii(\eps x))\cdot\varphi\;dS\right|.
\end{align*}
The first term on the right hand side can be estimated by Hölder's inequality, the scaled trace inequality and Theorem \ref{thm:U2_main}, whereas the remaining terms can be estimated by Hölder's inequality, the scaled trace inequality and Lemma \ref{lma:remainder_est} item (i). Thus we conclude
\[
    \begin{split}
    \bigg|\eps\gamma_m\int_{\Gamma^m_\eps}(& \eps^{-1}\Ui_\eps-\eps^{d-2}\Ri(\eps x)  -\eps^{d-2}\vi\circ \Phi_\eps \\
                                           & \hspace{2cm} -\eps^{d-2}\Rii(\eps x)-\eps^{d-2}\vii\circ \Phi_\eps)\cdot\varphi\;dS\bigg|\le C\eps\|\varphi\|_\eps.
\end{split}
\]

\item A similar computation to Lemma \ref{lma:remainder_est} and the scaled trace inequality yield
\[\left|\eps^{-1}\int_{\Gamma^N_\eps}\VC_2^\top [\epsb(\Qi)-\eps^{d-1}\epsb(\Ti)(\eps x)]n\cdot\varphi\; dS\right|\le C\eps^{\frac{1}{2}}\ln(\eps^{-1})\|\varphi\|_\eps.\]

\item A similar argument shows
\[\left|\int_{\Gamma^N_\eps}\VC_2^\top \epsb(\Qii)n\cdot\varphi\; dS\right|\le C\eps^{\frac{1}{2}}\ln(\eps^{-1})\|\varphi\|_\eps.\]

\item Furthermore, the remaining terms can be estimated by Hölder's inequality, the scaled trace inequality and Lemma \ref{lma:remainder_est} item (iii) with $m=d-1$ and $m=d$ respectively, to deduce  
\[\left|\gamma_g\eps^{-1}\int_{\Gamma^N_\eps}[\nabla \Ui-\eps^d\nabla(\Ri)(\eps x)] n\cdot\varphi\; dS\right|\le C\eps^{\frac{d}{2}}\|\varphi\|_\eps,\]
\[\left|\gamma_g\int_{\Gamma^N_\eps}[\nabla \Uii-\eps^{d-1}\nabla(\Rii)(\eps x)] n\cdot\varphi\; dS\right|\le C\eps^{\frac{d}{2}}\|\varphi\|_\eps.\]

\end{itemize}
 Hence, we conclude $\|G_\eps^6\|\le C\eps^{\frac{1}{2}}\ln(\eps^{-1})$.

\subsection*{Acknowledgements}
Phillip Baumann has been funded by the Austrian Science Fund (FWF) project P 32911.

\bibliographystyle{plain}
\bibliography{refs}

\end{document}